\documentclass[smallcondensed]{svjour3}
\smartqed

\usepackage{latexsym}
	\usepackage{mathtools}
	\usepackage{amsmath}
	\usepackage{amssymb}
	\usepackage{stmaryrd}
	\usepackage{enumerate}
	\usepackage{color}
	\usepackage{graphicx}
	\usepackage{algorithm}
	\usepackage{url}
	\usepackage[colorlinks=true]{hyperref}
	\usepackage[numbers,sort&compress,square,comma]{natbib}
	\usepackage{thmtools}
	\usepackage{tikz}
	\usepackage{booktabs}
\usepackage{multirow}
\usepackage[a4paper, margin=1in]{geometry}

\makeatletter
\let\cl@chapter\undefined
\makeatletter
\usepackage[capitalise,noabbrev]{cleveref}

\usepackage{letltxmacro}
\LetLtxMacro\orgvdots\vdots
\LetLtxMacro\orgddots\ddots

\makeatletter
\DeclareRobustCommand\vdots{\mathpalette\@vdots{}}
\newcommand*{\@vdots}[2]{\sbox0{$#1\cdotp\cdotp\cdotp\m@th$}\sbox2{$#1.\m@th$}\vbox{\dimen@=\wd0 \advance\dimen@ -3\ht2 \kern.5\dimen@
\dimen@=\wd2 \advance\dimen@ -\ht2 \dimen2=\wd0 \advance\dimen2 -\dimen@
    \vbox to \dimen2{\offinterlineskip
      \copy2 \vfill\copy2 \vfill\copy2 }}}
\DeclareRobustCommand\ddots{\mathinner{\mathpalette\@ddots{}\mkern\thinmuskip
  }}
\newcommand*{\@ddots}[2]{\sbox0{$#1\cdotp\cdotp\cdotp\m@th$}\sbox2{$#1.\m@th$}\vbox{\dimen@=\wd0 \advance\dimen@ -3\ht2 \kern.5\dimen@
\dimen@=\wd2 \advance\dimen@ -\ht2 \dimen2=\wd0 \advance\dimen2 -\dimen@
    \vbox to \dimen2{\offinterlineskip
      \hbox{$#1\mathpunct{.}\m@th$}\vfill
      \hbox{$#1\mathpunct{\kern\wd2}\mathpunct{.}\m@th$}\vfill
      \hbox{$#1\mathpunct{\kern\wd2}\mathpunct{\kern\wd2}\mathpunct{.}\m@th$}}}}

\DeclareRobustCommand\bddots{\mathinner{\mathpalette\@bddots{}\mkern\thinmuskip
  }}
\newcommand*{\@bddots}[2]{\sbox0{$#1\cdotp\cdotp\cdotp\m@th$}\sbox2{$#1.\m@th$}\vbox{\dimen@=\wd0 \advance\dimen@ -3\ht2 \kern.5\dimen@
\dimen@=\wd2 \advance\dimen@ -\ht2 \dimen2=\wd0 \advance\dimen2 -\dimen@
    \vbox to \dimen2{\offinterlineskip
      \hbox{$#1\mathpunct{\kern\wd2}\mathpunct{\kern\wd2}\mathpunct{.}\m@th$}\vfill
      \hbox{$#1\mathpunct{\kern\wd2}\mathpunct{.}\m@th$}\vfill
      \hbox{$#1\mathpunct{.}\m@th$}}}}
\makeatother 
\def\cbl{\color{blue}}

\definecolor{gold}{rgb}{0.85,0.65,0}

\newcommand{\floor}[1]{\ensuremath{\left\lfloor #1 \right\rfloor}}

\newcommand{\abs}[1]{\ensuremath{\left\lvert #1 \right\rvert}}

\newcommand{\by}{\times}

\newcommand{\norm}[1]{\ensuremath{\left\lVert #1 \right\rVert}}
\newcommand{\ip}[1]{\ensuremath{\left\langle #1 \right\rangle}}

\let\emptyset\varnothing
\newcommand{\set}[1]{\left\{#1\right\}}

\def\C{{\mathbb{C}}}

\def\H{{\mathbb{H}}}
\def\I{{\mathbb{I}}}

\def\N{{\mathbb{N}}}

\def\R{{\mathbb{R}}}
\def\S{{\mathbb{S}}}
\def\T{{\mathbb{T}}}

\def\cA{{\cal A}}

\def\cI{{\cal I}}
\def\cJ{{\cal J}}

\def\cL{{\cal L}}

\def\cS{{\cal S}}

\DeclareMathOperator{\Opt}{Opt}

\DeclareMathOperator{\rank}{rank}
\DeclareMathOperator{\Diag}{Diag}
\DeclareMathOperator{\diag}{diag}

\DeclareMathOperator{\range}{range}

\newenvironment{smallpmatrix}
    {\left(
    \begin{smallmatrix}}
    {\end{smallmatrix}
    \right)
    }

\DeclareMathOperator{\spann}{span}

\DeclareMathOperator{\conv}{conv}

\renewcommand{\Re}{\operatorname{Re}}
\renewcommand{\Im}{\operatorname{Im}}
\newcommand{\imag}{\mathrm{i}}
\spnewtheorem{assumption}{Assumption}{\bf}{\it}
\spnewtheorem{observation}{Observation}{\bf}{\it}
\crefname{assumption}{Assumption}{Assumptions}
\crefname{observation}{Observation}{Observations}

\usepackage{amsmath}

\usepackage{amsthm}

\newcommand{\mathprog}[1]{#1}
\DeclareMathOperator{\supp}{supp}

\begin{document}

\title{New notions of simultaneous diagonalizability of quadratic forms with applications to QCQPs}
\titlerunning{New notions of simultaneous diagonalizability}

\author{Alex L.\ Wang \and
        Rujun Jiang
}

\institute{F. Author \at
              first address \\
              Tel.: +123-45-678910\\
              Fax: +123-45-678910\\
              \email{fauthor@example.com}           
           \and
           S. Author \at
              second address
}

\institute{Alex L. Wang \at
Carnegie Mellon University, Pittsburgh, PA, USA\\
Present Address: Daniels School of Business, Purdue University, West Lafayette, IN, USA
\and
Rujun Jiang (Corresponding author)\at
School of Data Science, Fudan University, Shanghai, China\\
Shanghai Key Laboratory for Contemporary Applied Mathematics, Shanghai, China
}
\date{\today}


\maketitle

\begin{abstract}


A set of quadratic forms is simultaneously diagonalizable via congruence (SDC) if there exists a basis under which each of the quadratic forms is diagonal.
This property appears naturally when analyzing quadratically constrained quadratic programs (QCQPs) and has important implications in globally solving such problems using branch-and-bound methods.
This paper extends the reach of the SDC property by studying two new weaker notions of simultaneous diagonalizability. Specifically, we say that a set of quadratic forms is almost SDC (ASDC) if it is the limit of SDC sets and $d$-restricted SDC ($d$-RSDC) if it is the restriction of an SDC set in up to $d$-many additional dimensions.
In the context of QCQPs, these properties correspond to problems that may be diagonalized after arbitrarily small perturbations or after the introduction of $d$ additional variables.
Our main contributions are complete characterizations of the ASDC pairs and nonsingular triples of symmetric matrices, as well as a sufficient condition for the $1$-RSDC property for pairs of symmetric matrices.
Surprisingly, we show that \emph{every} singular symmetric pair is ASDC and that \emph{almost every} symmetric pair is $1$-RSDC.
We accompany our theoretical results with preliminary numerical experiments applying these constructions to solve QCQPs within branch-and-bound schemes.

\end{abstract}


\section{Introduction}
This paper investigates two new notions of \textit{simultaneous diagonalizability} of quadratic forms and their applications in solving quadratically constrained quadratic programs (QCQPs).%

Let $\S^n$ denote the real vector space of $n\by n$ symmetric matrices.%
\footnote{While all of our results hold with only minor modifications over both $\C^n$ and Hermitian matrices and $\R^n$ and symmetric matrices, we will simplify our presentation in the main body by discussing only the real setting; see \cref{sec:hermitian_proofs} for a discussion of our results in the complex setting.}
Recall that a set of matrices $\cA\subseteq\S^n$ is said to be \textit{simultaneously diagonalizable via congruence} (SDC) if there exists an invertible $P\in\R^{n\by n}$ such that $P^\intercal AP$ is diagonal for every $A\in\cA$.
This property has attracted significant interest in the optimization community in recent years in the context of solving subclasses of QCQPs and their relaxations~\cite{jiang2016simultaneous,wang2021tightness,nguyen2020simultaneous,le2020equivalent,zhou2019simultaneous,zhou2020simultaneous,luo2020effective}.
Specifically, the SDC property corresponds to the ability to rewrite a given QCQP as a \emph{diagonal QCQP} (see \cref{subsec:motivation} below).
The SDC property also finds applications in areas such as signal processing, multivariate statistics, medical imaging analysis, and genetics; see \cite{bustamante2020solving,vollgraf2006quadratic} and references therein.

In this paper, we take a step towards increasing the practical importance of the SDC property in the context of globally solving QCQPs by investigating two weaker notions of simultaneous diagonalizability.
These weaker notions
formalize methods for diagonalizing classes of \textit{a priori} non-diagonalizable QCQPs.

\subsection{Motivation}
\label{subsec:motivation}
A general QCQP can be written as
\begin{align}
\label{eq:intro_qcqp}
\Opt\coloneqq \inf_{x\in\R^n}\set{x^\intercal A_1x + 2b_1^\intercal x + c_1:\, \begin{array}
	{l}
	x^\intercal A_i x + 2b_i^\intercal x + c_i \,\boxempty_i\, 0,\,\forall i\in[2,m]\\
	x\in\cL
\end{array}},
\end{align}
where for every $i\in[m]$, we have $A_i \in\S^n$, $b_i\in\R^n$, $c_i\in\R$, and $\boxempty_i\in\set{\leq,=}$; and $\cL\subseteq\R^n$ is a polyhedron. In words, the objective is to minimize a quadratic function subject to quadratic (in)equality constraints and linear (in)equality constraints.
QCQPs are highly expressive and capture numerous hard problems of both applied and theoretical interest; see~\cite{bao2011semidefinite,shor1990dual,wang2021tightness} and references therein. In fact, this class of problems is NP-hard even if $\cL=[-1,1]^n$ and there are no quadratic constraints (e.g., via max-cut).

We will refer to a QCQP in which the set of symmetric matrices $\cA = \set{A_1,\dots,A_m}$ is SDC as a \textit{diagonalizable QCQP}.
By definition, a diagonalizable QCQP can be rewritten as a \emph{diagonal QCQP} (one in which $\cA$ is a set of \emph{diagonal} matrices) upon a linear change of variables. Indeed, letting $y = P^{-1}x$ and $D_i = P^\intercal A_i P$ gives
\begin{align*}
\inf_{y\in\R^n}\set{y^\intercal D_1 y + 2(P^\intercal b_1)^\intercal y + c_1:\, \begin{array}
	{l}
	y^\intercal D_i y + 2(P^\intercal b_i)^\intercal y + c_i \,\boxempty_i\, 0,\,\forall i\in[2,m]\\
	y\in P^{-1}\cL
\end{array}}.
\end{align*}
While diagonal QCQPs are still NP-hard in general, they benefit from a number of advantages over more general QCQPs:
\begin{itemize}
	\item It is well known that the standard Shor semidefinite program (SDP) relaxation of a diagonal QCQP is equivalent to a second-order cone program (SOCP) \cite{wang2021tightness}. Consequently, the SDP relaxation can be solved substantially faster for diagonal QCQPs than for general QCQPs. Similar ideas have be used to build cheap but strong convex relaxations within branch and bound (BB) frameworks for nonconvex QCQPs \cite{zhou2019simultaneous,zhou2020simultaneous}.

	As we will see in \cref{sec:application}, when $P$ is well-conditioned, the computational savings of replacing an SDP with an SOCP within every node of a BB tree can outweigh the computational costs of preprocessing a diagonalizable QCQP into a diagonal QCQP.

\item Additionally, qualitative properties of the standard SDP relaxation are often easier to analyze in the context of diagonal QCQPs.
For example, a long line of work has investigated when the SDP relaxations of certain diagonal QCQPs are \textit{exact} (for various definitions of exact) and have given sufficient conditions for these properties~\cite{ben1996hidden,ben2014hidden,jiang2016simultaneous,burer2019exact,ho2017second,locatelli2016exactness,hsia2013trust,jeyakumar2014trust,wang2020generalized,blekherman2024aggregations}.
{For instance, Blekherman et al. \cite{blekherman2024aggregations} established in Theorem 2.12 that the convex hull of a set defined by a particular class of diagonal quadratic inequalities can be described using finitely many convex quadratic inequalities.
In general, such arguments often rely} on conditions (such as convexity\footnote{The convexity of the quadratic image is sometimes referred to as ``hidden convexity.''} or polyhedrality) of the quadratic image~\cite{polyak1998convexity} or the set of convex Lagrange multipliers~\cite{wang2021tightness}. In this context, the SDC property ensures that both of these sets are polyhedral.
While such conditions have been generalized beyond only diagonal or diagonalizable QCQPs, the sufficient conditions often become much more difficult to verify~\cite{wang2021tightness,wang2020geometric}.

As we will see in \cref{sec:application}, the SDP relaxation of a diagonal QCQP with bound constraints (as are encountered within BB schemes) admits low-rank solutions. Heuristically, {this \emph{may} suggest that the corresponding SDP relaxations should be strong---despite the fact that there is no formal connection between the rank of the solution and the quality of the relaxation.}
\end{itemize}


\subsection{Main contributions and outline}
In this paper, we define and analyze the \textit{almost SDC} (ASDC) and \textit{$d$-restricted SDC} ($d$-RSDC) properties; see \cref{sec:preliminaries,sec:rsdc} for precise definitions. Informally, $\cA\subseteq\S^n$ is ASDC if it is the limit of SDC sets and $d$-RSDC if it is the restriction of an SDC set in $\S^{n+d}$ to $\S^n$.
In the context of QCQPs, if the set $\cA=\set{A_1,\dots,A_m}$ is ASDC, then the QCQP can be diagonalized after arbitrarily small perturbations to the $A_i$ matrices. In a similar vein, if $\cA$ is $d$-RSDC, then the QCQP can be diagonalized after the introduction of $d$ additional ``dummy'' variables.

A summary of our contributions, along with an outline of the paper, follows:
\begin{itemize}
	\item We conclude this section in \cref{subsec:related_work} by reviewing related work on BB methods for QCQPs, the SDC property, and the almost simultaneously diagonalizable via similarity property.
	\item In \cref{sec:preliminaries}, we formally define the SDC and ASDC properties and review known characterizations of the SDC property. We additionally highlight a number of behaviors of the SDC property which will later contrast with those of the ASDC property.
	\item In \cref{sec:asdc_pairs}, we give a complete characterization of the ASDC property for pairs of symmetric matrices. In particular, \cref{thm:asdc_singular} states that \textit{every} singular\footnote{See \cref{def:singular}.} pair $\set{A,B}\subseteq\S^n$ is ASDC.
	The proof of this statement relies on the canonical form for pairs of symmetric matrices~\cite{uhlig1976canonical} under congruence transformations and the invertibility of a certain matrix related to the eigenvalues of an ``arrowhead'' matrix.
	\item In \cref{sec:asdc_triples}, we give a complete characterization of the ASDC property for \emph{nonsingular} triples of symmetric matrices.
	Our proof and constructions rely on facts about block matrices with Toeplitz upper triangular blocks. We review the relevant properties of such matrices in \cref{sec:upper_tri_toeplitz}.
	\item In \cref{sec:rsdc}, we formally define the $d$-RSDC property and highlight its relation to the ASDC property.
    {We will see that any pair of symmetric matrices $\set{A,B}\subset\S^n$ is naively $n$-RSDC.
    We strengthen this observation by showing} in \cref{thm:rsdc_pair} that the $1$-RSDC property holds for almost every pair of symmetric matrices.
    We also give a construction for the $d$-RSDC property for $d\geq 1$ and almost every pair of symmetric matrices. This second construction makes use of additional degrees of freedom and empirically leads to improved performance in the context of globally solving QCQPs (see \cref{sec:application}).

	\item In \cref{sec:obstructions}, we construct obstructions to \textit{a priori} plausible generalizations of our developments in \cref{sec:asdc_pairs,sec:asdc_triples,sec:rsdc}. \cref{subsec:sing_triples_obstruction} shows that, in contrast to \cref{thm:asdc_singular}, there exist singular triples of symmetric matrices which are \emph{not} ASDC. The same construction can be interpreted as a triple of symmetric matrices which is not $d$-RSDC for any $d<\floor{n/2}$; this contrasts with \cref{thm:rsdc_pair}. Next, \cref{subsec:nonsing_obstruction} shows that a natural generalization of our characterizations of the ASDC property for pairs and triples of symmetric matrices cannot hold for general $m$-tuples; specifically this natural generalization fails for $m \geq 7$.
	\item In \cref{sec:application}, we revisit one of the key motivations for studying the ASDC and $d$-RSDC properties---solving QCQPs more efficiently. In this context, we begin by deriving a number of theoretical results that give heuristic reasons why one would expect SOCP-based BB methods for diagonal QCQPs to outperform SDP-based BB methods for more general QCQPs. We then present a number of preliminary numerical experiments that corroborate this intuition.
    {While our $d$-RSDC based reformulations significantly outperform the direct SDP-based BB approach, it is slightly worse than the simpler $n$-RSDC based reformulation (see \cref{sec:application} for details). Nonetheless, we believe that diagonalization-based methods can be a useful tool in the development of more efficient methods for QCQPs and that the theory developed in this paper is a first step in this direction.}
\end{itemize}

\begin{remark}\label{rem:real_proofs}
In the main body of this paper, we will state and prove our results for only the real symmetric setting. Nevertheless, our results and proofs extend almost verbatim to the Hermitian setting by replacing the canonical form of a pair of real symmetric matrices (\cref{prop:canonical_form}) by the canonical form for a pair of Hermitian matrices (see \cite[Theorem 6.1]{lancaster2005canonical}).
As no new ideas or insights are required for handling the Hermitian setting, we defer formally stating our results in the Hermitian setting and discussing the necessary modifications to our proofs to \cref{sec:hermitian_proofs}.\mathprog{\qed}
\end{remark}


\subsection{Related work}
\label{subsec:related_work}
\paragraph{Branch-and-bound methods for QCQPs}
Most existing works for globally solving QCQPs are based on spatial BB methods.
\citet{audet2000branch} developed an LP-based branch and cut method for QCQPs using the reformulation-linearization technique (RLT) \cite{sherali2013reformulation}.
\citet{linderoth2005simplicial} proposed a triangle-based BB algorithm for solving nonconvex QCQPs, where two-dimensional triangles and rectangles are used to partition the feasible region.
Recently, \citet{zhou2020simultaneous} proposed a BB algorithm for QCQPs with nonconvex objective functions and convex quadratic constraints, based on the SDC property between the objective function and a specific aggregation of the convex quadratic constraints under a positive definiteness assumption.
\citet{luo2020effective} propose a BB algorithm based on the SDC property of two positive semidefinite matrices for solving a nonconvex QCQP arising from optimal portfolio deleveraging problems.
Please refer to \cite{chen2012globally,billionnet2016exact,chen2017spatial,lu2017eigenvalue,eltved2022strengthened} for other recent developments in globally solving nonconvex QCQPs.



\paragraph{The SDC property for sets of quadratic forms and SDC algorithms.}
The SDC property for a pair of symmetric matrices (more generally, Hermitian matrices) is well-understood and
follows from results due to \citet{weierstrass1868zur} and Kronecker (see \cite{kronecker1968collected}). We review these results in \cref{sec:preliminaries} (see also \cref{prop:canonical_form}).
More recently, there has been much interest in the optimization literature towards understanding the SDC property for general $m$-tuples of quadratic forms~\cite{jiang2016simultaneous,nguyen2020simultaneous,le2020equivalent}.
In fact, the search for ``sensible and ``palpable'' conditions'' for this property appeared as an open question on a short list of 14 open questions in nonlinear analysis and optimization~\cite{hiriart2007potpourri}.
In the real symmetric setting, \citet{jiang2016simultaneous} gave a complete characterization of this property under a semidefiniteness assumption. This result was then improved upon by \citet{nguyen2020simultaneous} who removed the semidefiniteness assumption.
\citet{le2020equivalent} additionally extend these characterizations to the case of Hermitian matrices.
\citet{bustamante2020solving} gave a complete characterization of the simultaneous diagonalizability of an $m$-tuple of \emph{symmetric complex} matrices under ${}^\intercal$-congruence.\footnote{We emphasize that \citet{bustamante2020solving} consider complex symmetric matrices and adopt ${}^\intercal$-congruence as their notion of congruence.}

We remark that this line of work is ``algorithmic'' and gives numerical procedures for deciding if a given set of quadratic forms is SDC. See \cite{le2020equivalent} and references therein.

\paragraph{The almost SDS property.}
An analogous theory for the \emph{almost} simultaneous diagonalizability of \emph{linear operators} has been studied in the literature.
In this setting, the congruence transformation is naturally replaced by a similarity transformation\footnote{Recall that two matrices $A,B\in\C^{n\by n}$ are similar if there exists an invertible $P\in\C^{n\by n}$ such that $A = P^{-1}BP$.} and the SDC property is replaced by simultaneous diagonalizability \emph{via similarity} (SDS).
A widely cited theorem due to \citet{motzkin1955pairs} shows that every pair of commuting linear operators, i.e., a pair of matrices in $\C^{n\by n}$, is almost SDS. This line of investigation was more recently picked up by \citet{omeara2006approximately} who showed that triples of commuting linear operators are almost SDS under a regularity assumption on the dimensions of eigenspaces associated with the linear operators.


\subsection{Notation}
\label{sub:notation}
Let $\N = \set{1,2,\dots}$ and $\N_0 = \set{0,1,\dots}$.
For $m,n\in\N_0$, let $[m,n]= \set{m,m+1,\dots,n}$ and $[n]= \set{1,\dots,n}$. By convention, if $m\geq n + 1$ (respectively, $n\leq 0$), then $[m,n] = \emptyset$ (respectively, $[n] = \emptyset$).
Given $x\in\R^n$, let $\supp(x)\coloneqq\set{i\in[n]:\, x_i\neq 0}$ denote the support of $x$.
Let $\abs{I}$ be the cardinality of a  set $I$.
For $n\in\N$, let $\S^n$ denote the vector space of $n\by n$ symmetric matrices.
For $v\in\R^n$ and $A\in\R^{n\by n}$ let $v^\intercal$ and $A^\intercal$ denote the transpose of $v$ and $A$ respectively.
Let $\spann(\cdot)$ and $\dim(\cdot)$
denote the span and dimensions
respectively.
For $n,m\in\N$, let $I_n$, $0_n$, and $0_{n\by m}$ denote the $n\by n$ identity matrix, $n\by n$ zero matrix, and $n\by m$ zero matrix respectively.
When $n\in\N$ is clear from context, let $e_i\in\R^n$ denote the $i$th standard basis vector.
Given a complex subspace $V\subseteq \R^n$ with dimension $k$, a surjective map $U:\R^k\to V$, and $A\in\S^n$, let $A|_V\in\S^k$ denote the restriction of $A$ to $V$, i.e., $A|_V = U^\intercal A U$.
When the map $U$ is inconsequential, we will omit specifying $U$.
For $\alpha_1,\dots,\alpha_k\in\R$, let $\Diag(\alpha_1,\dots,\alpha_k)\in\R^{k\by k}$ denote the diagonal matrix with $i$th entry $\alpha_i$. For $A_1,\dots,A_k$ square matrices, let $\Diag(A_1,\dots,A_k)$ denote the block diagonal matrix with $i$th block $A_i$.
Given $A\in \R^{n\by n}$ and $B\in\R^{m\by m}$, let
$A\oplus B\in \R^{(n+m)\by (n+m)}$ and $A\otimes B\in\R^{nm\by nm}$ denote the direct sum and Kronecker product of $A$ and $B$ respectively.
Given $A,B\in\R^{n\by n}$, let $[A,B]\coloneqq AB - BA$ denote the commutator of $A$ and $B$.
For $A\in\R^{n\by n}$, let $\norm{A}$ denote the spectral norm of $A$.
Given $\alpha\in\C$, let $\Re(\alpha)$, $\Im(\alpha)$, and $\alpha^*$ denote the real and imaginary parts and complex conjugate of $\alpha$ respectively.
For $A\in\C^{n\by n}$, let $A^*$ denote the conjugate transpose of $A$.
We will denote the imaginary unit by the symbol $\imag$ in order to distinguish it from the variable $i$, which will often be used as an index.

\section{Preliminaries} 
\label{sec:preliminaries}

In this section, we define our main objects of study and recall some useful results from the literature.

\begin{definition}
A set $\cA\subseteq\S^n$ is \emph{simultaneously diagonalizable via congruence} (SDC) if there exists an invertible $P\in\R^{n\by n}$ such that
$P^\intercal AP$ is diagonal for all $A\in\cA$.\mathprog{\qed}
\end{definition}

\begin{remark}\label{rem:congruence}
The SDC property is the natural notion for simultaneous diagonalization in the context of quadratic forms. Indeed, suppose $\cA\subseteq\S^n$ is SDC and let $P$ denote the corresponding invertible matrix. Then, performing the change of variables $y = P^{-1}x$, we have that $x^\intercal A x = y^\intercal(P^\intercal AP)y$ is separable in $y$ for every $A\in\cA$.\mathprog{\qed}
\end{remark}

\begin{observation}
\label{obs:sdc_span_subset}
The SDC property is closed under taking spans and subsets. In particular, $\cA\subseteq\S^n$ is SDC if and only if $\set{A_1,\dots,A_m}$ is SDC for some basis $\set{A_1,\dots,A_m}$ of $\spann(\cA)$.
\end{observation}

We begin by studying the following relaxation of the SDC property.

\begin{definition}
A set $\cA \subseteq\S^n$ is \emph{almost simultaneously diagonalizable via congruence} (ASDC) if there exist sequences $A_{i}\to A$ for every $A\in\cA$ such that for every $i\in\N$, the set $\set{A_i:\, A\in\cA }$ is SDC.\mathprog{\qed}
\end{definition}

\begin{observation}
\label{obs:asdc_span_subset}
The ASDC property is closed under taking spans and subsets. In particular, $\cA\subseteq\S^n$ is ASDC if and only if $\set{A_1,\dots,A_m}$ is ASDC for some basis $\set{A_1,\dots,A_m}$ of $\spann(\cA)$.
\end{observation}

When $\abs{\cA}$ is finite, we will use the following equivalent definition of ASDC.
\begin{observation}
\label{obs:asdc_finite}
A finite set $\set{A_1,\dots,A_m}\subseteq\S^n$ is ASDC if and only if for all $\epsilon>0$, there exist $\tilde A_1,\dots,\tilde A_m\in\S^n$ such that
\begin{itemize}
	\item for all $i\in[m]$, the spectral norm $\norm{A_i - \tilde A_i}\leq \epsilon$, and
	\item $\set{\tilde A_1,\dots,\tilde A_m}$ is SDC.
\end{itemize}
\end{observation}

We will additionally need the following two definitions.

\begin{definition}
\label{def:singular}
A set $\cA\subseteq\S^n$ is \emph{nonsingular} if there exists a nonsingular $A\in\spann(\cA)$. Else, it is \emph{singular}.\mathprog{\qed}
\end{definition}

\begin{definition}
Given a set $\cA\subseteq\S^n$, we will say that $S\in\cA$ is a \textit{max-rank element} of $\spann(\cA)$ if $\rank(S) = \max_{A\in\cA}\rank(A)$.\mathprog{\qed}
\end{definition}
{
\begin{remark}
Our procedures for determining SDC and ASDC will assume we have access to a max-rank element $S\in\cA$. This element is easy to produce algorithmically. In fact, we claim that a \emph{generic} $S\in\cA$ will be a max-rank element of $\cA$. To see this, let $k = \max_{A\in\cA}\rank(A)$ and let $\hat S$ be any max-rank element of $\cA$. Then, there must be some principal minor of $\hat S$ indexed by $\cI$ of size $\abs{\cI} = k$ such that $\det(\hat S_\cI) \neq 0$. Thus, $\det(S_\cI)$ is a nonzero polynomial on $\cA$ so that for generic $S\in\cA$ we have that $\det(S_{\cI})\neq 0$ and $\rank(S)= k$. Thus, with probability one, $\sum_{i=1}^m \alpha_i A_i$ is a max-rank element of $\cA$ if we sample $\alpha\in[0,1]^m$ uniformly.
\end{remark}
}




\subsection{Characterization of SDC}
A number of necessary and/or sufficient conditions for the SDC property have been given in the literature \cite{horn2012matrix,bustamante2020solving,lancaster2005canonical}. For our purposes, we will need the following two results. The first result gives a characterization of the SDC property for nonsingular sets of symmetric matrices and is well-known (see \cite[Theorem 4.5.17]{horn2012matrix}). The second result, due to \citet{bustamante2020solving}, gives a characterization of the SDC property for singular sets of symmetric matrices by reducing to the nonsingular case.
For completeness, we provide a short proof for each of these results in \cref{sec:proof_sdc_characterization}.

\begin{restatable}{proposition}{sdccharacterizationinvertible}\label{prop:sdc_characterization_invertible}
Let $\cA\subseteq\S^n$ and suppose $S\in\spann(\cA)$ is nonsingular. Then, $\cA$ is SDC if and only if $S^{-1}\cA$ is a commuting set of diagonalizable matrices with real eigenvalues.
\end{restatable}

\begin{restatable}
	{proposition}{sdccharacterization}
	\label{prop:sdc_characterization}
	Let $\cA\subseteq\S^n$ and suppose $S\in\spann(\cA)$ is a max-rank element of $\spann(\cA)$. Then, $\cA$ is SDC if and only if $\range(A)\subseteq\range(S)$ for every $A\in\cA$ and $\set{A|_{\range(S)}:\, A\in\cA}$ is SDC.
\end{restatable}

We close this section with two lemmas highlighting consequences of the SDC property which we will compare and contrast with consequences of the ASDC property.

\begin{lemma}
\label{lem:sdc_characterization_pd}
Let $\cA\subseteq\S^n$ and suppose $S\in\spann(\cA)$ is positive definite. Then, $\cA$ is SDC if and only if $S^{-1/2}\cA S^{-1/2}$ is a commuting set.
\end{lemma}
\begin{proof}
This follows as an immediate corollary to \cref{prop:sdc_characterization_invertible} and the fact that $S^{-1}A$ has the same eigenvalues as the symmetric matrix $S^{-1/2}AS^{-1/2}$.
\end{proof}

In particular, when $\spann(\cA)$ contains a positive definite matrix, the SDC and ASDC properties can be shown to be equivalent.

\begin{corollary}
\label{cor:sdc_iff_asdc_pd}
Let $\cA\subseteq\S^n$ and suppose $S\in\spann(\cA)$ is positive definite. Then, $\cA$ is SDC if and only if $\cA$ is ASDC.
\end{corollary}


Despite \cref{cor:sdc_iff_asdc_pd}, we will see soon that the ASDC property is qualitatively quite different to the SDC property in a number of settings (in particular, for singular pairs of symmetric matrices; see \cref{thm:asdc_singular}). Specifically, we will contrast the following consequence of the SDC property.
\begin{lemma}[{\cite[Lemma 9]{le2020equivalent}}]
\label{lem:sdc_closed_under_padding}
Let $\cA\subseteq\S^n$ and suppose there exists a common block decomposition
\begin{align*}
A = \begin{pmatrix}
	\bar A &\\
	& 0_d
\end{pmatrix}
\end{align*}
for all $A\in\cA$.
Then $\cA$ is SDC if and only if $\set{\bar A:\, A\in\cA}\subseteq\S^{n-d}$ is SDC.
\end{lemma}




\section{The ASDC property of symmetric pairs}
\label{sec:asdc_pairs}
In this section, we will give a complete characterization of the ASDC property for pairs of symmetric matrices (henceforth, \emph{symmetric pairs}).
We will switch the notation above and label our matrices $\cA = \set{A,B}$.
Our analysis will proceed in two cases: when $\set{A,B}$ is nonsingular and singular respectively.

\subsection{A canonical form for symmetric pairs}
In this section and the next, we will make regular use of the canonical form for symmetric pairs \cite{lancaster2005canonical,uhlig1976canonical}.

We will need to define the following special matrices.
For $n\geq 2$, let $F_n,\,G_n,\,H_n\in\S^n$ denote the matrices of the form
\begin{align*}
F_n &= \begin{smallpmatrix}
&&1\\
&\bddots&\\
1&&
\end{smallpmatrix},\,\qquad
G_n = \begin{smallpmatrix}
&&&0\\
&&\bddots&1\\
&0 &\bddots&\\
0 &1&&
\end{smallpmatrix},\,\quad\text{and}\quad
H_n = \begin{smallpmatrix}
&&1&0\\
&\bddots&0&\\
1 &\bddots &&\\
0 &&&
\end{smallpmatrix}.
\end{align*}
Set $F_1 = (1)$ and $G_1  = H_1 = (0)$.

The following proposition is adapted\footnote{The original statement of \cite[Theorem 9.1]{lancaster2005canonical} contains one additional type of block: those corresponding to the eigenvalues at infinity. These blocks do not exist in our setting by the assumption that $A$ is a max-rank element of $\spann(\set{A,B})$.} from \cite[Theorem 9.1]{lancaster2005canonical}.
\begin{proposition}
\label{prop:canonical_form}
        Let $A,B\in\S^n$ and suppose $A$ is a max-rank element of $\spann(\set{A,B})$. Then, there exists an invertible $P\in\R^{n\by n}$ such that $P^\intercal AP=\Diag(S_1,\dots,S_m)$ and $P^\intercal BP=\Diag(T_1,\dots,T_m)$ are block diagonal matrices with compatible block structure. Here, $m = m_1+m_2+m_3+m_4$ corresponds to four different types of blocks where each $m_i\in\N_0$ may be zero. Additionally, $m_4\in\set{0,1}$.

        The first $m_1$-many blocks of $P^\intercal AP$ and $P^\intercal BP$ have the form
        \begin{align*}
        S_i = \sigma_i F_{n_i},\qquad T_i = \sigma_i(\lambda_i F_{n_i}+G_{n_i}),
        \end{align*}
        where $n_i\in\N$, $\sigma_i\in\set{\pm 1}$, and $\lambda_i\in\R$. The next $m_2$-many blocks of $P^\intercal AP$ and $P^\intercal BP$ have the form
        \begin{align}
        \label{eq:canonical_form_imaginary}
        S_i = \begin{smallpmatrix}
                & F_{n_i}\\
                F_{n_i} &
        \end{smallpmatrix},\qquad
        T_i = F_{n_i}\otimes \begin{smallpmatrix}
                \Im(\lambda_i)& \Re(\lambda_i)\\
                \Re(\lambda_i) & -\Im(\lambda_i)
        \end{smallpmatrix} + G_{n_i}\otimes F_2,
        \end{align}
        where $n_i\in\N$ and $\lambda_i \in\C\setminus\R$. The next $m_3$-many blocks of $P^\intercal AP$ and $P^\intercal BP$ have the form
        \begin{align*}
        S_i = \begin{smallpmatrix}
                && F_{n_i}\\
                &0&\\
                F_{n_i}&&
        \end{smallpmatrix},\qquad
        T_i = G_{2n_i +1},
        \end{align*}
        where $n_i\in\N$. If $m_4 =1$, then the last block of $P^\intercal AP$ and $P^\intercal BP$ has the form $S_m = T_m = 0_{n_m}$ for some $n_m\in\N$.
\end{proposition}

We will repeatedly encounter real matrices that represent complex numbers, e.g., the blocks $S_i^{-1}T_i$ for $i$ corresponding to $m_2$ in the canonical form. We recall some useful facts: Let $J\in\C^{2\by 2}$ be the unitary matrix
\begin{align*}
J \coloneqq \begin{pmatrix}
        \tfrac{\imag}{\sqrt{2}} & \tfrac{-\imag}{\sqrt{2}}\\
        \tfrac{1}{\sqrt{2}} & \tfrac{1}{\sqrt{2}}
\end{pmatrix}\in\C^{2\by 2}.
\end{align*}
Then, a matrix of the form $\begin{smallpmatrix}
        \Im(\lambda_i)& \Re(\lambda_i)\\
        \Re(\lambda_i) & -\Im(\lambda_i)
\end{smallpmatrix}$ has the same eigenvalues as
\begin{align*}
J^*\begin{smallpmatrix}
        \Re(\lambda_i) & -\Im(\lambda_i)\\
        \Im(\lambda_i)& \Re(\lambda_i)
\end{smallpmatrix}J = \begin{smallpmatrix}
        \lambda_i &\\&\lambda_i^*
\end{smallpmatrix}.
\end{align*}

\subsection{The nonsingular case}
In this section, we will show that if $A$ is invertible, then $\set{A,B}$ is ASDC if and only if $A^{-1} B$ has real eigenvalues. We begin by examining two examples that are representative of the situation when $A$ is invertible.
Note in this case, that $m_3 = m_4 = 0$ in the canonical form (\cref{prop:canonical_form}).

\begin{example}
Let $\lambda\in\R$ and consider
\begin{align*}
A = \begin{pmatrix}
        &1\\
        1&
\end{pmatrix} = F_2,\qquad
B = \begin{pmatrix}
        0 & \lambda\\
        \lambda & 1
\end{pmatrix} = \lambda F_2 + G_2.
\end{align*}
Noting that $A^{-1}B$ is not diagonalizable, we conclude via \cref{prop:sdc_characterization_invertible} that $\set{A,B}$ is not SDC. On the other hand, let $\epsilon>0$ and define
\begin{align*}
\tilde B = \begin{pmatrix}
        \epsilon & \lambda\\
        \lambda & 1
\end{pmatrix}.
\end{align*}
Now, $A^{-1}\tilde B$ has eigenvalues $\lambda \pm \sqrt{\epsilon}$, whence by \cref{prop:sdc_characterization_invertible} $\set{A,\tilde B}$ is SDC.\mathprog{\qed}
\end{example}

\begin{example}
Let $\lambda\in\C\setminus\R$ and consider
\begin{align*}
A= F_2 = \begin{pmatrix}
        & 1\\
        1 &
\end{pmatrix},\qquad
B=
\begin{pmatrix}
        \Im(\lambda) & \Re(\lambda)\\
        \Re(\lambda) & -\Im(\lambda)
\end{pmatrix}.
\end{align*}
Noting that $A^{-1}B$ has non-real eigenvalues, we conclude via \cref{prop:sdc_characterization_invertible} (and the fact that eigenvalues vary continuously) that $\set{A,B}$ is not ASDC. \mathprog{\qed}
\end{example}

The following technical lemma will be useful in proving the main result of this section and shows that it is possible to perturb $B$ to ensure that $A^{-1}B$ has simple eigenvalues while maintaining its number of real/complex eigenvalues.

\begin{lemma}
\label{lem:perturbed_canonical_form}
Let $\set{A,B}\subseteq\S^n$ and suppose $A$ is invertible. For all $\epsilon>0$, there exists $\tilde B$ such that
\begin{itemize}
        \item $\norm{B -\tilde B}\leq \epsilon$,
        \item $A^{-1}\tilde B$ has simple eigenvalues (whence $A^{-1}\tilde B$ is diagonalizable), and
        \item $A^{-1}\tilde B$ and $A^{-1}B$ have the same number of real eigenvalues counted with multiplicity.
\end{itemize}
\end{lemma}
\begin{proof}
Without loss of generality, we may assume that $A= \Diag(S_1,\dots,S_m)$ and $B=\Diag(T_1,\dots,T_m)$ are in canonical form (\cref{prop:canonical_form}). Note that as $A$ is invertible, we will have $m_3 = m_4=0$.
For notational convenience, let $r = m_1$ and let $\sigma_1,\dots,\sigma_r, n_1\dots, n_m, \lambda_1,\dots,\lambda_m$ denote the quantities furnished by \cref{prop:canonical_form}.
We will give a probabilistic construction (summarized in \cref{proc:simple_eigs}) for $\tilde B$ that satisfies all three conditions with probability one.

Let $\delta = \frac{\epsilon}{2}$ and pick a random $\eta$ uniformly from $[-\delta,\delta]^m$. Define the blocks $\tilde T_i$ as
\begin{align}
\label{eq:perturbed_canonical_form_T_construction}
\tilde T_i &\coloneqq T_i + \sigma_i\left(\eta_i F_{n_i} + \delta H_{n_i}\right),\,\forall i\in[r],\nonumber\\
\tilde T_i &\coloneqq T_i + \left(\eta_i F_{n_i} + \delta H_{n_i}\right) \otimes F_2,\,\forall i\in[r+1,m],
\end{align}
and set $\tilde B \coloneqq \Diag(\tilde T_1,\dots,\tilde T_m)$. Then, $A^{-1}\tilde B = \Diag(S_1^{-1}\tilde T_1,\dots,S_k^{-1}\tilde T_k)$ is again a block diagonal matrix. Note that for $i\in[r]$, the block
\begin{align*}
S_i^{-1}\tilde T_i = (\lambda_i + \eta_i)I_{n_i} + F_{n_i}G_{n_i} + \delta F_{n_i}H_{n_i}
\end{align*}
is a Toeplitz tridiagonal matrix. Next, for $i\in[r+1,m]$, the block $S_i^{-1}\tilde T_i$ has the form
\begin{align}
\label{eq:perturbed_canonical_form_ST_construction}
S_i^{-1}\tilde T_i =
I_{n_i}\otimes \begin{pmatrix}
        \Re(\lambda_i) & -\Im(\lambda_i)\\\Im(\lambda_i) & \Re(\lambda_i)
\end{pmatrix} + \left(\eta_i I_{n_i} + F_{n_i}G_{n_i} + \delta F_{n_i}H_{n_i}\right)\otimes I_2.
\end{align}
Note that $S_i^{-1}\tilde T_i$ has the same eigenvalues as
\begin{align*}
&\left(I_{n_i}\otimes J\right)^{-1}S_i^{-1}\tilde T_i
\left(I_{n_i}\otimes J\right)\\
&\qquad= I_{n_i}\otimes \begin{pmatrix}\lambda_i &\\& \lambda_i^*
\end{pmatrix} + \left(\eta_i I_{n_i} + F_{n_i}G_{n_i} + \delta F_{n_i}H_{n_i}\right)\otimes I_2.
\end{align*}
This is, up to a simultaneous permutation of rows and columns, a direct sum of two Toeplitz tridiagonal matrices.

Using the closed form expression for eigenvalues of Toeplitz tridiagonal matrices \cite{horn2012matrix}, we have that $A^{-1}\tilde B$ has eigenvalues
\begin{align*}
&\bigcup_{i=1}^r \set{\lambda_i + \eta_i + 2\sqrt{\delta}\cos\left(\frac{\pi j}{n_i + 1}\right):j\in[n_i]}\\
&\qquad\cup\bigcup_{i=r+1}^m\set{\lambda + \eta_i + 2\sqrt{\delta}\cos\left(\frac{\pi j}{n_i + 1}\right):j\in[n_i],\,\lambda\in\set{\lambda_i,\lambda_i^*}}.
\end{align*}
Note that the quantity $\eta_i + 2\sqrt{\delta}\cos\left(\frac{\pi j}{n_i + 1}\right)$ is real so that $A^{-1}B$ and $A^{-1}\tilde B$ have the same number of real eigenvalues counted with multiplicity.
Furthermore, as $\eta\in[-\delta,\delta]^m$ was picked uniformly at random, we have that $A^{-1}\tilde B$ has only simple eigenvalues with probability one.

Finally, $\norm{B - \tilde B} = \norm{\Diag(T_1 - \tilde T_1,\dots,T_m - \tilde T_m)} = \max_{i}\norm{T_i - \tilde T_i} \leq \epsilon$.
\end{proof}

\begin{algorithm}[t]
\caption{Construction for simple eigenvalues}
\label{proc:simple_eigs}
\textbf{Input}: $A',B'\in\S^n$ such that $A'$ is invertible and $\epsilon'>0$\\
{\textbf{Output}: $\tilde B'$ satisfies $\norm{\tilde B'-B'}\leq \epsilon'$ and $(A')^{-1}\tilde B'$ with simple eigenvalues with probability one}
\begin{enumerate}
        \item Compute the canonical form~\cite{lancaster2005canonical} for $\set{A',B'}$, i.e.,
        \begin{align*}
        P^\intercal A'P &= A = \Diag(S_1,\dots,S_m),\text{ and}\\
        P^\intercal B'P &= B = \Diag(T_1,\dots,T_m).
        \end{align*}
        \item Set $\epsilon = \epsilon'/\norm{P^{-1}}^2$ and $\delta = \frac{\epsilon}{2}$
        \item Pick $\eta$ uniformly at random from $[-\delta,\delta]^m$
        \item Return $\tilde B' \coloneqq P^{-\intercal}\Diag(\tilde T_1,\dots,\tilde T_m)P^{-1}$, where $\tilde T_i$ are defined in \eqref{eq:perturbed_canonical_form_T_construction}
\end{enumerate}
\end{algorithm}

The following theorem follows as a simple corollary to our developments thus far.
\begin{theorem}
\label{thm:asdc_characterization_invertible}
Let $A,B\in\S^n$ and suppose $A$ is invertible. Then, $\set{A,B}$ is ASDC if and only if $A^{-1}B$ has real eigenvalues.
\end{theorem}
\begin{proof}
$(\Rightarrow)$ This direction holds trivially by continuity of eigenvalues and the assumption that $A$ is invertible.

$(\Leftarrow)$ Let $\epsilon>0$. Then, applying \cref{lem:perturbed_canonical_form} to $\set{A,B}$, we get $\tilde B$ such that $\norm{B-\tilde B}\leq \epsilon$ and $A^{-1}\tilde B$ is a matrix with real simple eigenvalues. We deduce by \cref{prop:sdc_characterization_invertible} that $\set{A,\tilde B}$ is SDC.
\end{proof}

\begin{corollary}
Let $\cA = \set{A_1,\dots,A_m}$ in \eqref{eq:intro_qcqp} and suppose $\spann(\cA) = \spann\set{A,B}$ where $A$ is invertible.
Furthermore, suppose $A^{-1}B$ has real eigenvalues.
Then for any $\epsilon>0$, there exist $\norm{\tilde A_i - A_i}\leq \epsilon$ such that
\begin{align}
\label{eq:asdc_nonsingular_QCQP}
\inf_{x\in\R^n}\set{x^\intercal \tilde A_1 x+ 2b_1^\intercal x + c_1:\, \begin{array}
        {l}
        x^\intercal \tilde A_i x+ 2b_i^\intercal x + c_i \,\boxempty_i\, 0,\,\forall i\in[2,m]\\
        x\in\cL
\end{array}}
\end{align}
is a diagonalizable QCQP. The matrices $\tilde A_i$ and the invertible matrix $P$ diagonalizing \eqref{eq:asdc_nonsingular_QCQP} can be computed via \cref{proc:simple_eigs}.
\end{corollary}

\subsection{The singular case}
In the remainder of this section, we investigate the ASDC property when $\set{A,B}$ is singular. We will show, surprisingly, that \emph{every} singular symmetric pair is ASDC. We begin with an example and some intuition.




\begin{example}
In contrast to the SDC property (cf.\ \cref{lem:sdc_closed_under_padding}), the ASDC property of a pair $\set{A,B}$ in the singular case does not reduce to the ASDC property of $\set{\bar A,\bar B}$, where $\bar A$ and $\bar B$ are the restrictions of $A$ and $B$ to the joint range of $A$ and $B$. For example, let
\begin{align*}
A = \begin{smallpmatrix}
         & 1 & \\
        1 &  &\\
         && 0
\end{smallpmatrix},\qquad
B = \begin{smallpmatrix}
        1 & & \\
        & -1 & \\
         &  & 0
\end{smallpmatrix},
\end{align*}
and let $\bar A$ and $\bar B$ denote the respective $2\by 2$ leading principal submatrices.

By \cref{thm:asdc_characterization_invertible}, $\set{\bar A,\bar B}$ is not ASDC (and in particular not SDC). On the other hand, we claim that $\set{A,B}$ is ASDC: For $\epsilon>0$, consider the matrices
\begin{align*}
\tilde A = \begin{smallpmatrix}
        & 1 & \\
        1 &&\\
        &&\epsilon
\end{smallpmatrix},\qquad
\tilde B = \begin{smallpmatrix}
        1 && \sqrt{\epsilon}\\
         & -1 & \sqrt{\epsilon}\\
        \sqrt{\epsilon} & \sqrt{\epsilon} & 0
\end{smallpmatrix}.
\end{align*}
A straightforward computation shows that $\tilde A^{-1} \tilde B$ has simple eigenvalues $\set{-1,0,1}$ whence $\set{\tilde A,\tilde B}$ is SDC.

The fact that $\set{\bar A, \bar B}$ is not SDC is equivalent to the statement: there does \emph{not} exist a basis $\set{p_1,p_2}\in\R^2$ such that the quadratic forms $x^\intercal\bar Ax$ and $x^\intercal\bar B x$ can be expressed as
\begin{align*}
x^\intercal \bar A x &= \alpha_1 \left(p_1^\intercal  x\right)^2 + \alpha_2 \left(p_2^\intercal x\right)^2,\quad \text{and}\\
x^\intercal \bar B x &= \beta_1 \left(p_1^\intercal  x\right)^2 + \beta_2 \left(p_2^\intercal x\right)^2,
\end{align*}
for some $\alpha_i,\beta_i\in\R$. On the other hand, the fact that $\set{\tilde A,\tilde B}$ is SDC shows that there exists a spanning set $\set{p_1,p_2,p_3}\subseteq\R^2$ and $\alpha_i,\beta_i\in\R$ such that
\begin{align*}
x^\intercal  \bar A x &= \alpha_1 \left(p_1^\intercal  x\right)^2 + \alpha_2 \left(p_2^\intercal x\right)^2 + \alpha_3 \left(p_3^\intercal x\right)^2,\quad \text{and}\\
x^\intercal  \bar B x &= \beta_1 \left(p_1^\intercal  x\right)^2 + \beta_2 \left(p_2^\intercal x\right)^2 + \beta_3 \left(p_3^\intercal x\right)^2.
\end{align*}
Intuitively, the ASDC property asks whether a set of quadratic forms can be (almost) diagonalized using $n$ (the ambient dimension)-many linear forms whereas the SDC property may be forced to use a smaller number of linear forms.\mathprog{\qed}
\end{example}

\begin{theorem}
\label{thm:asdc_singular}
Let $\set{A,B}\subseteq\S^n$. If $\set{A,B}$ is singular, then it is ASDC.
\end{theorem}
\begin{proof}
We make simplifying assumptions:
Without loss of generality, we may assume that $A$ is a max-rank element of $\spann(\set{A,B})$ and $A = \Diag(S_1,\dots,S_m)$ and $B = \Diag(T_1,\dots,T_m)$ are in canonical form (\cref{prop:canonical_form}).
We may assume $m_1 = 0$ (else consider the submatrix of $A,B$ corresponding to the remaining blocks).
As $A$ is singular, we have $m_3 + m_4\geq 1$. In fact, we may assume $m_3 + m_4 = 1$ (else, perturb the submatrix of $A$ corresponding to the first $m - 1$ blocks so that $A$ is nonsingular on those blocks).
Similarly, if $m_4 = 1$, we may assume that $n_m = 1$.
Finally, assume $\Diag(S_1^{-1}T_1,\dots,S_{m-1}^{-1}T_{m-1})$ has simple eigenvalues (else apply \cref{proc:simple_eigs} to the first $m-1$ blocks).
For notational convenience, let $m_2 = k$.

After the above simplifying assumptions, there are three cases left to consider:
where $m\geq2$ and $m_4=1$,
where $m\geq 2$ and $m_3 = 1$, and
where $m = 1$.
In the first two cases, $A,B$ have the form
\begin{align}
\label{eq:asdc_singular_barA_barB}
A &= \left(\begin{array}
        {@{}c|c|c|c@{}}
        \begin{smallmatrix}
                &1\\1&
        \end{smallmatrix} &&& \\\hline
        & \ddots && \\\hline
        && \begin{smallmatrix}
                &1\\1&
        \end{smallmatrix}&\\\hline
        &&&S_m
\end{array}\right),\qquad
B = \left(\begin{array}
        {@{}c|c|c|c@{}}
        \begin{smallmatrix}
                \Im(\lambda_1)&\Re(\lambda_1)\\\Re(\lambda_1)&-\Im(\lambda_1)
        \end{smallmatrix} && &\\\hline
        & \ddots && \\\hline
        && \begin{smallmatrix}
                \Im(\lambda_k)&\Re(\lambda_k)\\\Re(\lambda_k)&-\Im(\lambda_k)
        \end{smallmatrix}&\\\hline
        &&&T_m
\end{array}\right)
\end{align}
where $\lambda_1,\lambda_1^*,\dots,\lambda_k,\lambda_k^*\in\C\setminus\R$ are distinct.

\paragraph{Case 1.}
In case 1, $S_m = T_m = 0_1$. Set
\begin{align}
\label{eq:perturbed_A_B_case_1}
\tilde A_\delta &= \left(\begin{array}
        {@{}c|c|c|c@{}}
        \begin{smallmatrix}
                &1\\1&
        \end{smallmatrix} &&& \\\hline
        & \ddots && \\\hline
        && \begin{smallmatrix}
                &1\\1&
        \end{smallmatrix}&\\\hline
        &&&\delta
\end{array}\right),\,
\tilde B_\delta = \left(\begin{array}
        {@{}c|c|c|c@{}}
        \begin{smallmatrix}
                \Im(\lambda_1)&\Re(\lambda_1)\\\Re(\lambda_1)&-\Im(\lambda_1)
        \end{smallmatrix} &&& \begin{smallmatrix}
                \sqrt{\delta}\Re(\alpha_1)\\
                \sqrt{\delta}\Im(\alpha_1)
        \end{smallmatrix}\\\hline
        & \ddots && \vdots \\\hline
        && \begin{smallmatrix}
                \Im(\lambda_k)&\Re(\lambda_k)\\\Re(\lambda_k)&-\Im(\lambda_k)
        \end{smallmatrix}&\begin{smallmatrix}
                \sqrt{\delta}\Re(\alpha_k)\\
                \sqrt{\delta}\Im(\alpha_k)
        \end{smallmatrix}\\\hline
        \begin{smallmatrix}
                \sqrt{\delta}\Re(\alpha_1) & \sqrt{\delta}\Im(\alpha_1)
        \end{smallmatrix}&\cdots&\begin{smallmatrix}
                \sqrt{\delta}\Re(\alpha_k) & \sqrt{\delta}\Im(\alpha_k)
        \end{smallmatrix}&\delta z
\end{array}\right)
\end{align}
for some $\alpha\in\C^k$, $z\in\R$, and $\delta>0$ to be chosen later.
The eigenvalues of $\tilde A_\delta^{-1}\tilde B_\delta$ are equal to the eigenvalues of
\begin{align*}
&\left(\begin{array}
        {@{}c|c|c|c@{}}
        J &&&\\\hline
        & \ddots && \\\hline
        && J&\\\hline
        &&&\tfrac{1}{\sqrt{\delta}}
\end{array}\right)^{-1}\tilde A^{-1}_\delta \tilde B_\delta \left(\begin{array}
        {@{}c|c|c|c@{}}
        J &&&\\\hline
        & \ddots && \\\hline
        && J&\\\hline
        &&& \tfrac{1}{\sqrt{\delta}}
\end{array}\right)\\
&\qquad= \left(\begin{array}
        {@{}c|c|c|c@{}}
        \begin{smallmatrix}
                \lambda_1 & \\
                 & \lambda_1^*
        \end{smallmatrix} &&& \begin{smallmatrix}
                \alpha_1^*/\sqrt{2}\\\alpha_1/\sqrt{2}
        \end{smallmatrix}\\\hline
        & \ddots && \vdots \\\hline
        && \begin{smallmatrix}
                \lambda_k & \\
                 & \lambda_k^*
        \end{smallmatrix}&\begin{smallmatrix}
                \alpha_k^*/\sqrt{2}\\\alpha_k/\sqrt{2}
        \end{smallmatrix}\\\hline
        \begin{smallmatrix}
                -\alpha_1^*\imag^*/\sqrt{2} & -\alpha_1\imag/\sqrt{2}
        \end{smallmatrix}&\cdots&\begin{smallmatrix}
                -\alpha_k^*\imag^*/\sqrt{2} & -\alpha_k\imag/\sqrt{2}
        \end{smallmatrix}& z
\end{array}\right).
\end{align*}
The characteristic polynomial (in $\xi$)  of this latter matrix is
\begin{gather}
\label{eq:case_1_char_poly}
(z-\xi)\prod_{i=1}^k(\lambda_i-\xi)(\lambda_i^* -\xi)+ \sum_{i=1}^k \left(\Im\left(\alpha_i^2\right)\xi - \Im\left(\alpha_i^2\lambda_i\right)\right)\prod_{j\neq i} (\lambda_j - \xi)(\lambda_j^* - \xi)
\end{gather}
and is independent of $\delta>0$.
As $\lambda_i$ are all non-real, given any $x,y\in\R^k$, it is possible to construct $\alpha\in\C^k$ such that
\begin{align}
\label{eq:alpha_to_x_y}
\Im(\alpha_i^2) = y_i\quad\text{and}\quad-\Im(\alpha_i^2\lambda_i) = x_i,\quad\forall i\in[k].
\end{align}
Setting $\alpha$ in this manner reduces the characteristic polynomial to
\begin{align}
\label{eq:case_1_char_poly_xy}
(z-\xi)\prod_{i=1}^k(\lambda_i-\xi)(\lambda_i^* -\xi) + \sum_{i=1}^k \left(x_i + y_i \xi \right)\prod_{j\neq i} (\lambda_j - \xi)(\lambda_j^* - \xi).
\end{align}
It suffices to show that there exist $x,y\in\R^k$ and $z\in\R$ such that the roots of \eqref{eq:case_1_char_poly_xy} are all real, as we may take $\delta>0$ to zero independently of our choice of $x,y,z$.

Define the following polynomials.
\begin{gather*}
f_i(\xi) \coloneqq \prod_{j\neq i} (\lambda_j - \xi)(\lambda_j^* - \xi),\quad
g_i(\xi) \coloneqq \xi f_i(\xi),\,\forall i\in[k],\quad\text{and}\\
h(\xi) \coloneqq \prod_{i=1}^k(\lambda_i-\xi)(\lambda_i^* -\xi).
\end{gather*}
As $\set{\lambda_1,\lambda_1^*,\dots,\lambda_k,\lambda_k^*}$ are distinct values in $\C\setminus\R$, we have that $\set{f_1,g_1,\dots,f_k,g_k, h}$ are a basis for the degree-$2k$ polynomials in $\xi$.
Now pick $2k + 1$ distinct values $\xi_1,\dots,\xi_{2k+1}\in\R$. Note that $\set{\xi_1,\dots,\xi_{2k+1}}$ are the roots to \eqref{eq:case_1_char_poly_xy} if and only if $x,y\in\R^n$ and $z\in\R$ satisfy
\begin{align}
\label{eq:f_matrix_system}
\begin{smallpmatrix}
        \begin{smallmatrix}
        f_1(\xi_1) & g_1(\xi_1) & \cdots & f_k(\xi_1) & g_k(\xi_1)\\
        \vdots & \vdots & \ddots & \vdots& \vdots\\
        f_1(\xi_{2k+1}) & g_1(\xi_{2k+1}) & \cdots & f_k(\xi_{2k+1}) & g_k(\xi_{2k+1})
        \end{smallmatrix} &
        \begin{smallmatrix}
                h(\xi_1)\\
                \vdots\\
                h(\xi_{2k+1})
        \end{smallmatrix}
\end{smallpmatrix}
\begin{smallpmatrix}
        x_1\\
        y_1\\
        \vdots\\
        x_k\\
        y_k\\
        z
\end{smallpmatrix}
=
\begin{smallpmatrix}
\xi_1 h(\xi_1)\\
\vdots\\
\xi_{2k+1} h(\xi_{2k+1})
\end{smallpmatrix}.
\end{align}
Note that the matrix on the left is invertible (as $\set{f_1,g_1,\dots,f_k,g_k,h}$ is independent and the $\xi_i$ are distinct) and real (as the $\xi_i$ are real). Consequently, the matrix on the left has a real inverse. Note also that the vector on the right is real. We deduce that there exist $x,y\in\R^k$ (and thus $\alpha\in\C^k$) and $z\in\R$ such that the eigenvalues of $\tilde A_\delta^{-1}\tilde B_\delta$ are real and simple.

\paragraph{Case 2.}
In case 2, $S_m = \begin{smallpmatrix}
        &&F_{n_m}\\
        &0&\\
        F_{n_m}&&
\end{smallpmatrix}$ and $T_m = G_{2n_m +1}$. Set
\begin{gather}
\label{eq:perturbed_A_B_case_2}
\tilde A_\delta = \left(\begin{array}
        {@{}c|c|c|c|c|c@{}}
        \begin{smallmatrix}
                &1\\1&
        \end{smallmatrix} &&&&& \\\hline
        & \ddots &&&& \\\hline
        && \begin{smallmatrix}
                &1\\1&
        \end{smallmatrix}&&&\\\hline
        &&&&&F_{n_m}\\\hline
        &&&&\delta&\\\hline
        &&&F_{n_m}&&
\end{array}\right),\text{ and}\nonumber\\
\tilde B_\delta = \left(\begin{array}
        {@{}c|c|c|c|c|c@{}}
        \begin{smallmatrix}
                \Im(\lambda_1)&\Re(\lambda_1)\\\Re(\lambda_1)& -\Im(\lambda_1)
        \end{smallmatrix} &&&& \begin{smallmatrix}
                \sqrt{\delta}\Re(\alpha_1)\\\sqrt{\delta}\Im(\alpha_1)
        \end{smallmatrix}&\\\hline
        & \ddots && &\vdots & \\\hline
        && \begin{smallmatrix}
                \Im(\lambda_k)&\Re(\lambda_k)\\\Re(\lambda_k)& -\Im(\lambda_k)
        \end{smallmatrix}&&\begin{smallmatrix}
                \sqrt{\delta} \Re(\alpha_k)\\\sqrt{\delta}\Im(\alpha_k)
        \end{smallmatrix}&\\\hline
        &&&&&G_{n_m}\\\hline
        \begin{smallmatrix}
                \sqrt{\delta}\Re(\alpha_1) & \sqrt{\delta}\Im(\alpha_1)
        \end{smallmatrix}&\cdots&\begin{smallmatrix}
                \sqrt{\delta}\Re(\alpha_k) & \sqrt{\delta}\Im(\alpha_k)
        \end{smallmatrix}&&\delta z& e_1^*\\\hline
        &&&G_{n_m}&e_1
\end{array}\right)
\end{gather}
for some $\alpha\in\C^k$, $z\in\R$, and $\delta>0$ to be chosen later.
The eigenvalues of $\tilde A^{-1}_\delta \tilde B_\delta$ are equal to the eigenvalues of
\begin{align*}
&\left(\begin{array}
        {@{}c|c|c|c@{}}
        J &&&\\\hline
        & \ddots && \\\hline
        && J&\\\hline
        &&&\begin{smallmatrix}
                I_{n_m}/\sqrt{\delta} &&\\
                & 1/\sqrt{\delta}&\\
                && \sqrt{\delta} I_{n_m}
        \end{smallmatrix}
\end{array}\right)^{-1}\tilde A^{-1}_\delta \tilde B_\delta \left(\begin{array}
        {@{}c|c|c|c@{}}
        J &&&\\\hline
        & \ddots && \\\hline
        && J&\\\hline
        &&& \begin{smallmatrix}
                I_{n_m}/\sqrt{\delta} &&\\
                & 1/\sqrt{\delta}&\\
                && \sqrt{\delta} I_{n_m}
        \end{smallmatrix}
\end{array}\right)\\
&\qquad= \left(\begin{array}
        {@{}c|c|c|c|c|c@{}}
        \begin{smallmatrix}
                \lambda_1&\\&\lambda_1^*
        \end{smallmatrix} &&&& \begin{smallmatrix}
                \alpha_1^*/\sqrt{2} \\
                \alpha_1/\sqrt{2}
        \end{smallmatrix}&\\\hline
        & \ddots && &\vdots & \\\hline
        && \begin{smallmatrix}
                \lambda_k&\\&\lambda_k^*
        \end{smallmatrix}&&\begin{smallmatrix}
                \alpha_k^*/\sqrt{2} \\\alpha_k/\sqrt{2}
        \end{smallmatrix}&\\\hline
        &&&F_{n_m}G_{n_m}&e_{n_m}&\\\hline
        \begin{smallmatrix}
                -(\alpha_1\imag)^*/\sqrt{2} & -\alpha_1\imag/\sqrt{2}
        \end{smallmatrix}&\cdots&\begin{smallmatrix}
                -(\alpha_k\imag)^*/\sqrt{2} & -\alpha_k\imag/\sqrt{2}
        \end{smallmatrix}&&z& e_1^\intercal\\\hline
        &&&&&F_{n_m}G_{n_m}
\end{array}\right).
\end{align*}
The characteristic polynomial (in $\xi$) of this latter matrix is
\begin{align}
\label{eq:case_2_char_poly}
\xi^{2n_m}&\bigg((z-\xi)\prod_{i=1}^k(\lambda_i-\xi)(\lambda_i^* -\xi)\nonumber\\
&\qquad+ \sum_{i=1}^k \left(\Im(\alpha_i^2)\xi - \Im(\alpha_i^2\lambda_i)\right)\prod_{j\neq i} (\lambda_j - \xi)(\lambda_j^* - \xi)\bigg)
\end{align}
and is independent of $\delta>0$. As in Case 1 (cf.\ \eqref{eq:case_1_char_poly}), we may pick $\alpha\in\C^k$ and $z\in\R$ such that $\tilde A_{\delta}^{-1}\tilde B_\delta$ has real (but no longer necessarily simple) eigenvalues. Finally, applying \cref{thm:asdc_characterization_invertible}, we deduce that for all $\delta>0$, $\set{\tilde A_{\delta},\tilde B_{\delta}}$ is ASDC. We conclude that $\set{A, B}$ is ASDC.

\paragraph{Case 3.} In the final case, we have that $m= m_3+m_4=1$. If $m_4=1$ (so that $A=B=0$), it is clear that $\set{A,B}$ is actually SDC. Finally, suppose $m_3=1$ so that
\begin{align*}
A = \begin{smallpmatrix}
        &&F_{n_m}\\
        &0&\\
        F_{n_m}&&
\end{smallpmatrix},\qquad
B = G_{2n_m+1}.
\end{align*}
Then for $\delta\neq 0$, set
\begin{align*}
\tilde A_\delta = \begin{smallpmatrix}
        &&F_{n_m}\\
        &\delta&\\
        F_{n_m}&&
\end{smallpmatrix}.
\end{align*}
Note that $\tilde A^{-1}B$ is upper triangular with all diagonal entries equal to zero. Then applying \cref{thm:asdc_characterization_invertible}, we deduce that for all $\delta\neq 0$, $\set{\tilde A_\delta,B}$ is ASDC. We conclude that $\set{A, B}$ is ASDC.
\end{proof}

\begin{corollary}
Let $\cA = \set{A_1,\dots,A_m}$ in \eqref{eq:intro_qcqp} and suppose $\spann(\cA) = \spann\set{A,B}$ is singular. Then for any $\epsilon>0$, there exist $\norm{\tilde A_i - A_i}\leq \epsilon$ such that
\begin{align}
\label{eq:asdc_singular_QCQP}
\inf_{x\in\R^n}\set{x^\intercal \tilde A_1 x+ 2b_1^\intercal x + c_1:\, \begin{array}
        {l}
        x^\intercal \tilde A_i x+ 2b_i^\intercal x + c_i \,\boxempty_i\, 0,\,\forall i\in[2,m]\\
        x\in\cL
\end{array}}
\end{align}
is a diagonalizable QCQP. The matrices $\tilde A_i$ and an invertible matrix $P$ diagonalizing \eqref{eq:asdc_singular_QCQP} can be computed via the construction in \cref{thm:asdc_singular}.
\end{corollary}

\section{The ASDC property of nonsingular symmetric triples}
\label{sec:asdc_triples}
In this section, we will prove the following characterization of the ASDC property for nonsingular triples of symmetric matrices (henceforth, \emph{symmetric triples}).

\begin{theorem}
\label{thm:asdc_nonsingular_triple}
Let $\set{A,B,C}\subseteq\S^n$ and suppose $A$ is invertible. Then, $\set{A,B,C}$ is ASDC if and only if $\set{A^{-1}B, A^{-1}C}$ are a pair of commuting matrices with real eigenvalues.
\end{theorem}

As always, the forward direction follows trivially from \cref{prop:sdc_characterization_invertible} and continuity. For the reverse direction, we will extend an inductive argument due to \citet{motzkin1955pairs} to show that we may repeatedly perturb either $A^{-1}B$ or $A^{-1}C$ to increase the number of simple eigenvalues. In contrast to the original argument in \cite{motzkin1955pairs}, which establishes that any commuting pair $\set{S,T}\subseteq\C^{n\by n}$ is almost simultaneously diagonalizable \emph{via similarity} (and thus only needs to inductively maintain commutativity of $S$ and $T$), for our proof we will further need to maintain that $A,B,C$ are symmetric matrices and that $A^{-1}B$ and $A^{-1}C$ have real eigenvalues.

Our proof will require two technical facts about block matrices consisting of upper triangular Toeplitz blocks. We present these facts below and defer their proofs to \cref{sec:upper_tri_toeplitz}.

\begin{definition}
$T\in\R^{n_i\by n_j}$ is an \emph{upper triangular Toeplitz matrix} if $T$ is of the form
\begin{align*}
T = \left(\begin{array}
	{@{}c|c@{}}
0_{n_i \times (n_j - n_i)}
&
\begin{smallmatrix}
t^{(1)} & t^{(2)} & \cdots & t^{(n_i)}\\
& t^{(1)} & \ddots & \vdots\\
& & \ddots & t^{(2)}\\
& & & t^{(1)}
\end{smallmatrix}
\end{array}\right)
\quad \text{ or } \quad
T = \left(\begin{array}{@{}c@{}}
	\begin{smallmatrix}
	t^{(1)} & t^{(2)} & \cdots & t^{(n_i)}\\
	& t^{(1)} & \ddots & \vdots\\
	& & \ddots & t^{(2)}\\
	& & & t^{(1)}
	\end{smallmatrix} \\\hline
	0_{(n_i - n_j) \times n_j}
\end{array}\right)
\end{align*}
if $n_i \leq n_j$ and $n_j \leq n_i$ respectively. \mathprog{\qed}
\end{definition}

\begin{definition}
Let $(n_1, \dots, n_k)$ such that $\sum_i n_i = n$. Let $\T(n_1,\dots,n_k)\subseteq\R^{n\by n}$ denote the linear subspace of matrices $T$ such that each block $T_{i,j}$ (when the rows and columns of $T$ are partitioned according to $(n_1,\dots,n_k)$) is an upper triangular Toeplitz matrix. When the partition $(n_1,\dots,n_k)$ is clear from context, we will simply write $\T$. \mathprog{\qed}
\end{definition}

The following fact characterizes the set of matrices which commute with a nilpotent Jordan chain (see for example \cite[Theorem 6]{suprunenko1968commutative}).
\begin{lemma}
\label{lem:T_iff_commuting}
Let $(n_1, \dots, n_k)$ such that $\sum_i n_i = n$. Let $\cJ\in\R^{n\by n}$ be a block diagonal matrix with diagonal block $\cJ_{i,i} = F_{n_i}G_{n_i}$, i.e., a nilpotent Jordan block of size $n_i$. Then, $T\in\R^{n\by n}$ commutes with $\cJ$ if and only if $T\in\T$.
\end{lemma}

\begin{definition}
Let $(n_1,\dots,n_k)$ such that $\sum_i n_i = n$. Define the linear map $\Pi_{(n_1,\dots,n_k)}: \T(n_1,\dots,n_k) \to \R^{k \by k}$ by
\begin{align*}
(\Pi_{(n_1,\dots,n_k)}(T))_{i,j} = \begin{cases}
	T_{i,j}^{(1)} & \text{if } n_i = n_j,\\
	0 & \text{else}.
\end{cases}
\end{align*}
When the partition $(n_1,\dots,n_k)$ is clear from context, we will simply write $\Pi$.\mathprog{\qed}
\end{definition}

The following fact follows from the observation that the characteristic polynomial of a matrix $T\in\T$ depends on only a few of its entries (see \cref{lem:toeplitz_main_diagonals}).

\begin{restatable}{lemma}{toeplitzmapping}\label{lem:toeplitz_mapping}
Let $(n_1,\dots,n_k)$ such that $\sum_i n_i = n$. Then, for any $T\in \T$, the matrices $T\in\R^{n\by n}$ and $\Pi(T)\in\R^{k\by k}$ have the same eigenvalues.
\end{restatable}

We are now ready to prove \cref{thm:asdc_nonsingular_triple}.

\begin{proof}[Proof of \cref{thm:asdc_nonsingular_triple}]
It suffices to show that if $\set{A^{-1}B,A^{-1}C}$ are a pair of commuting matrices with real eigenvalues, then $\set{A,B,C}$ is ASDC.
Note that any set $\set{A,B,C}\subseteq\S^1$ is SDC. Thus, we may assume that $n\geq 2$ and that the statement is true inductively for all smaller $n$.

We make the following simplifying assumptions:
Without loss of generality, we may assume that either $A^{-1}B$ has multiple eigenvalues or that $A^{-1}B$ and $A^{-1}C$ are both nilpotent. Indeed, if $A^{-1}B$ and $A^{-1}C$ both have a single eigenvalue, then we may consider the basis $\set{A,B+\lambda_B A, C+\lambda_C A}$ for $\spann\set{A,B,C}$ where $A^{-1}(B+\lambda_B A) = A^{-1}B + \lambda_B I$ and $A^{-1}(C+\lambda_C A) = A^{-1}C + \lambda_C I$ are both nilpotent.
We will work in the basis for $\R^n$ furnished by \cref{prop:canonical_form} so that $A^{-1}B$ is in Jordan canonical form (note that $m_2 = m_3 =m_4 = 0$ by the assumptions that $A$ is invertible and $A^{-1}B$ has real eigenvalues) and assume that the blocks of $A^{-1}B$ are ordered first according to increasing eigenvalue then increasing block sizes.

We will break our proof into four cases:
First, we will consider where $A^{-1}B$ has multiple eigenvalues.
The remaining three cases will consider when the Jordan block structure of $A^{-1}B$ has: multiple block sizes, multiple blocks of the same size, and a single block.

\paragraph{Case 1.} Suppose $A^{-1}B$ has $\ell$-many distinct eigenvalues. Write $C$ as an $\ell\times \ell$ block matrix according to the partition induced by the eigenvalues of $A^{-1}B$. Then, as $A^{-1}C$ and $A^{-1}B$ commute, we have that $A^{-1}C$ (perforce $C$) is block diagonal. Thus, according to the block structure induced by the eigenvalues of $A^{-1}B$, the matrices $A$, $B$, $C$ are jointly block diagonal, with each diagonal block satisfying the conditions of the inductive hypothesis. We conclude that $\set{A,B,C}$ is ASDC.

\paragraph{Case 2.} Suppose $A^{-1}B$ and $A^{-1}C$ are nilpotent and that $A^{-1}B$ has distinct block sizes. For concreteness, suppose $A^{-1}B$ has $k$ blocks of size $\eta = n_1 =  \dots = n_k < n_{k+1} \leq \dots \leq n_m$. By \cref{prop:canonical_form},
\begin{align*}
A = \Diag(\sigma_1 F_\eta,\dots, \sigma_k F_\eta, \sigma_{k+1} F_{n_{k+1}},\dots, \sigma_m F_{n_m})
\end{align*}
for some $\sigma_i \in\set{\pm 1}$. Set
\begin{align}
\label{eq:case_2_construction}
\tilde C = C + \delta\Diag(\sigma_1 F_\eta, \dots, \sigma_k F_\eta, 0_{n_{k+1}},\dots, 0_{n_m}).
\end{align}
Applying \cref{lem:T_iff_commuting}, we have that $A^{-1}\tilde C$ commutes with $A^{-1}B$ and that $\tilde C\in\S^n$.
Let $\Pi$ denote the linear map furnished by \cref{lem:toeplitz_mapping}. As $n_i\neq n_j$ for all $i\leq k$ and $j\geq k+1$, we have that $\Pi(A^{-1}C)$  can be written as a block diagonal matrix
\begin{align*}
\Pi(A^{-1}C) = \begin{smallpmatrix}
	\Pi(A^{-1}C)_{1,1} & \\
	& \Pi(A^{-1}C)_{2,2}
\end{smallpmatrix}
\end{align*}
with blocks of size $k\times k$ and $(m - k)\times (m - k)$ respectively. As $\Pi$ preserves eigenvalues for inputs in $\T$, we have that $\Pi(A^{-1}C)_{1,1}$ and $\Pi(A^{-1}C)_{2,2}$ are both nilpotent. Then, as $A^{-1}\tilde C$ has the same eigenvalues as
\begin{align*}
\Pi(A^{-1}\tilde C) = \begin{smallpmatrix}
		\Pi(A^{-1}C)_{1,1} + \delta I_{k} & \\
	& \Pi(A^{-1}C)_{2,2}
\end{smallpmatrix},
\end{align*}
we deduce that $A^{-1}\tilde C$ has eigenvalues $\set{0,\delta}$. We have reduced to case (1) whence $\set{A,B,\tilde C}$ is ASDC. We conclude that $\set{A,B,C}$ is ASDC.

\paragraph{Case 3.} Suppose $A^{-1}B$ and $A^{-1}C$ are nilpotent and that $A^{-1}B$ has Jordan blocks all of the same dimension. For concreteness, suppose $A^{-1}B$ has $m\geq 2$ Jordan blocks of dimension $\eta$.
In this case \cref{prop:canonical_form} states that
\begin{align*}
A = \Diag(\sigma_1,\dots,\sigma_m)\otimes F_\eta
\quad\text{and}\quad
B = \Diag(\sigma_1,\dots,\sigma_m)\otimes G_\eta
\end{align*}
where $\sigma_i \in\set{\pm 1}$. Write $C$ as a $m\by m$ block matrix with blocks $C_{i,j}\in\R^{\eta \by \eta}$.
By \cref{lem:T_iff_commuting}, $A^{-1}C\in\T$ and we may write
\begin{align*}
C_{i,j} = F_\eta\left(\gamma_{i,j}^{(1)} I_\eta + \sum_{\ell=2}^\eta \gamma_{i,j}^{(\ell)}(F_\eta G_\eta)^{\ell - 1}\right).
\end{align*}
Let $\Pi$ denote the linear map furnished by \cref{lem:toeplitz_mapping}. Let
\begin{align}
\label{eq:case_3_construction}
\bar A = \Diag(\sigma_1,\dots,\sigma_m) \quad\text{and}\quad
\bar C = \begin{pmatrix}
	\gamma_{i,j}^{(1)}
\end{pmatrix}.
\end{align}
Note that as $C\in\S^n$, we have $\gamma_{i,j}^{(1)} = \gamma_{j,i}^{(1)}$, whence $\bar A,\bar C\in\S^m$. As $\Pi$ preserves the eigenvalues for inputs in $\T$ and $\bar A^{-1}\bar C = \Pi(A^{-1}C)$, we deduce that $\bar A^{-1}\bar C$ has real eigenvalues (in fact, the single eigenvalue $0$). Then applying \cref{lem:perturbed_canonical_form}, there exists $\bar C'\in\S^m$ such that $\norm{\bar C - \bar C'}\leq \delta$ and $\bar A^{-1}\bar C'$ has $m$-many distinct real eigenvalues. Finally, set
\begin{align*}
\tilde C &= C + (\bar C' - \bar C) \otimes F_\eta.
\end{align*}
Then \cref{lem:T_iff_commuting} implies that $A^{-1}B$ and $A^{-1}\tilde C$ commute. Furthermore, by construction, $A^{-1}\tilde C$ has upper triangular  Toeplitz blocks so that its eigenvalues are the same as the eigenvalues of $\Pi(A^{-1}\tilde C) = \bar A^{-1}\bar C'$. We have reduced to case (1) and $\set{A, B, \tilde C}$ is ASDC. We conclude that $\set{A,B,C}$ is also ASDC.

\paragraph{Case 4.} Suppose $A^{-1}B$ and $A^{-1}C$ are nilpotent and that $A^{-1}B$ is a single Jordan block.
Then, by \cref{prop:canonical_form}, $A = \sigma F_n$ and $B = \sigma G_n$ for some $\sigma\in\set{\pm 1}$. Furthermore, by \cref{lem:T_iff_commuting} and the assumption that $A^{-1}C$ is nilpotent, we may write
\begin{align*}
C = \sigma F_n\left(\sum_{i=2}^{n} c_i (F_n G_n)^{i-1}\right)
\end{align*}
for some $c_2,\dots,c_n\in\R$.

If $n = 2$, then $C = c_2 \sigma G_2$. We may set
\begin{align}
\label{eq:case_4_construction_n_eq_2}
\tilde B = \sigma (G_2 + \delta H_2)\quad\text{and}\quad
\tilde C = c_2 \sigma(G_2 +\delta H_2).
\end{align}
Then, $\set{A^{-1}\tilde B, A^{-1}\tilde C}$ are a pair of commuting matrices with real simple eigenvalues.

If $n\geq 3$, set
\begin{align}
\label{eq:case_4_construction_n_geq_3}
\tilde B = B + \delta\sigma(e_1e_n^\intercal + e_ne_1^\intercal)
\quad\text{and}\quad
\tilde C = C + \sigma(e_n\gamma^\intercal + \gamma e_n^\intercal)
\end{align}
where $\gamma\in\R^n$ is defined recursively as $\gamma_n = \gamma_{n-1} = 0$ and $\gamma_i = \delta(c_{i+1} + \gamma_{i+1})$ for $i \in[n-2]$. A straightforward calculation shows that $A^{-1}\tilde B$ and $A^{-1}\tilde C$ commute and both have real eigenvalues. Finally, as $A^{-1}\tilde B$ has distinct eigenvalues $\set{0,\delta}$, we have reduced to case (1) and $\set{A, \tilde B, \tilde C}$ is ASDC. We conclude that $\set{A,B,C}$ is also ASDC.
\end{proof}

\begin{corollary}
Let $\cA = \set{A_1,\dots,A_m}$ in \eqref{eq:intro_qcqp} and suppose $\spann(\cA)=\spann\set{A,B,C}$ where $A$ is invertible.
Furthermore, suppose $A^{-1}B$ and $A^{-1}C$ commute and have real eigenvalues. Then, for any $\epsilon>0$, there exist $\norm{\tilde A_i - A_i}\leq \epsilon$ such that
\begin{align}
\label{eq:asdc_nonsingular_triple_QCQP}
\inf_{x\in\R^n}\set{x^\intercal \tilde A_1 x+ 2b_1^\intercal x + c_1:\, \begin{array}
        {l}
        x^\intercal \tilde A_i x+ 2b_i^\intercal x + c_i \,\boxempty_i\, 0,\,\forall i\in[2,m]\\
        x\in\cL
\end{array}}
\end{align}
is a diagonalizable QCQP. The matrices $\tilde A_i$ and an invertible matrix $P$ diagonalizing \eqref{eq:asdc_nonsingular_triple_QCQP} can be computed via the construction in \cref{thm:asdc_nonsingular_triple}.
\end{corollary}


\section{Restricted SDC}
\label{sec:rsdc}
In this section, we investigate a second new notion of simultaneous diagonalizability called \emph{restricted} SDC. We will see soon that we have in fact already seen this property before in \cref{sec:asdc_pairs}.

\begin{definition}
\label{def:rsdc}
Let $\cA\subseteq\S^n$ and $d\in\N$. We will say that $\cA$ is \emph{$d$-restricted SDC} ($d$-RSDC) if there exist matrices $\bar A\in\S^{n+d}$ containing $A$ as its top-left $n\by n$ principal submatrix for every $A\in\cA$ such that $\set{\bar A:\, A\in\cA}$ is SDC.\mathprog{\qed}
\end{definition}

We record some simple consequences of the $d$-RSDC property that follow from \cref{obs:sdc_span_subset} and \cref{lem:sdc_closed_under_padding}.
\begin{observation}
\label{obs:rsdc_span_subset}
Let $\cA\subseteq\S^n$ and $d\in\N$.
\begin{itemize}
        \item $\cA$ is $d$-RSDC if and only if $\set{A_1,\dots,A_m}$ is $d$-RSDC for some basis $\set{A_1,\dots,A_m}$ of $\spann(\cA)$.
        \item If $\cA$ is $d$-RSDC, then $\cA$ is $d'$-RSDC for all $d'\geq d$.
\end{itemize}
\end{observation}

The following lemma explains the connection between the $d$-RSDC property and the ASDC property.
\begin{lemma}
\label{lem:asdc_almost_rsdc}
Let $A_1,\dots,A_m \in\S^n$ and let $d\in\N$. If $\cA = \set{A_1,\dots,A_m}$ is $d$-RSDC, then
\begin{align*}
\cA\oplus 0_d \coloneqq \set{\begin{pmatrix}
        A_i &\\
        & 0_d
\end{pmatrix}:\, i\in[m]}
\end{align*}
is ASDC. On the other hand, if $\cA\oplus 0_d$ is ASDC, then for all $\epsilon>0$, there exist $\tilde A_1,\dots,\tilde A_m\in\S^n$ such that
\begin{itemize}
        \item for all $i\in[m]$, the spectral norm $\norm{A_i - \tilde A_i}\leq \epsilon$, and
        \item $\set{\tilde A_1,\dots,\tilde A_m}$ is $d$-RSDC.
\end{itemize}
\end{lemma}
\begin{proof}
First, suppose $\set{A_1,\dots,A_m}$ is $d$-RSDC and let $\set{\tilde A_1,\dots,\tilde A_m}\subseteq\S^{n+d}$ denote the matrices furnished by $d$-RSDC. Next, let $\epsilon>0$ and set
\begin{align*}
P = \begin{pmatrix}
        I_n &\\
        & \sqrt{\epsilon} I_d
\end{pmatrix}.
\end{align*}
Clearly, $P$ is invertible so that $\set{P^\intercal\tilde A_i P:\, i\in[m]}$ is also SDC. Then, note that
\begin{align*}
P^\intercal \tilde A_i P &= P^\intercal \begin{pmatrix}
        A_i & (\tilde A_i)_{1,2}\\
        (\tilde A_i)_{1,2}^* & (\tilde A_i)_{2,2}
\end{pmatrix}P = \begin{pmatrix}
        A_i & \sqrt{\epsilon}(\tilde A_i)_{1,2}\\
        \sqrt{\epsilon}(\tilde A_i)_{1,2}^* & \epsilon (\tilde A_i)_{2,2}
\end{pmatrix}
\end{align*}
so that $\cA\oplus 0_d$ is ASDC.

Next, suppose $\cA\oplus 0_d$ is ASDC and let $\epsilon>0$. Then, there exist $\bar A_1,\dots,\bar A_m\in\S^{n+d}$ such that $\norm{\bar A_i - A_i \oplus 0_d}\leq \epsilon$ and $\set{\bar A_1,\dots,\bar A_m}$ is SDC. Finally, note that $\norm{A_1 - (\bar A_1)_{1,1}}\leq \epsilon$.
\end{proof}

\begin{remark}
While the restriction of an SDC set does not necessarily result in an SDC set, there is a setting arising naturally when analyzing QCQPs in which the restriction of an SDC set is again SDC. Specifically, let $Q_1,\dots,Q_m\in\S^{n+1}$ where $Q_i$ has $A_i$ as its top-left $n\by n$ principal submatrix. Furthermore suppose that there exists a positive definite matrix in $\spann(\set{A_1,\dots,A_m})$. Then, if $\set{Q_1,\dots,Q_m,e_{n+1}e_{n+1}^\intercal}$ is SDC, so is $\set{A_1,\dots,A_m}$. In words, if the homogenized quadratic forms in a QCQP, along with $e_{n+1}e_{n+1}^\intercal$, are SDC, then so are the original quadratic forms (under a standard ``definiteness'' assumption). See \cref{sec:hom_to_inhom_SDC} for details.\mathprog{\qed}
\end{remark}

\subsection{$1$-restricted SDC}
We record a recasting of \cref{thm:asdc_singular} in terms of these new definitions.
\begin{theorem}
\label{thm:rsdc_pair}
Let $A,B\in\S^n$.
Then for every $\epsilon>0$, there exist $\tilde A,\tilde B\in\S^n$ such that $\norm{A - \tilde A},\norm{B - \tilde B}\leq \epsilon$ and $\set{\tilde A,\tilde B}$ is $1$-RSDC.
Furthermore, if $A$ is invertible and $A^{-1}B$ has simple eigenvalues, then $\set{A,B}$ is itself $1$-RSDC.
\end{theorem}
\begin{proof}
The first claim follows from \cref{thm:asdc_singular} and \cref{lem:asdc_almost_rsdc} applied to $\set{A,B}\oplus 0_1$.
The second claim follows from the additional observation that if $A$ is invertible and $A^{-1}B$ has simple eigenvalues, then the construction of \cref{thm:asdc_singular} follows case 1 and \emph{does not} perturb either $A$ or $B$ (see also \cref{alg:1_rsdc}).
\end{proof}

\begin{algorithm}
        [t]
        \caption{$1$-RSDC construction}
        \label{alg:1_rsdc}
        \textbf{Input}: $A,\,B\in\S^n$ such that $A$ is invertible and $A^{-1}B$ has simple eigenvalues\\
       { \textbf{Output}: $\tilde A,\tilde B\in\S^{n+1}$ such that $\tilde A$ has $A$ as its top-left submatrix, $\tilde B$ has $B$ as its top-left submatrix, and ${\tilde A,\tilde B}$ is SDC}
        \begin{enumerate}
        \item Let $P\in\R^{n\by n}$ denote the invertible matrix guaranteed by \cite{uhlig1976canonical}; this can be computed using an eigenvalue decomposition of $A^{-1}B$. Then $P^\intercal AP = \Diag(\sigma_1,\dots,\sigma_r, F_2,\dots,F_2)$ and $P^\intercal B P = \Diag(\sigma_1\mu_1,\dots,\sigma_r\mu_r,T_1,\dots,T_k)$. Here, $\sigma_1,\dots,\sigma_r \in\set{\pm 1}$, $\mu_1,\dots,\mu_r\in\R$ and for $i \in[k]$, the matrix $T_i$ has the form
        \begin{align*}
        T_i = \begin{pmatrix}
                \Im(\lambda_i) & \Re(\lambda_i)\\
                \Re(\lambda_i) & -\Im(\lambda_i)
        \end{pmatrix}
        \end{align*}
        for some $\lambda_i\in \C\setminus\R$.
        \item Choose an arbitrary set of $2k+1$ distinct points $\xi_1,\dots,\xi_{2k+1}\in\R$
        \item Solve for $x,\,y\in\R^k$ and $z\in\R$ in the linear system \eqref{eq:f_matrix_system}
        \item Let $\alpha\in\C^k$ solve \eqref{eq:alpha_to_x_y} and define $\gamma\in\R^{r+2k}$ as
        \begin{align*}
        \gamma = \begin{pmatrix}
                0_{1\times k} &
                \Re(\alpha_1) &
                \Im(\alpha_1) &
                \hdots &
                \Re(\alpha_k) &
                \Im(\alpha_k)
        \end{pmatrix}^\intercal.
        \end{align*}
        \item Return
        \begin{align*}
        \tilde A = \begin{pmatrix}
                A &\\
                & 1
        \end{pmatrix},\qquad
        \tilde B =
        \begin{pmatrix}
                B & P^{-\intercal}\gamma\\
                \gamma^\intercal P^{-1} & z
        \end{pmatrix}.
        \end{align*}
\end{enumerate}
\end{algorithm}

\begin{corollary}
Let $\cA = \set{A_1,\dots,A_m}$ in \eqref{eq:intro_qcqp} and suppose $\spann(\cA) = \set{A,B}$. Then, for any $\epsilon>0$, there exist $\bar A_i\in\S^{n+1}$ such that $\norm{(\bar A_i)_{1,1} - A_i}\leq \epsilon$ and
\begin{align}
\label{eq:1rsdc_qcqp}
\inf_{x\in\R^n,w}\set{\begin{pmatrix}
        x\\
        w
\end{pmatrix}^\intercal \bar A_1 \begin{pmatrix}
        x\\w
\end{pmatrix}+ 2b_1^\intercal x + c_1:\, \begin{array}
        {l}
        \begin{pmatrix}
        x\\
        w
\end{pmatrix}^\intercal \bar A_i \begin{pmatrix}
        x\\w
\end{pmatrix}+ 2b_i^\intercal x + c_i \,\boxempty_i\, 0,\,\forall i\in[2,m]\\
        x\in\cL\\
        w = 0
\end{array}}
\end{align}
is a diagonalizable QCQP.
If $A$ is invertible and $A^{-1}B$ has simple eigenvalues, then $(\bar A_i)_{1,1}$ can be taken to be equal to $A_i$.
The matrices $\bar A_i$ and an invertible matrix $P$ diagonalizing \eqref{eq:asdc_nonsingular_triple_QCQP} can be computed via \cref{proc:simple_eigs,alg:1_rsdc}.
\end{corollary}


\subsection{$d$-restricted SDC}
Let $\set{A,B}\subseteq\S^n$ such that $A$ is invertible and $A^{-1}B$ has simple eigenvalues.
By \cref{obs:rsdc_span_subset} and \cref{thm:rsdc_pair}, we have that $\set{A,B}$ is $d$-RSDC for any $d\geq 1$. In this section, we record an alternate construction showing that $\set{A,B}$ is $d$-RSDC for $d\geq 1$. This alternate construction applies \cref{alg:1_rsdc} on smaller blocks of the canonical form and has empirically better numerical performance in QCQP applications (see Section \ref{sec:experiments}).

\begin{theorem}\label{thm:d_rsdc_pair}
Let $A,B\in\S^n$ such that $A$ is invertible and $A^{-1}B$ has simple eigenvalues. Then, $\set{A,B}$ is $d$-RSDC for any $d\geq 1$.
\end{theorem}
\begin{proof}
Without loss of generality, we may assume that $A,B$ are in canonical form (\cref{prop:canonical_form}) and $m_1 = 0$ (else consider the submatrix of $A,B$ corresponding to the remaining blocks).

Partition $[m]$ into $d$-many contiguous subintervals. Write $A$ and $B$ as diagonal block matrices of diagonal block matrices according to the partition of $[m]$. In other words, write $A = \Diag(A_1,\dots,A_d)$ and $B = \Diag(B_1,\dots,B_d)$ where each $A_i$ is a diagonal block matrix of $F_2$-matrices and each $B_i$ is a diagonal block matrix with matrices of the form $\begin{smallpmatrix}
        \Im(\lambda)&\Re(\lambda)\\\Re(\lambda)&-\Im(\lambda)
\end{smallpmatrix}$ for $\lambda\in\C\setminus\R$.
Set
\begin{align}
\label{eq:d_rsdc_construction}
\tilde A = \left(\begin{array}
        {@{}c|c|c|c|c|c@{}}
        A_1 & & &&&\\\hline
        &\ddots&&&&\\\hline
        &&A_d&&&\\\hline
        &&&1&&\\\hline
        &&&&\ddots&\\\hline
        &&&&&1
\end{array}\right)
\quad\text{and}\quad
\tilde B = \left(\begin{array}
        {@{}c|c|c|c|c|c@{}}
        B_1 & & & x_1&&\\\hline
        &\ddots&&&\ddots&\\\hline
        &&B_d&&&x_d\\\hline
        x_1^\intercal &&&z_1&&\\\hline
        &\ddots&&&\ddots&\\\hline
        &&x_d^\intercal &&&z_d
\end{array}\right)
\end{align}
for $z_1,\dots,z_d\in\R$ and vectors $x_1,\dots,x_d$ of the appropriate dimensions to be chosen later.
After a simultaneous permutation of the coordinates, we can write $\tilde A$ and $\tilde B$ as diagonal block matrices with blocks of the form
\begin{align*}
\begin{pmatrix}
        A_i&\\&1
\end{pmatrix}\quad\text{and}\quad
\begin{pmatrix}
        B_i&x_i\\x_i^\intercal& z_i
\end{pmatrix}.
\end{align*}
By \cref{thm:rsdc_pair} (summarized in \cref{alg:1_rsdc}) and the assumption that $A^{-1}B$ has simple eigenvalues, we may, for each $i\in[d]$, pick $x_i\in\R^n$ and $z_i\in\R$ such that the pair of matrices above is SDC. It remains to note that the diagonal block concatenation of SDC matrices is SDC.
\end{proof}

\begin{algorithm}
        [t]
        \caption{$d$-RSDC construction}
        \label{alg:d_rsdc}
        \textbf{Input}: $A,\,B\in\S^n$ such that $A$ is invertible and $A^{-1}B$ has simple eigenvalues\\
        {\textbf{Output}: $\tilde A,\tilde B\in\S^{n+d}$ such that $\tilde A$ has $A$ as its top-left submatrix, $\tilde B$ has $B$ as its top-left submatrix, and ${\tilde A,\tilde B}$ is SDC}
        \begin{enumerate}
        \item Let $P\in\R^{n\by n}$ denote the invertible matrix guaranteed by \cite{uhlig1976canonical}; this can be computed using an eigenvalue decomposition of $A^{-1}B$. Then $P^\intercal AP = \Diag(\sigma_1,\dots,\sigma_r, F_2,\dots,F_2)$ and $P^\intercal B P = \Diag(\sigma_1\mu_1,\dots,\sigma_r\mu_r,T_1,\dots,T_k)$. Here, $\sigma_1,\dots,\sigma_r \in\set{\pm 1}$, $\mu_1,\dots,\mu_r\in\R$ and for $i \in[k]$, the matrix $T_i$ has the form
        \begin{align*}
        T_i = \begin{pmatrix}
                \Im(\lambda_i) & \Re(\lambda_i)\\
                \Re(\lambda_i) & -\Im(\lambda_i)
        \end{pmatrix}
        \end{align*}
        for some $\lambda_i\in \C\setminus\R$.
        \item Partition $[k]= \kappa_1\cup \dots\kappa_d$ into contiguous subintervals where $\kappa_i=[\text{start}_i,\text{end}_i]$
        \item For each $i\in[d]$, apply \cref{alg:1_rsdc} to get $x_i\in\R^{\abs{\kappa_i}}$ and $z_i\in\R$ such that
        \begin{align*}
        \left(\begin{array}
                {@{}c|c@{}}
                \begin{smallmatrix}
                        F_2\\&\ddots\\&&F_2
                \end{smallmatrix} & \\\hline
                & 1
        \end{array}\right)
        \quad\text{and}\quad
        \left(\begin{array}
                {@{}c|c@{}}
                \begin{smallmatrix}
                        T_{\text{start}_i}\\&\ddots\\&&T_{\text{end}_i}
                \end{smallmatrix} & x_i\\\hline
                x_i^\intercal& z_i
        \end{array}\right)
        \end{align*}
        are SDC
        \item Let $Q \coloneqq P \oplus I_1$ and return
        \begin{align*}
        Q^{-\intercal} \tilde A Q^{-1} \quad\text{and}\quad
        Q^{-\intercal} \tilde B Q^{-1}
        \end{align*}
        where $\tilde A$ and $\tilde B$ are defined in \eqref{eq:d_rsdc_construction}.
\end{enumerate}
\end{algorithm}

\begin{corollary}
Let $\cA = \set{A_1,\dots,A_m}$ in \eqref{eq:intro_qcqp} and suppose $\spann(\cA) = \set{A,B}$. Then, for any $\epsilon>0$, there exist $\bar A_i\in\S^{n+d}$ such that $\norm{(\bar A_i)_{1,1} - A_i}\leq \epsilon$ and
\begin{align}
\label{eq:drsdc_qcqp}
\inf_{x\in\R^n,w\in\R^d}\set{\begin{pmatrix}
        x\\
        w
\end{pmatrix}^\intercal \bar A_1 \begin{pmatrix}
        x\\w
\end{pmatrix}+ 2b_1^\intercal x + c_1:\, \begin{array}
        {l}
        \begin{pmatrix}
        x\\
        w
\end{pmatrix}^\intercal \bar A_i \begin{pmatrix}
        x\\w
\end{pmatrix}+ 2b_i^\intercal x + c_i \,\boxempty_i\, 0,\,\forall i\in[2,m]\\
        x\in\cL\\
        w = 0
\end{array}}
\end{align}
is a diagonalizable QCQP.
If $A$ is invertible and $A^{-1}B$ has simple eigenvalues, then $(\bar A_i)_{1,1}$ can be taken to be equal to $A_i$.
The matrices $\bar A_i$ and an invertible matrix $P$ diagonalizing \eqref{eq:asdc_nonsingular_triple_QCQP} can be computed via \cref{proc:simple_eigs,alg:d_rsdc}.
\end{corollary}

{

\subsection{Relation between $d$-RSDC and a naive diagonalization}

The following lemma states that the
quantity $d$ in
$d$-RSDC is equivalent to the number of additional vectors in $\R^n$ required to diagonalize a set of matrices in the following sense.

\begin{lemma}
    \label{lem:naive_lifting}
Let $\cA\subseteq\S^n$. Then $\cA$ is $d$-RSDC if and only if there exists $\set{\ell_1,\dots,\ell_{n+d}}\subseteq\R^n$ such that for each $A\in\cA$ we can write
\begin{align*}
    A = \sum_{i=1}^{n+d} \mu_i \ell_i\ell_i^\intercal
\end{align*}
for some choice of $\mu\in\R^{n+d}$.
\end{lemma}
\begin{proof}
Suppose $\cA$ is $d$-RSDC. Let $\bar \cA\subseteq\S^{n+d}$ denote the SDC set furnished by the definition of $d$-RSDC.
Let $\bar \ell_1,\dots,\bar\ell_{n+d}\in\R^{n+d}$ be a basis diagonalizing $\bar \cA$.
Also define $\ell_1,\dots,\ell_{n+d}\in\R^{n}$ to be the vectors formed by taking the first $n$ coordinates of $\bar\ell_i$.
Now, suppose $A \in\cA$ and let $\bar A \in \bar\cA$ denote a matrix with $A$ in its top-left $n\by n$ minor. Then, there exists $\mu\in\R^{n+d}$ such that
\begin{align*}
    \bar A = \sum_{i=1}^{n+d} \mu_i \bar \ell_i\bar\ell_i^\intercal.
\end{align*}
Note that the top-left block of $\bar A$ is $A$ and the top-left block of $\bar \ell_i\bar\ell_i^\intercal$ is $\ell_i\ell_i^\intercal$. We conclude that every $A\in\cA$ can be written in the form
\begin{align*}
    A = \sum_{i=1}^{n+d} \mu_i \ell_i\ell_i^\intercal
\end{align*}
for some choice of $\mu$. Note that as $\set{\bar\ell_i}$ spans $\R^{n+d}$, it also holds that $\set{\ell_i}$ spans $\R^n$.

In the reverse direction, suppose $\ell_1,\dots,\ell_{n+d}\in\R^n$ span $\R^n$ and are such that each $A\in\cA$ can be written in the form $A = \sum_{i}\mu_i \ell_i\ell_i^\intercal$ for some choice of $\mu$. Form the matrix $L = (\ell_1,\dots,\ell_{n+d})\in\R^{n\by (n+d)}$. As this matrix has full column rank, it can be extended to an invertible matrix $\bar L\in\R^{(n+d)\by (n+d)}$ where the top $n$ rows of $\bar L$ coincide with $L$. Now, if $A\in\cA$ is given by $\sum_{i=1}^{n+d}\mu_i \ell_i\ell_i^\intercal$, we also have that
\begin{align*}
    \bar L \Diag(\mu)\bar L^\intercal
\end{align*}
has $A$ as its top-left $n\by n$ principal minor.
\end{proof}

\begin{remark}
    \label{rem:naive_diagionalization}
From this lemma, we see that any set of $m$ matrices $\cA=\set{A_1,\dots,A_m}\subseteq\S^n$ is naively $(m-1)n$-RSDC and that we can ``lift'' it to a diagonlizable set by simply diagonalizing each matrix independently.
\end{remark}}

\section{Obstructions to further generalization}
\label{sec:obstructions}
In this section, we record explicit counterexamples to \textit{a priori} plausible extensions to \cref{thm:asdc_singular,thm:asdc_nonsingular_triple,thm:asdc_characterization_invertible}.

\subsection{Singular symmetric triples}
\label{subsec:sing_triples_obstruction}
In \cref{thm:asdc_singular}, we showed that any singular symmetric pair is ASDC. A natural question to ask is whether any singular set of symmetric matrices (regardless of the dimension of its span) is ASDC. The following theorem presents an obstruction to generalizations in this direction. Specifically, in contrast to \cref{thm:asdc_singular} (where it was shown that singularity implies ASDC in the context of symmetric pairs), \cref{thm:commutator_lower_bound_kernel} below shows that even symmetric \emph{triples} with \emph{``large amounts''} of singularity can fail to be ASDC.

\begin{theorem}
\label{thm:commutator_lower_bound_kernel}
Let $\set{A = I_n, B, C}\subseteq\S^n$. Then, if $d < \rank([B,C])/2$, the set
\begin{align*}
\set{\begin{pmatrix}
A &\\
& 0_d
\end{pmatrix}, \begin{pmatrix}
	B &\\
	& 0_d
\end{pmatrix}, \begin{pmatrix}
	C &\\
	& 0_d
\end{pmatrix}}
\end{align*}
is not ASDC.
\end{theorem}
\begin{proof}
Suppose for the sake of contradiction that this set is ASDC. Let $\epsilon\in(0,1/2)$ and let $\set{\tilde A, \tilde B, \tilde C}\subseteq\S^{n+d}$ denote an SDC set furnished by the ASDC assumption. Without loss of generality, $\tilde A$ has rank $n + d$.
Write
\begin{align*}
\tilde A = \begin{pmatrix}
	\tilde A_{1,1} & \tilde A_{1,2}\\
	\tilde A_{1,2}^\intercal & \tilde A_{2,2}
\end{pmatrix}.
\end{align*}
Similarly decompose $\tilde B$ and $\tilde C$.
As $\epsilon \in(0,1/2)$, we have that $\tilde A_{1,1}$ is invertible.
Let
\begin{align*}
P = \begin{pmatrix}
	\tilde A_{1,1}^{-1/2} & -\tilde A_{1,1}^{-1}\tilde A_{1,2}\\
	0 & I_d
\end{pmatrix}.
\end{align*}
Then as $P$ is invertible, $\set{P^\intercal \tilde AP, P^\intercal\tilde BP, P^\intercal\tilde CP}$ is again SDC.
Note that $P^\intercal\tilde AP$ has the form
\begin{gather*}
P^\intercal\tilde A P = \begin{pmatrix}
	I_n & \\
	& \tilde A_{2,2} - \tilde A_{1,2}^\intercal \tilde A_{1,1}^{-1} \tilde A_{1,2}
\end{pmatrix}.
\end{gather*}
Furthermore,
\begin{align*}
&\norm{P^\intercal\tilde B P - B}\\
&\qquad= \norm{(P-I_{n+d})^\intercal\tilde B (P-I_{n+d}) + \tilde B (P-I_{n+d}) + (P-I_{n+d})^\intercal\tilde B + (\tilde B - B)}\\
&\qquad\leq \norm{\tilde B}\norm{P-I_{n+d}}^2 + 2\norm{\tilde B}\norm{P-I_{n+d}} + \epsilon.
\end{align*}
We claim that $\norm{P - I_{n+d}}$ can be bounded in terms of $\epsilon$:
\begin{align*}
\norm{P-I_{n+d}} &\leq \norm{\tilde A_{1,1}^{-1/2} - I} + \norm{\tilde A_{1,1}^{-1}}\norm{\tilde A_{1,2}}\\
&\leq \max\set{\frac{1}{\sqrt{1-\epsilon}} - 1,\, 1 - \frac{1}{\sqrt{1+\epsilon}}} + \frac{\epsilon}{1-\epsilon}\\
&\leq \frac{2\epsilon}{1-\epsilon}.
\end{align*}
Here, we have used the fact that $\norm{\tilde A - A \oplus 0_d}\leq \epsilon$, so that $\norm{\tilde A_{1,1} - I_n} \leq \epsilon$ and $\norm{\tilde A_{1,2}}\leq \epsilon$. Consequently, as we may also bound $\norm{\tilde B}\leq \norm{B} + \epsilon$,
we deduce that for any $\delta>0$, we can pick $\epsilon\in(0,1/2)$ small enough such that $\norm{P^\intercal\tilde BP - B}\leq \delta$. An identical calculation holds for $\norm{P^\intercal\tilde CP - C}$.
We conclude that for all $\delta>0$, there exist $\bar A,\bar B,\bar C$ of the form
\begin{align*}
\bar A&= \begin{pmatrix}
	I_n &\\
	& \bar A_{2,2}
\end{pmatrix},\quad
\bar B = \begin{pmatrix}
	\bar B_{1,1} & \bar B_{1,2}\\
	\bar B_{1,2}^\intercal & \bar B_{2,2}
\end{pmatrix},\quad
\bar C = \begin{pmatrix}
	\bar C_{1,1} & \bar C_{1,2}\\
	\bar C_{1,2}^\intercal & \bar C_{2,2}
\end{pmatrix}
\end{align*}
such that $\set{\bar A,\bar B,\bar C}$ is SDC, $\norm{A - \bar A},\norm{B-\bar B},\norm{C-\bar C}\leq \delta$, and $\bar A_{2,2}$ is invertible. Then by \cref{prop:sdc_characterization_invertible}, the top-left block of the commutator $[\bar A^{-1}\bar B,\bar A^{-1}\bar C]$ is equal to $0_{n}$. Expanding this top-left block, we deduce
\begin{align}
\label{eq:barAinvbarB_commutator_barAinvbarC}
[\bar B_{1,1},\bar C_{1,1}] = \bar C_{1,2}\bar A_{2,2}^{-1}\bar B_{1,2}^\intercal - \bar B_{1,2}\bar A_{2,2}^{-1}\bar C_{1,2}^\intercal.
\end{align}
Finally, by lower semi-continuity of rank, we have $\rank([\bar B_{1,1},\bar C_{1,1}]) \geq \rank([B,C])$ for all $\delta>0$ small enough. This is a contradiction as the expression on the right of \eqref{eq:barAinvbarB_commutator_barAinvbarC} has rank at most $2d<\rank([B,C])$.
\end{proof}

This same construction can be viewed as an obstruction to generalizations of \cref{thm:rsdc_pair} to symmetric triples with constant $d$.
\begin{corollary}
\label{cor:krsdc_obstruction}
Let $\set{A = I_n, B,C}\subseteq\S^n$. Then $A^{-1}B$ and $A^{-1}C$ are both diagonalizable with real eigenvalues and $\set{A,B,C}$ is not $d$-RSDC for any $d < \rank([B,C])/2$.
\end{corollary}

\begin{remark}\label{rem:existence_of_large_commutator}
Note that for all $n\in\N$, there exist $B,\,C\in\S^{2n}$ such that $\rank([B,C]) = 2n$. For example, set
\begin{align*}
B = \begin{pmatrix}
	I_n &\\&-I_n
\end{pmatrix},\qquad
C = \begin{pmatrix}
	&I_n\\I_n&
\end{pmatrix}.
\end{align*}
Then, $\set{A=I_{2n},B,C}\subseteq\S^{2n}$ is a nonsingular symmetric triple such that $A^{-1}B$ and $A^{-1}C$ are both diagonalizable.
On the other hand, \cref{cor:krsdc_obstruction,thm:commutator_lower_bound_kernel} imply that
\begin{align*}
\set{\begin{pmatrix}
	A&\\&0_{n-1}
\end{pmatrix},\, \begin{pmatrix}
	B&\\&0_{n-1}
\end{pmatrix},\, \begin{pmatrix}
	C&\\&0_{n-1}
\end{pmatrix}}
\end{align*}
is not ASDC and $\set{A,B,C}$ is not $(n-1)$-RSDC.\mathprog{\qed}
\end{remark}

\subsection{Nonsingular septuples}
\label{subsec:nonsing_obstruction}
We may reinterpret \cref{thm:asdc_nonsingular_triple,thm:asdc_characterization_invertible} as saying that if $\cA$ satisfies $\dim(\spann(\cA))\leq 3$ and contains an invertible matrix $S$, then $\cA$ is ASDC if and only if $S^{-1}\cA$ consists of a set of commuting matrices with real eigenvalues.
A natural question to ask is whether the same statement holds without any assumption on the dimension of the span of $\cA$.
\cref{thm:obstruction_m=7_invertible} below presents an obstruction to generalizations in this direction. Specifically, \cref{thm:obstruction_m=7_invertible} constructs a non-ASDC set $\cA = \set{A_1,\dots, A_7}\subseteq\S^6$ where $A_1$ is invertible and $A_1^{-1}\cA$ consists of a set of commuting matrices with real eigenvalues.

The following lemma adapts a technique introduced by \citet{omeara2006approximately} for studying the almost simultaneously diagonalizable via similarity property of subsets of $\C^{n\by n}$.
\begin{lemma}
\label{lem:sdc_real_algebra_dimension}
Let $\cA=\set{A_1,\dots,A_m}\subseteq\S^n$ where $A_1\in\cA$ is invertible.
If $\cA$ is SDC, then $\dim(\R[A_1^{-1}\cA])\leq n$. Here, $\R[A_1^{-1}\cA]$ is the real algebra generated by $A_1^{-1}\cA$.
\end{lemma}
\begin{proof}
Let $P$ denote the invertible matrix furnished by SDC and suppose $P^\intercal A_i P = D_i$. Then,
\begin{align*}
\dim\left(\R\left[A_1^{-1}\cA\right]\right)
&= \dim\left(\R\left[\set{D_1^{-1}D_i:\, i\in[m]}\right]\right)\leq n.\qedhere
\end{align*}
\end{proof}

The following corollary then follows by lower semi-continuity.
\begin{corollary}
\label{cor:asdc_real_algebra_dimension}
Let $\cA=\set{A_1,\dots,A_m}\subseteq\S^n$ where $A_1\in\cA$ is invertible.
If $\cA$ is ASDC, then $\dim(\R[A_1^{-1}\cA])\leq n$. Here, $\R[A_1^{-1}\cA]$ is the real algebra generated by $A_1^{-1}\cA$.
\end{corollary}

\begin{theorem}
\label{thm:obstruction_m=7_invertible}
There exists a set $\cA = \set{A_1,\dots,A_7}\subseteq\S^{6}$ such that $A_1$ is invertible, $A_1^{-1}\cA$ is a set of commuting matrices with real eigenvalues, and $\cA$ is \emph{not} ASDC.\footnote{The analogous example in the complex setting states that there exists a set $\cA=\set{A_1,\dots,A_5}\subseteq\H^4$ such that $A_1$ is invertible, $A_1^{-1}\cA$ is a set of commuting matrices with real eigenvalues, and $\cA$ is not ASDC. See \cref{sec:hermitian_proofs}.}
\end{theorem}
\begin{proof}
Set
\begin{gather*}
A_1 = \begin{smallpmatrix}
        &&&&&1\\
        &&&&1&\\
        &&&1\\
        &&1&\\
        &1&&\\
        1&&&
\end{smallpmatrix},\quad
A_2 = \begin{smallpmatrix}
        0&&&\\
        &0&&\\
        &&0&&&\\
        &&&1&&\\
        &&&&0\\
        &&&&&0
\end{smallpmatrix},\quad
A_3 = \begin{smallpmatrix}
        0&&&\\
        &0&&\\
        &&0&&&\\
        &&&&1\\
        &&&1&\\
        &&&&&0
\end{smallpmatrix},\\
A_4 = \begin{smallpmatrix}
        0&&&\\
        &0&&\\
        &&0&&&\\
        &&&&&1\\
        &&&&0\\
        &&&1
\end{smallpmatrix},\quad
A_5 = \begin{smallpmatrix}
        0&&&\\
        &0&&\\
        &&0&&&\\
        &&&0\\
        &&&&1\\
        &&&&&0
\end{smallpmatrix},\\
A_6 = \begin{smallpmatrix}
        0&&&\\
        &0&&\\
        &&0&&&\\
        &&&0\\
        &&&&&1\\
        &&&&1
\end{smallpmatrix},\quad
A_7 = \begin{smallpmatrix}
        0&&&\\
        &0&&\\
        &&0&&&\\
        &&&0\\
        &&&&0\\
        &&&&&1
\end{smallpmatrix}.
\end{gather*}
Note that $A_1$ is invertible. It is not hard to verify that $A_1^{-1}\cA$ forms a set of commuting matrices with real eigenvalues. On the other hand, note that
\begin{align*}
\R[A_1^{-1}\cA] &= \set{\begin{smallpmatrix}
	a &&&d & f & g \\
	&a&&c & e & f\\
	&&a&b & c & d\\
	&&&a\\
	&&&&a&\\
	&&&&&a
\end{smallpmatrix}:\,a,b,c,d,e,f,g\in\R}
\end{align*}
so that $\dim(\R[A_1^{-1}\cA]) = 7>6 = n$. We deduce from \cref{cor:asdc_real_algebra_dimension} that $\cA$ is \emph{not} ASDC.
\end{proof}




\section{Applications to QCQPs}
\label{sec:application}
In this section, we suggest several applications of diagonalization to optimizing QCQPs.
We begin by proving properties regarding the SDP and SOCP relaxations of diagonal QCQPs with bound constraints. Note that QCQPs with bound constraints are the main problems of interest within spatial branch and bound (BB) schemes for QCQPs. These results give heuristic reasons why one might expect the SOCP relaxation to give strong yet efficiently computable lower bounds within BB schemes. These results serve as additional motivation for the ASDC and $d$-RSDC properties. We then examine these assertions numerically with preliminary computational experiments.



\subsection{The SOCP relaxation of a diagonal QCQP with bound constraints}
\label{sec:applicationtheory}
Consider solving a QCQP over $\R^n$ of the form \eqref{eq:intro_qcqp} within a BB scheme. At each node of the BB tree, we encounter the original QCQP with additional bound constraints,
\begin{align}
\label{eq:qcqp_with_bounds}
\inf_{x\in\R^n}\set{x^\intercal A_1 x + 2b_1^\intercal x + c_1:\, \begin{array}
        {l}
        x^\intercal A_i x + 2b_i^\intercal x + c_i \,\boxempty_i\, 0,\,\forall i\in[2,m]\\
        x\in\cL\\
        x \in[\ell, u]
\end{array}},
\end{align}
and desire to construct and solve strong convex relaxations of \eqref{eq:qcqp_with_bounds}.

One powerful method for constructing such convex relaxations
combines the reformulation-linearization technique with semidefinite programming~\cite{anstreicher2009semidefinite}.
We begin by linearizing \eqref{eq:qcqp_with_bounds} using the variable $Y = xx^\intercal$. Specifically, replace $x^\intercal A_i x$ with $\ip{A_i,Y}$ and include the constraint $Y = xx^\intercal$. Then, relax the pair of constraints $x\in[\ell,u]$ and $Y = xx^\intercal$ to the constraint that $(x,Y)$ belong to the set
\begin{align*}
\cS_\textup{SDP} \coloneqq \set{(x,Y)\in\R^n\times \S^n:\, \begin{array}
        {l}
        Y_{i,j} \geq \ell_jx_i + \ell_i x_j -\ell_j\ell_j,\,\forall i,j\\
        Y_{i,j} \leq u_jx_i + \ell_ix_j - u_j\ell_i,\,\forall i,j\\
        Y_{i,j} \geq u_jx_i +u_ix_j - u_iu_j,\,\forall i,j\\
        Y\succeq xx^\intercal
\end{array}}.
\end{align*}
The SDP+RLT relaxation then reads
\begin{align}
\label{eq:qcqp_with_bounds_relaxation}
\inf_{x\in\R^n,Y\in\S^n}\set{\ip{A_1,Y}+ 2b_1^\intercal x + c_1:\, \begin{array}
        {l}
        \ip{A_i, Y} + 2b_i^\intercal x + c_i \,\boxempty_i\, 0,\,\forall i\in[2,m]\\
        x\in\cL\\
        (x,Y)\in\cS_\textup{SDP}
\end{array}}.
\end{align}

Note that for diagonal QCQPs (i.e., the setting where $A_1,\dots,A_m$ are diagonal) that $\ip{A_i,Y}=\diag(A_i)^\intercal \diag(Y)$ so that the SDP+RLT relaxation does not depend on the off-diagonal entries of $Y$. In particular, we may replace the variable $Y\in\S^n$ with a variable $y\in\R^n$ representing its diagonal entries. Naturally, we then replace the term
the term $\ip{A_i,Y}$ and the constraint $(x,Y)\in\cS_\textup{SDP}$ with the term $\diag(A_i)^\intercal y$ and the constraint
\begin{align*}
(x,y)\in \cS_\textup{SOCP} \coloneqq \set{(x,\diag(Y)):\, (x,Y)\in\cS_\textup{SDP}}.
\end{align*}
The following lemma shows that $\cS_\textup{SOCP}$ is SOC-representable so that the resulting relaxation
is in fact an SOCP. Thus, the SDP+RLT relaxation of a QCQP with bound constraints can be solved substantially faster when $A_1,\dots,A_m$ are diagonal.

In the remainder of this section, let $\circ$ denote the elementwise product between two vectors.
{The following lemma is elementary and is not new (see for example \cite[Section 2]{burer2015gentle}).}
\begin{lemma}
\label{lem:sdp_rlt_socp}
The following identities hold
\begin{align*}
\cS_\textup{SOCP} &\coloneqq \set{(x,\diag(Y)):\, (x,Y)\in\cS_\textup{SDP}}\\
&= \set{(x,y)\in\R^n\times\R^n:\, \begin{array}
        {l}
        x\circ x \leq y \leq (u+\ell)\circ x - u\circ \ell
\end{array}}\\
&= \conv\set{(x,y)\in\R^n\times\R^n:\, \begin{array}
        {l}
        x\in[\ell,u]\\
        x \circ x = y
\end{array}}.
\end{align*}
In particular, $\cS_\textup{SOCP}$ is SOC-representable.
\end{lemma}
\begin{proof}
For notational convenience, let $\cS_1,\,\cS_2,\,\cS_3$ denote the three sets on the right in order.
Note $\cS_\textup{SOCP} = \cS_1$ by definition.
We will show $\cS_1\subseteq \cS_2=\cS_3 \subseteq\cS_1$.

The containment $\cS_1\subseteq\cS_2$ follows by definition: Given $(x,Y)\in\cS_\textup{SDP}$, we have that $\diag(Y)\geq x\circ x$ by the SDP constraint and $\diag(Y) \leq (u+\ell)\circ x - u\circ \ell$ by the RLT constraints.

The identity $\cS_2 =\cS_3$ follows the well-known (and easy to verify) fact that in one-dimension
\begin{align*}
\set{(x_i,y_i)\in\R^2:\, \begin{array}
        {l}
        y_i\geq x_i^2\\
        y_i \leq (u_i+\ell_i)x_i - u_i\ell_i
\end{array}} = \conv\set{(x_i,y_i)\in\R^2:\,\begin{array}{l}
        x_i\in[\ell_i,u_i]\\
        x_i^2 = y_i
\end{array}}
\end{align*}
and the fact that direct products commute with convex hulls.

Finally, suppose $(x,y)\in\R^n\times\R^n$ satisfies $x\in[\ell,u]$ and $y=x\circ x$. Set $Y = xx^\intercal$ so that $\diag(Y) = y$. It is straightforward to show that $(x,Y)\in\cS_\textup{SDP}$ so that $(x,y)\in\cS_1$. Then as $\cS_1$ is convex, we have that $\cS_3 \subseteq\cS_1$.
\end{proof}
The following corollary shows how to construct optimizers of \eqref{eq:qcqp_with_bounds_relaxation} with bounded rank when $A_1,\dots,A_m$ are diagonal.
The bound depends only on $m$ and the complexity of $\cL\cap [\ell,u]$.
This gives a heuristic reason why one would expect the SDP+RLT relaxation (and hence the SOCP+RLT relaxation) of a diagonal QCQP with bound constraints to be stronger than that of a general QCQP with bound constraints, especially when $m$ is small and $\cL$ is simple.

\begin{lemma}
Suppose $A_1,\dots,A_m$ are diagonal and that $\cL\cap[\ell,u]$ can be expressed as the intersection of $[\ell,u]$ with $k$ additional linear (in)equality constraints. If \eqref{eq:qcqp_with_bounds_relaxation} has a solution, then it has a solution $(x^*,Y^*)$ such that
\begin{align*}
\rank\begin{pmatrix}
        Y^* & x^*\\
        x^{*\intercal} & 1
\end{pmatrix}\leq m + k.
\end{align*}
\end{lemma}
\begin{proof}
By assumption there exists an affine function $L:\R^n\to\R^k$ such that
\begin{align*}
[\ell,u]\cap \cL = [\ell,u] \cap \set{x\in\R^n:\, L(x)_i \,\boxempty_i\, 0,\,\forall i\in[k]}
\end{align*}
where $\,\boxempty_i\,\in\set{\leq,=}$.
Define $Q:\R^n\times\R^n\to\R^m$ by
\begin{align*}
Q(x,y)_i = \ip{\diag(A_i),y} + 2b_i^\intercal x + c_i,\,\forall i\in[m]
\end{align*}
and let $\tilde Q: \R^n\times\R^n\to\R^{m+k}$ denote the affine function mapping $(x,y) \mapsto (Q(x,y),L(x))$.

Let $(x^*, Y^*)$ denote an optimizer of \eqref{eq:qcqp_with_bounds_relaxation} and set $y^* = \diag(Y^*)$.
By \cref{lem:sdp_rlt_socp}, there exist points $x^{(i)}\in[\ell,u]$ and convex combination weights $\alpha_i > 0$ such that $(x^*,y^*) = \sum_i \alpha_i (x^{(i)},x^{(i)}\circ x^{(i)})$.
Then, by linearity, we have $\tilde Q(x^*,y^*) = \sum_i \alpha_i \tilde Q(x^{(i)}, x^{(i)}\circ x^{(i)})$.

We claim that $\set{\tilde Q(x^{(i)}, x^{(i)}\circ x^{(i)})}_i$ span an affine subspace of dimension $<m+k$.
Indeed, supposing otherwise, $\tilde Q(x^*,y^*)$ is in the interior of $\conv\set{\tilde Q(x^{(i)}, x^{(i)}\circ x^{(i)})}_i$. Thus, there exists $(x',y')\in\cS_\textup{SOCP}$ achieving $\tilde Q(x',y') = \tilde Q(x^*,y^*) - \epsilon e_1$ for some $\epsilon>0$. This contradicts optimality of $(x^*,Y^*)$.

Applying Carath\'eodory's Theorem, $(x^*,y^*)$ is a convex combination of at most $m + k$ points from $\set{(x^{(i)},x^{(i)}\circ x^{(i)})}$, say $(x^*,y^*) = \sum_{i=1}^{m+k}\alpha_i (x^{(i)},x^{(i)}\circ x^{(i)})$. Then,
\begin{align*}
\sum_{i=1}^{m+k} \alpha_i \begin{pmatrix}
        x^{(i)}x^{(i)\intercal} & x^{(i)}\\
        x^{(i)\intercal} & 1
\end{pmatrix}
\end{align*}
is an optimal solution to \eqref{eq:qcqp_with_bounds_relaxation} with rank $\leq m+k$.
\end{proof}



\subsection{Numerical results}
\label{sec:experiments}
In this subsection, we present preliminary numerical results on diagonalization and the $d$-RSDC property in solving QCQP problems with one quadratic constraint and additional linear constraints.
This restricted class of QCQPs is still NP-hard in general---as mentioned in the introduction, even the problem of minimizing a general quadratic function over the hypercube is NP-hard.
{Problems in this form, and their variants, have received recent interest in the area of optimal portfolio deleveraging \cite{luo2020effective}, the trust region method for solving constrained optimization problems, robust optimization problems under matrix norm or polyhedral uncertainty~\cite{jeyakumar2014trust}, portfolio selection models with complementarity constraints in the presence of linear transaction costs~\cite{xu2023simultaneous,braun2005semidefinite}, and as subproblems in iterative algorithms for more general optimization problems~\cite{huang2016consensus,bienstock2014polynomial}.}


We will consider random instances of the following problem
\begin{align}
\label{pb:qcqp}
\min_{x\in\R^n}\set{x^\intercal A_1 x:\, \begin{array}
        {l}
        x^\intercal  A_2x+2b_2^\intercal x \leq 1\\
        x\in \cL
\end{array}}
\end{align}
where $A_1,\,A_2\in \S^n$, $b_2\in \R^n$, and $\cL\subseteq\R^n$ is a polytope.

\paragraph{Random model.}
We will consider a family of distributions over instances of \eqref{pb:qcqp} parameterized by $n\in\N$ and $k\in\N_0$.
Here, $k$ will parameterize the number of (pairs of) complex eigenvalues of $A_1^{-1}A_2$.
Specifically, given $(n,k)$ such that $2k\leq n$:
\begin{enumerate}
        \item Let $r = n - 2k$.
        \item
        Generate a random orthogonal matrix $V$ by taking $M$ to be a random $n\by n$ matrix with entries i.i.d.\ $N(0,1)$ and then taking $V$ to be a matrix of left singular vectors of $M$.
        Let $\sigma_1,\dots,\sigma_r$ be i.i.d.\ Rademacher random variables.
        Let $\mu_1,\dots,\mu_r$ be i.i.d.\ $N(0,1)$ random variables.
        Let $x_1,\dots,x_k,y_1,\dots,y_k$ be i.i.d.\ $N(0, 1)$ random variables.
        Then, set
        \begin{align*}
        A_1 &= V^\intercal\Diag(\sigma_1,\dots,\sigma_r, F_2,\dots,F_2)V\\
        A_2 &= V^\intercal\Diag(\sigma_1 \mu_1,\dots,\sigma_r \mu_r, T_1,\dots, T_k)V.
        \end{align*}
        Here, $T_i\in\S^2$ is the random matrix $\begin{smallpmatrix}
            x_i & y_i\\
            y_i & -x_i
        \end{smallpmatrix}$.
        \item Let the entries of $b_2$ be i.i.d.\ $N(0,1)$ random variables, and $L=\begin{smallpmatrix}I\\ -I\end{smallpmatrix}N$, where the entries of $N\in \R^{n\times n}$ are i.i.d.\ $N(0,1)$ random variables. This ensures that the set $\cL\coloneqq \{x:Lx\leq 1\}$ is bounded almost surely.
\end{enumerate}
Note that \cref{thm:rsdc_pair} (respectively, \cref{thm:d_rsdc_pair}) implies that $\set{A_1,\,A_2}$ is almost surely 1-RSDC (respectively, $k$-RSDC) in this random model.

\paragraph{Branch and bound methods.} We use BB methods to solve different reformulations of \eqref{pb:qcqp} with and without diagonalization.

We implement two classes of BB methods.
The first class, the SDP-based BB method, uses a \emph{simplified} SDP+RLT relaxation for computing a lower bound at each node. Specifically, we lower bound the value of \eqref{pb:qcqp} with the additional box constraint $x\in[\ell,u]$ by
\begin{equation}
\label{pb:sdprf}
\min_{x\in\R^n,\, Y\in\S^n}\set{\ip{A_1,Y}:\, \begin{array}
        {l}
        \ip{A_i, Y} + 2b_2^\intercal x - 1 \leq 0\\
        x\in\cL\\
        x\in[\ell,u]\\
        \diag(Y) \leq (u + \ell) \circ x - u \circ \ell\\
        Y\succeq xx^\top
\end{array}}.
\end{equation}
At the root node, we set $[\ell,u]$ to be coordinate-wise lower and upper bounds on $\cL$.
Note that in contrast to the \emph{full} SDP+RLT relaxation (see \eqref{eq:qcqp_with_bounds_relaxation}), we only impose RLT constraints on the diagonal entries of $Y$ to strike a balance between bound quality and computational cost.
This method then applies a spatial BB rule for each coordinate $x_i$
and updates the values of $[\ell,u]$.

The second class, the SOCP-based BB methods,
first diagonalize \eqref{pb:qcqp} before applying a BB scheme.
The method of diagonalization differs across the different SOCP-based BB methods but the BB part is the same.
Suppose we have already diagonalized \eqref{pb:qcqp} so that $A_i$ is a diagonal matrix for each $i = 1,2$. Write
$A_i=\Diag(a_i^+) + \Diag(a_i^-)$, where $a_i^+,\,a_i^-\in\R^n$ are nonnegative and nonpositive respectively.
Let $I = \supp(a_1^-)\cup \supp(a_2^-)$.
The SOCP-based BB method uses the SOCP+RLT relaxation for computing a lower bound at each node. Specifically we lower bound the value of \eqref{pb:qcqp} (assuming that the $A_i$s are diagonal) with the additional box constraint $x\in[\ell,u]$ by
\begin{equation}
\label{pb:socprf}
\min_{x\in\R^n,\, y\in\R^{\abs{I}}}\set{x^\intercal\Diag(a_1^+) x + (a_1^-)^\intercal y:\, \begin{array}
        {l}
        x^\intercal \Diag(a_2^+) x + (a_2^-)^\intercal y + 2b_2^\intercal x  -1 \leq 0\\
        x\in\cL\\
        x\in[\ell,u]\\
        x_{j_i}^2 \leq y_i \leq (u_{j_i} + \ell_{j_i}) x_{j_i} - u_{j_i}\ell_{j_i},\,\forall j_i\in I
\end{array}}.
\end{equation}
Again, at the root node, we set $[\ell,u]$ to be coordinate-wise lower and upper bounds on $\cL$.
This method then applies a spatial BB rule for each coordinate $x_{j_i}$ such that $j_i\in I$ and updates $[\ell,u]$.

In both methods, we use a successive convex approximation \cite{marks1978general}, which linearizes nonconvex terms in the quadratic objective and constraint, to attempt to construct feasible solutions and good upper bounds.

In more detail, we implemented the following five BB methods for solving instances of \eqref{pb:qcqp}.
\begin{itemize}
\item \texttt{SDPBB} solves \eqref{pb:qcqp} directly using the SDP-based BB method.
\item \texttt{SDCBB} is a solution method which can only be applied when $\set{A_1,A_2}$ is already SDC. In this case (letting $P$ denote the corresponding invertible matrix), \texttt{SDCBB} reformulates \eqref{pb:qcqp} as
\begin{align}
\label{pb:sdcqcqp}
\min_{x\in\R^n}\set{x^\intercal (P^\intercal A_1P) x :\, \begin{array}        {l}
        x^\intercal  (P^\intercal A_2P)x+2(P^\intercal b_2)^\intercal x \leq 1\\
        LPx \leq 1,
\end{array}}
\end{align}
and solves this reformulation using the SOCP-based BB method.
\item \texttt{1-RSDCBB} applies \cref{alg:1_rsdc} to construct an SDC pair
$\set{\tilde A_1,\,\tilde A_2}\in\S^{n+1}$ whose top-left $n\times n$ principal submatrices are $A_1$ and $A_2$, respectively.
Let $P\in\R^{(n+1)\by (n+1)}$ denote the invertible matrix furnished by the SDC property of $\set{\tilde A_1,\,\tilde A_2}$.
Also, set $\tilde b_2 = (b_2^\intercal, 0)^\intercal$ and $\tilde L = (L, 0_{m,1})$.
Then, \texttt{1-RSDCBB} reformulates \eqref{pb:qcqp} as
\begin{align}
\label{pb:1sdcqcqp}
\inf_{w\in\R^{n+1}}\set{w^\intercal (P^\intercal\tilde  A_1P) w :\, \begin{array}
        {l}
        w^\intercal (P^\intercal \tilde  A_2 P) w+2(P^\intercal\tilde b_2)^\intercal w \leq 1\\
        (\tilde LP)w \leq 1\\
        (Pw)_{n+1} = 0
\end{array}}
\end{align}
and solves this reformulation using the SOCP-based BB method. Note that for this reformulation, $\cL$ is the set of $w\in\R^{n+1}$ satisfying both the linear inequality and linear equality constraints.
%
%
%
%
\item \texttt{k-RSDCBB} applies \cref{alg:d_rsdc} with $d=k$ to construct an SDC pair
$\set{\tilde A_1,\,\tilde A_2}\in\S^{n+k}$ whose top-left $n\times n$ principal submatrices are $A_1$ and $A_2$, respectively.
Let $P\in\R^{(n+k)\by (n+k)}$ denote the invertible matrix furnished by the SDC property of $\set{\tilde A_1,\,\tilde A_2}$.
Also, set $\tilde b_2= (b_2^\intercal, 0_{k,1})^\intercal$ and $\tilde L = (L, 0_{m,k})$.
Then, \texttt{k-RSDCBB} reformulates \eqref{pb:qcqp} as
\begin{align}
\label{pb:ksdcqcqp}
\inf_{w\in\R^{n+1}}\set{w^\intercal (P^\intercal\tilde  A_1P) w :\, \begin{array}
        {l}
        w^\intercal (P^\intercal \tilde  A_2 P) w+2(P^\intercal\tilde b_2)^\intercal w \leq 1\\
        (\tilde LP)w \leq 1\\
        (Pw)_{n+1} = (Pw)_{n+2} = \cdots = (Pw)_{n+k}= 0
\end{array}}
\end{align}
and solves this reformulation using the SOCP-based BB method. Note that for this reformulation, $\cL$ is the set of $w\in\R^{n+k}$ satisfying both the linear inequality and linear equality constraints.
\item \texttt{eigBB} first performs an eigenvalue decomposition on $A_1$ to write $D_1 = P_1^\intercal A_1 P_1$, where $D_1$ is a diagonal matrix. Then, it performs a second eigenvalue decomposition to write $D_2= P_2^\intercal(P_1^\intercal A_2 P_1)P_2$, where $D_2$ is a diagonal matrix. Finally, \texttt{eigBB} reformulates \eqref{pb:qcqp} as
\begin{align}
\label{pb:eigqcqp}
\inf_{y,z\in\R^n}\set{y^\intercal D_1 y :\, \begin{array}
        {l}
        z^\intercal D_2 z + 2(P_1^\intercal b_2)^\intercal y + c_2 \leq 1\\
        (LP_1) y \leq 1\\
        y = P_2 z
\end{array}}
\end{align}
and solves this reformulation using the SOCP-based BB method. Note that for this reformulation, $\cL$ is the set of $(y,z)\in\R^{n}\times\R^n$ satisfying both the linear inequality and linear equality constraints.
\end{itemize}

All experiments are implemented using MATLAB R2021a
on a PC running Windows 10 Intel(R) Core(TM) i9-10900KF CPU (3.70GHz)
and 64GB RAM.
All the SDP and SOCP problems in the BB methods are solved by the commercial solver MOSEK \cite{mosek} through its Matlab interface.

\begin{remark}
\texttt{SDCBB}, \texttt{1-RSDCBB}, \texttt{k-RSDCBB}, and \texttt{eigBB} can be thought of as different reformulations within a parameterized family of reformulations of \eqref{pb:qcqp}. Specifically, these four algorithms reformulate \eqref{pb:qcqp} as diagonal QCQPs with $n$, $n+1$, $n+k$, and $2n$ variables respectively.

{In fact, the reformulation \eqref{pb:eigqcqp} used in \texttt{eigBB} is \emph{exactly} the reformulation that would arise in \eqref{pb:ksdcqcqp} when using the naive $n$-RSDC construction (see \cref{rem:naive_diagionalization,lem:naive_lifting}). Specifically,
if $D_1 = U_1^\intercal A_1 U_1$ and $D_2 = U_2^\intercal A_2 U_2$ where $U_i$ are orthogonal and $D_i$ are diagonal, then $\set{A_1,A_2}$ is $n$-RSDC via the lifting
\begin{align*}
    \bar A_1 = P^{-\intercal} \begin{pmatrix}
    D_1\\
    & 0
    \end{pmatrix} P^{-1}
    \qquad \text{and}\qquad
    \bar A_2 = P^{-\intercal}\begin{pmatrix}
        0\\
        & D_2
        \end{pmatrix}P^{-1},
\end{align*}
where
\begin{align*}
    P^{-\intercal} = \begin{pmatrix}
    U_1 & U_2\\
    0 & I
    \end{pmatrix}\qquad\text{and}\qquad
    P = \begin{pmatrix}
    U_1 & 0\\
    -U_2^\intercal U_1 & I
    \end{pmatrix}.
\end{align*}
Also, set $\tilde b_2= (b_2^\intercal, 0_{n,1})^\intercal$ and $\tilde L = (L, 0_{m,n})$.
Then, the $n$-RSDC reformulation reads as
\begin{align*}
    &\inf_{w=(y,z)\in\R^{2n}}\set{w^\intercal \begin{pmatrix}
        D_1\\
        & 0
        \end{pmatrix} w :\, \begin{array}
        {l}
        w^\intercal \begin{pmatrix}
            0\\
            & D_2
            \end{pmatrix} w+2(P^\intercal\tilde b_2)^\intercal w \leq 1\\
        (\tilde LP)w \leq 1\\
        (Pw)_{n+1} = (Pw)_{n+2} = \cdots = (Pw)_{2n}= 0
\end{array}}\\
&\qquad = \inf_{(y,z)\in\R^{2n}}\set{y^\intercal D_1 y :\, \begin{array}
    {l}
    z^\intercal D_2 z +2b_2^\intercal (U_1y) \leq 1\\
    L(U_1y) \leq 1\\
    y = U_1^\intercal U_2 z
\end{array}}.
\end{align*}
This is exactly the reformulation in \eqref{pb:eigqcqp}.}
\mathprog{\qed}
\end{remark}

\paragraph{Experiment setup.}
We tested the solution methods on random instances for various settings of $(n,k)$.
For $k = 0$, i.e., the case where $A_1$ and $A_2$ are guaranteed to be SDC, we compared \texttt{SDPBB}, \texttt{SDCBB}, and \texttt{eigBB}. 
For $k>0$, we compared \texttt{SDPBB}, \texttt{1-RSDCBB}, \texttt{k-RSDCBB}, and \texttt{eigBB}. For each $(n,k)$, we generated 5 random problems and used the command \texttt{boxplot} in MATLAB to present the statistics.
Each procedure was terminated when the CPU time reached 1800 seconds or when the relative gap (between the objective value of the current solution and the best lower bound) fell below the default tolerance threshold, $10^{-4}$. In all of our figures and tables, we set
\[\mathrm{Gap}=\frac{v_{\rm best}-v_0}{|v_{\rm best}|}\times 100,\]
where $v_0$ is the initial lower bound computed from the corresponding convex relaxation, and $v_{\rm best}$ is the best upper bound computed within the BB method.

\paragraph{Comparison for the SDC case.}
\begin{figure}[H]
\minipage{0.33\textwidth}
\includegraphics[width=\textwidth]{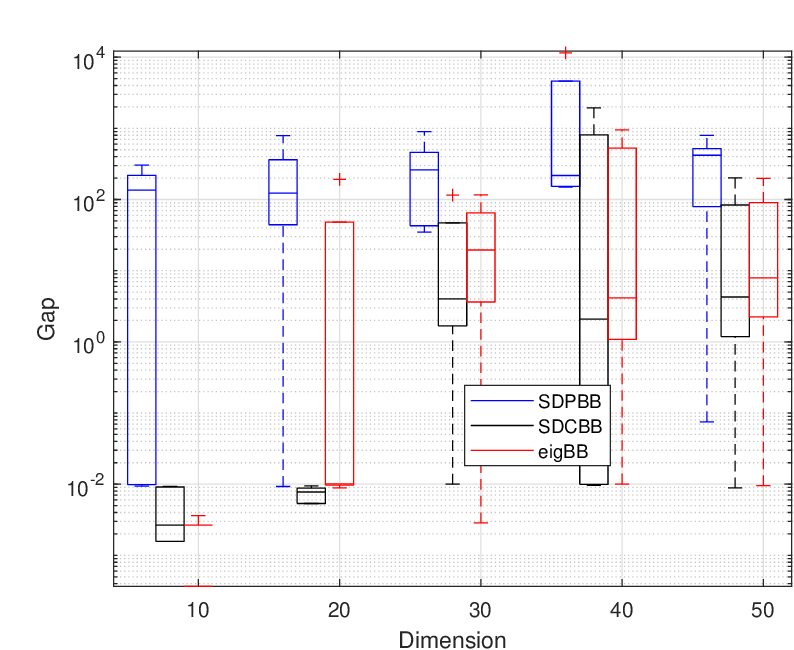}
\endminipage
\minipage{0.33\textwidth}
\center
\includegraphics[width=\textwidth]{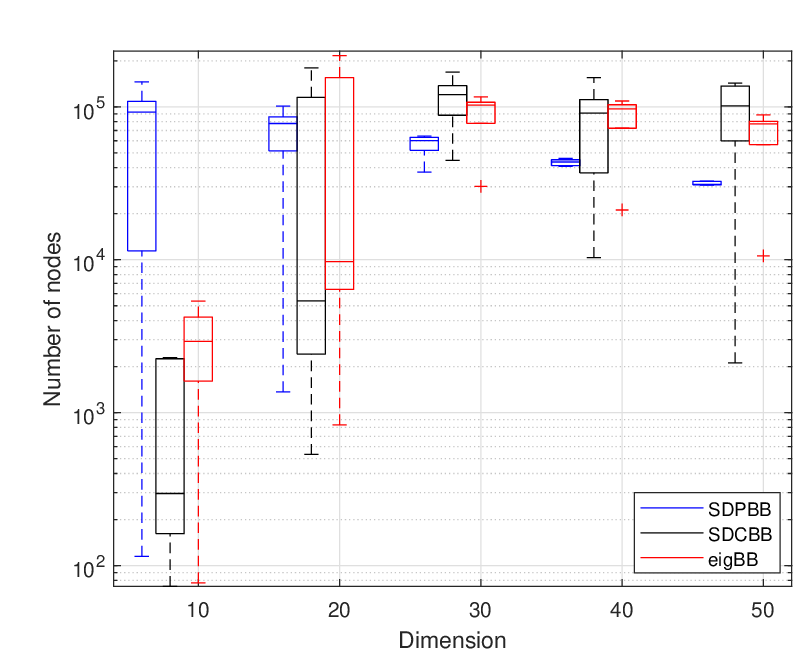}
\endminipage
\minipage{0.33\textwidth}
\raggedright
\includegraphics[width=\textwidth]{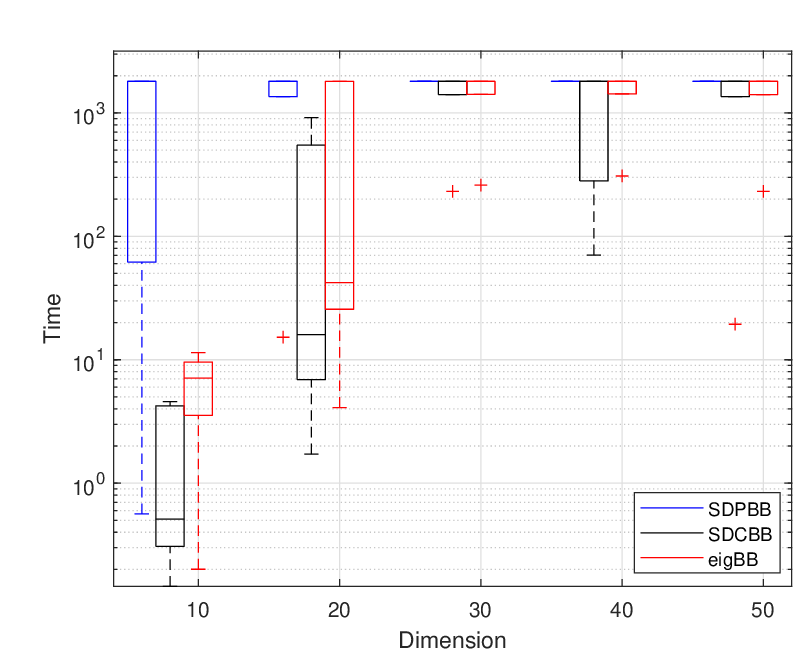}
\endminipage
\caption{Comparison of \texttt{SDPBB}, \texttt{SDCBB} and \texttt{eigBB} for the case with $k=0$.}
\label{fig:keq0}
\end{figure}

We first test instances where $\set{A_1,A_2}$ is SDC, i.e., $k=0$, for $n=10,\,20,\,30,\,40,\,50$.
The results on CPU time, relative gap, and number of explored nodes in the search tree are reported in \cref{fig:keq0}. \cref{fig:keq0} shows us that \texttt{SDCBB} performs the best in general, i.e., \texttt{SDCBB} achieves the lowest relative gap and smallest CPU time across all tested values of $n$.
Both of the SOCP-based methods are much more efficient than
\texttt{SDPBB}. In fact, \texttt{SDPBB} fails to solve any of the instances to relative gap $10^{-4}$ when $n>20$ and fails on four of the five instances with $n=20$. Moreover, for $n=10$, we observe that the SDP-based BB method explores more nodes than either of the SOCP-based BB methods, even though the SDP lower bounds are computationally more expensive than the SOCP lower bounds. Indeed, we will see soon that the SOCP relaxation experimentally yields tighter lower bounds (resulting in fewer search tree nodes) than the SDP relaxation.
We also observe that \texttt{eigBB} is comparable to but slightly less efficient than \texttt{SDCBB}.
Specifically, we note that \texttt{SDCBB} and \texttt{eigBB} explore similar numbers of nodes but that \texttt{SDCBB} does so in comparable or less time.

\begin{figure}[H]
\center
\minipage{0.48\textwidth}
\includegraphics[width=\textwidth]{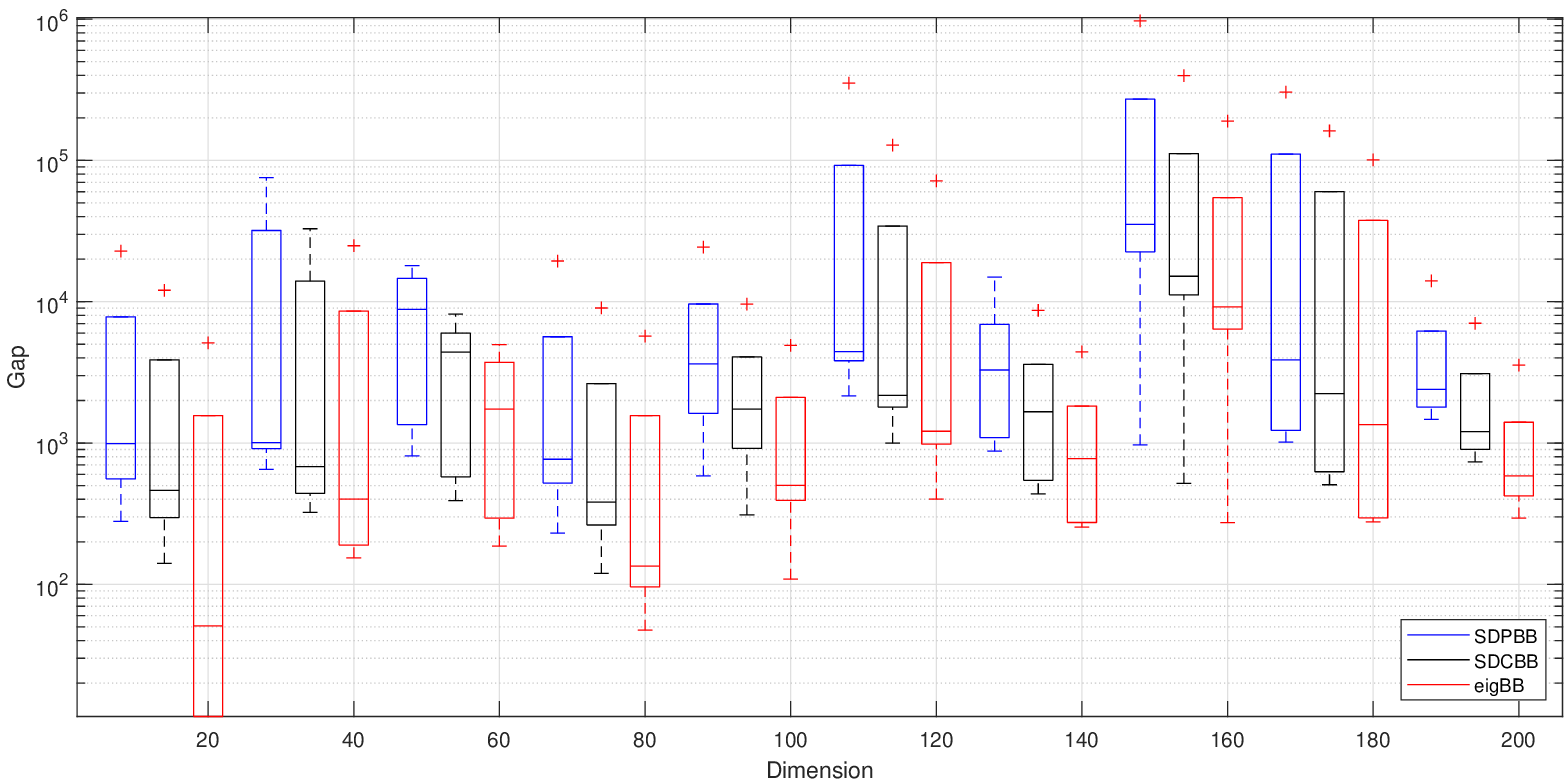}
\endminipage
\quad
\minipage{0.48\textwidth}
\raggedright
\includegraphics[width=\textwidth]{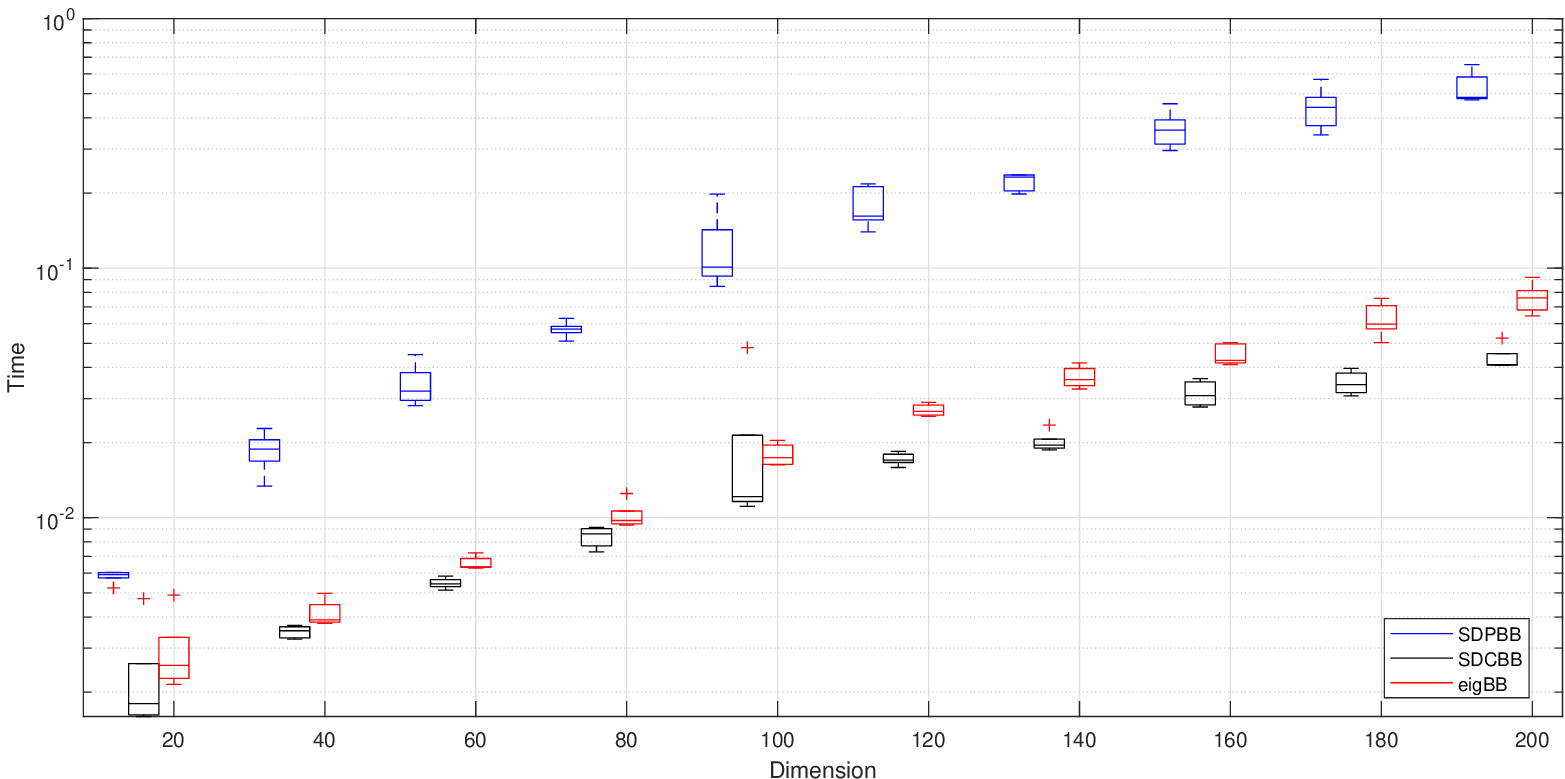}
\endminipage
\caption{Comparison of initial bound and time between SDP and SOCP relaxations for instances of different dimensions.}
\label{fig:sdpsocp}
\end{figure}
To further understand the performance between the SDP-based and SOCP-based BB methods, we compare initial bound quality and CPU time for \texttt{SDPBB}, \texttt{SDCBB} and \texttt{eigBB} in the case $k=0$.
For \cref{fig:sdpsocp} only, define
\[\mathrm{Gap}=\frac{v_{\texttt{SDCBB}}-v_0}{|v_{\texttt{SDCBB}}|}\times100,\]
where $v_0$ is the initial lower bound computed by \texttt{SDPBB}, \texttt{SDCBB} or \texttt{eigBB} and $v_{\texttt{SDCBB}}$ is the best upper bound computed by \texttt{SDCBB} after 1800 seconds.
\cref{fig:sdpsocp} shows that both SOCP relaxations are faster to compute than the SDP relaxation, as expected. More interestingly, both SOCP relaxations provide a \emph{better} initial lower bound as can be seen by the fact that the gap is significantly smaller for the SOCP relaxations than it is for the SDP relaxation.
See \cref{sec:applicationtheory} for heuristic explanations why we would expect this to hold.
Both observations in \cref{fig:sdpsocp} suggest that diagonalization can be used within branch and bound schemes to solve QCQPs more efficiently.

Comparing \texttt{SDCBB} and \texttt{eigBB} in \cref{fig:sdpsocp}, we see that \texttt{eigBB} generally produces tighter lower bounds but \texttt{SDCBB} needs less computation time to solve its relaxation. This
parallels the observation in \cref{fig:keq0} that \texttt{SDCBB} is capable of exploring more nodes than \texttt{eigBB} in similar amounts of time.
We believe that \texttt{SDCBB} solves its relaxation faster simply because its diagonal reformulation is smaller.
Indeed, \texttt{SDCBB} solves an SOCP \eqref{pb:sdcqcqp} with $\left(n + \abs{\supp(a_1^-)\cup \supp(a_2^-)}\right)$-many variables
while \texttt{eigBB} solves an SOCP \eqref{pb:eigqcqp} with roughly twice as many variables: $\left(2n + \abs{\supp(a_1^-)} + \abs{\supp(a_2^-)}\right)$-many variables.

\paragraph{Comparison for the non-SDC case.}
\begin{figure}[H]
\center
\includegraphics[width=0.48\textwidth]{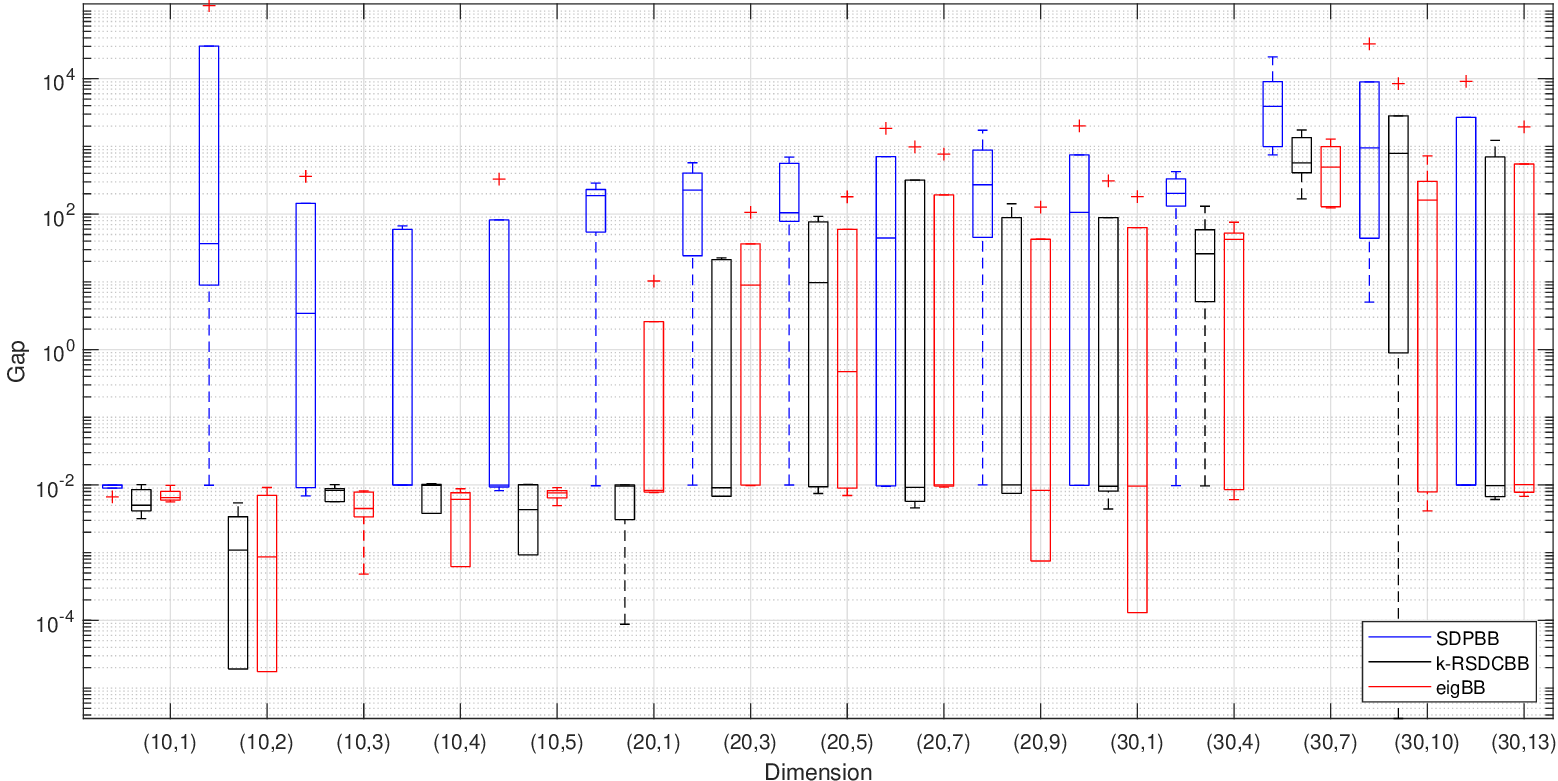}\quad
\includegraphics[width=0.48\textwidth]{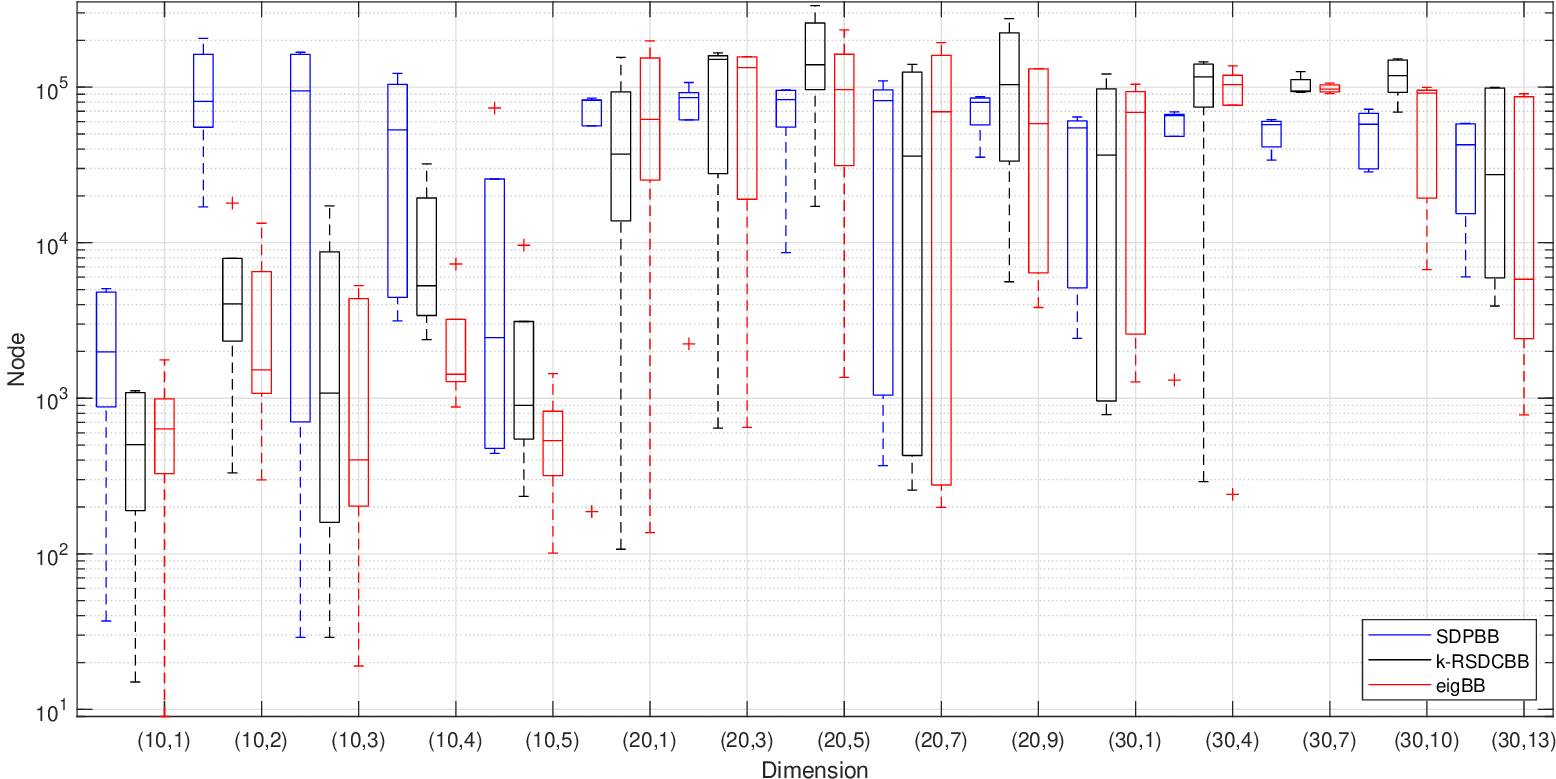}
\\
\includegraphics[width=0.48\textwidth]{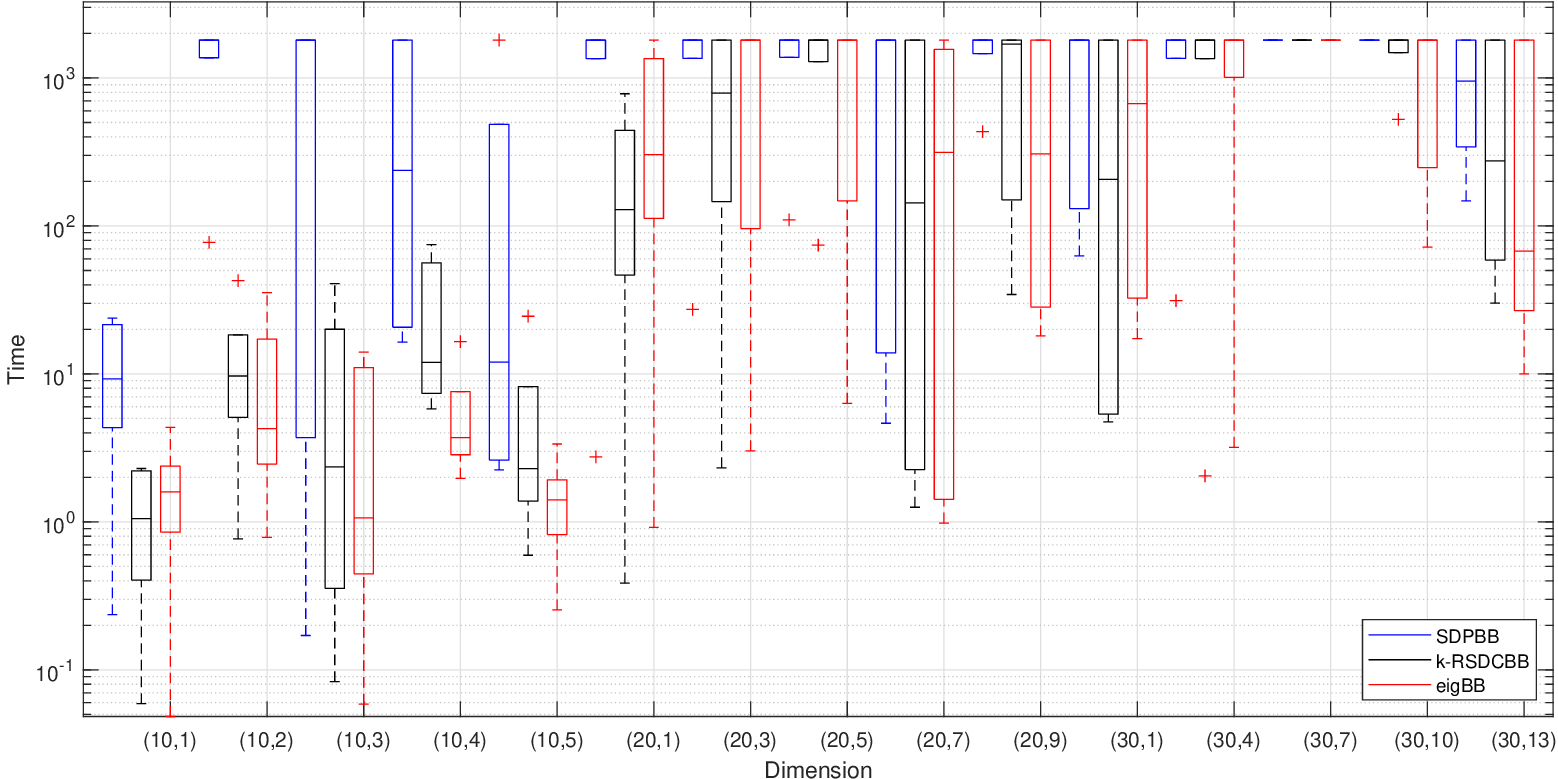}
\caption{Comparison of \texttt{SDPBB}, \texttt{k-RSDCBB} and \texttt{eigBB} for non-SDC instances.}
\label{fig:kge1}
\end{figure}
We now consider the case where $\set{A_1,A_2}$ is not SDC, i.e., $k>0$.
We tested \texttt{SDPBB}, \texttt{1-RSDCBB}, \texttt{k-RSDCBB} and \texttt{eigBB} for $n = 10,\,20,\,30$ and $k = 1,\, 1 + \tfrac{n}{10},\, 1 + \tfrac{2n}{10},\, 1 + \tfrac{3n}{10},\, 1 + \tfrac{4n}{10}$.
The results on CPU time, relative gap, and number of explored nodes in the search tree
for \texttt{SDPBB}, \texttt{k-RSDCBB} and \texttt{eigBB}
are reported in \cref{fig:kge1}.
\cref{fig:kge1} indicates that both \texttt{k-RSDCBB} and \texttt{eigBB} largely outperform \texttt{SDPBB}.
Indeed, \texttt{SDPBB} cannot solve most instances in the time limit, evidenced from the left plot in \cref{fig:kge1}, while \texttt{k-RSDCBB} and \texttt{eigBB} can solve more instances and have  lower relative gaps for unsolved instances in the time limit. In general, \texttt{k-RSDCBB} and \texttt{eigBB} are comparable and {\texttt{eigBB}  is slightly better than \texttt{k-RSDCBB}.}

It remains to comment on the numerical performance of \texttt{1-RSDCBB}.
Experimentally, we observed that the 1-RSDC construction (\cref{alg:1_rsdc}) yields very large condition numbers for the $P$ matrices in \eqref{pb:1sdcqcqp} (e.g., larger than \texttt{1e6}). This leads to inaccurate solutions or numerical failures in MOSEK when solving the SOCP+RLT relaxation, especially for $k\geq 5$. Note also that \texttt{1-RSDCBB} coincides with \texttt{k-RSDCBB} for $k=1$. Thus, we compare the three SOCP-based BB methods, \texttt{1-RSDCBB}, \texttt{k-RSDCBB}, and \texttt{eigBB}, for values of $1<k<5$ in \cref{tab:1}.

\begin{table}
\center
\mathprog{\setlength\tabcolsep{2pt}}
\resizebox{\linewidth}{!}{
\begin{tabular}{c|ccc|ccc|ccc|ccc|cc} \toprule
\multirow{2}{*}{$(n,k)$}
&\multicolumn{3}{c|}{\texttt{1-RSDCBB}}&\multicolumn{3}{c|}{\texttt{k-RSDCBB}}&
\multicolumn{3}{c|}{\texttt{eigBB}}&
\multicolumn{2}{c}{cond num}\\
 &time& node & gap (\%) &time& node & gap (\%)& time& node & gap (\%)& \texttt{1-RSDC}&\texttt{2-RSDC}\\
\hline\midrule
(10,2)&                   5.73& 2830& 0.00&       9.68& 4582& 0.00&        \textbf{3.02}& 1335& 0.01&       5.14e+01&3.79e+00     \\
(10,2)&         \textbf{27.87}& 11462& 0.00&     42.68& 17944& 0.00&      35.49& 13335& 0.00&     2.78e+01&5.09e+00     \\
(10,2)&        30.82& 13764& 0.00&     \textbf{6.52}& 2995& 0.00&       11.11& 4234& 0.00&       4.70e+02&4.54e+00     \\
(10,2)&            2.55& 972& 0.00&        \textbf{0.77}& 331& 0.00&     0.79& 299& 0.00&        4.22e+02&2.25e+00     \\
(10,2)&            15.84& 4423& 0.00&      10.23& 4045& 0.01&       \textbf{4.27}& 1521& 0.01&       1.59e+02&2.37e+00     \\
(10,3)&           2.71& 1264& 0.01&       \textbf{0.45}& 203& 0.01&        0.57& 264& 0.01&         2.29e+02&2.72e+00     \\
(10,3)&           16.67& 6848& 0.00&      \textbf{13.15}& 5899& 0.00&       14.04& 5295& 0.00&      1.89e+02&4.87e+00     \\
(10,3)&            19.55& 8176& 0.01&      40.75& 17257& 0.01&      \textbf{10.04}& 4056& 0.00&      5.36e+01&3.42e+00     \\
(10,3)&            1.91& 789& 0.00&        0.08& 29& 0.01&         \textbf{0.06}& 19& 0.00&          1.68e+03&2.24e+00     \\
(10,3)&             54.33& 20000& 0.01&     2.36& 1080& 0.01&     \textbf{1.06}& 402& 0.01&        2.28e+03&1.44e+01     \\
(10,4)&             259.95& 69602& 0.01&    11.95& 5289& 0.01&       \textbf{1.97}& 879& 0.01&        4.37e+03&3.31e+00     \\
(10,4)&            1800.05& 147765& 23.56&         7.93& 3746& 0.00&        \textbf{3.13}& 1414& 0.00&       1.17e+04&8.04e+00     \\
(10,4)&       46.22& 19976& 0.01&     74.85& 32075& 0.01&     \textbf{16.55}& 7295& 0.01&       3.63e+02&7.57e+01     \\
(10,4)&         1800.08& 130796& 158.72&        5.81& 2381& 0.01&     \textbf{4.61}& 1858& 0.00&       2.10e+04&6.55e+00     \\
(10,4)&      \cbl   77.54& 16565& \cbl 1.61&     50.20& 15150& 0.01&     \textbf{3.71}& 1427& 0.01&        2.73e+06&2.49e+01     \\
 \midrule
 (20,3)&           1800.07& 120343& 169.36&        193.58& 36815& 0.01&     \textbf{126.92}& 25152& 0.01&    3.64e+05&3.20e+01     \\
(20,3)&         1800.05& 107481& 216.00&        1800.05& 150828& 22.65&          1800.04& 156611& \textbf{8.94}&          8.99e+03&1.44e+01     \\
(20,3)&           1800.05& 162012& 49.30&         \textbf{790.61}& 166079& 0.00&    1800.05& 156891& 13.02&         2.35e+02&9.71e+00     \\
(20,3)&   1800.07& 115944& 331.43&        1800.07& 156808& \textbf{20.78}&          1800.07& 133551& 106.07&        1.06e+03&3.76e+00     \\
(20,3)&       6.74& 1866& 0.01&       \textbf{2.32}& 643& 0.01&        3.02& 650& 0.01&         6.00e+02&1.01e+01     \\
\midrule
(30,4)&           1800.08& 102100& 100.73&        1800.08& 116527& \textbf{25.97}&          1800.07& 103676& 42.39&         2.85e+03&5.81e+00     \\
(30,4)&          1800.06& 117590& 205.94&        1800.05& 138837& \textbf{34.78}&          1800.07& 113383& 44.58&         1.26e+04&6.50e+00     \\
(30,4)&           1800.07& 95644& 838.24&         1800.04& 145488& 6.80&   \textbf{1345.35}& 136907& 0.01&          2.27e+05&1.41e+01     \\
(30,4)&           1800.03& 110507& 1463.26&       1800.08& 99003& 130.64&          1800.08& 101895& \textbf{75.89}&         2.57e+05&6.43e+00     \\
(30,4)&     \cbl   66.06& 5380& \cbl 0.02&      \textbf{2.05}& 291& 0.01&        3.19& 241& 0.01&         1.15e+05&1.04e+01     \\
\bottomrule
\end{tabular}}
\caption{Comparison of different SOCP-based BB methods for $1<k<5$. In each row, the solution method with the lowest solution time is highlighted. For instances where all three methods time out (1800 seconds) before reaching optimality, the solution method with the lowest objective value is highlighted. Two outliers are highlighted in blue.}
\label{tab:1}
\end{table}

One may observe that, for $n=10$,  \texttt{1-RSDCBB} seems to perform worse (compared to \texttt{1-RSDCBB} and \texttt{eigBB}) as $k$ increases.
This trend can be explained by observing that the condition numbers of the $P$ matrices for \eqref{pb:1sdcqcqp} are likely to ``blow up'' as $k$ increases (see the two rightmost columns of \cref{tab:1}). In particular, we observed that the lower and upper bounds that we computed for the decision variables (i.e., the values of $\ell$ and $u$ at the root node) in \texttt{k-RSDCBB} and \texttt{eigBB} were relatively small intervals, while the corresponding bounds for those in \texttt{1-RSDCBB} were often much larger (e.g., on the order of 1000 times larger for $k = 3$).
Comparing the rightmost two columns of \cref{tab:1}, we see that the condition numbers of the invertible matrices $P$ that we construct are often much smaller for \texttt{k-RSDCBB} than for \texttt{1-RSDCBB}, especially as $k$ gets larger. We believe this explains why \texttt{k-RSDCBB} generally outperforms \texttt{1-RSDCBB} for larger values of the parameter $k$.
Finally, we observe that for the last instances in (10,4) and (30,4), \texttt{1-RSDCBB} returned solutions without reaching the prescribed gap or CPU times.
We believe that this was caused in both instances by numerical inaccuracies within the interior point solves in MOSEK due to the large condition numbers, i.e., \texttt{2.73e6} and \texttt{1.15e5}.
For $k\geq 5$, the condition number of  \texttt{1-RSDCBB} is even worse and  \texttt{1-RSDCBB} fails for almost all instances (not reported here).

\section*{Acknowledgments}
The authors would like to thank Kevin Pratt for offering an idea which helped to simplify the proof of \cref{lem:toeplitz_mapping}. The authors would like to thank the review team for helpful suggestions to improve the quality of this paper. The second author is supported in part by NSFC 12171100 and the Major Program of NFSC (72394360,72394364).


\begin{thebibliography}{47}
\providecommand{\natexlab}[1]{#1}
\providecommand{\url}[1]{\texttt{#1}}
\expandafter\ifx\csname urlstyle\endcsname\relax
  \providecommand{\doi}[1]{doi: #1}\else
  \providecommand{\doi}{doi: \begingroup \urlstyle{rm}\Url}\fi

\bibitem[Anstreicher(2009)]{anstreicher2009semidefinite}
K.~M. Anstreicher.
\newblock Semidefinite programming versus the reformulation-linearization
  technique for nonconvex quadratically constrained quadratic programming.
\newblock \emph{{J.\ Global Optim.}}, 43\penalty0 (2):\penalty0 471--484, 2009.

\bibitem[Audet et~al.(2000)Audet, Hansen, Jaumard, and Savard]{audet2000branch}
C.~Audet, P.~Hansen, B.~Jaumard, and G.~Savard.
\newblock A branch and cut algorithm for nonconvex quadratically constrained
  quadratic programming.
\newblock \emph{Mathematical Programming}, 87\penalty0 (1):\penalty0 131--152,
  2000.

\bibitem[Bao et~al.(2011)Bao, Sahinidis, and Tawarmalani]{bao2011semidefinite}
X.~Bao, N.~V. Sahinidis, and M.~Tawarmalani.
\newblock Semidefinite relaxations for quadratically constrained quadratic
  programming: {A} review and comparisons.
\newblock \emph{{Math.\ Program.}}, 129:\penalty0 129, 2011.

\bibitem[Ben-Tal and {den Hertog}(2014)]{ben2014hidden}
A.~Ben-Tal and D.~{den Hertog}.
\newblock Hidden conic quadratic representation of some nonconvex quadratic
  optimization problems.
\newblock \emph{{Math.\ Program.}}, 143:\penalty0 1--29, 2014.

\bibitem[Ben-Tal and Teboulle(1996)]{ben1996hidden}
A.~Ben-Tal and M.~Teboulle.
\newblock Hidden convexity in some nonconvex quadratically constrained
  quadratic programming.
\newblock \emph{{Math.\ Program.}}, 72:\penalty0 51--63, 1996.

\bibitem[Bienstock and Michalka(2014)]{bienstock2014polynomial}
D.~Bienstock and A.~Michalka.
\newblock Polynomial solvability of variants of the trust-region subproblem.
\newblock In \emph{Proceedings of the Twenty-Fifth Annual ACM-SIAM Symposium on
  Discrete Algorithms}, pages 380--390, 2014.

\bibitem[Billionnet et~al.(2016)Billionnet, Elloumi, and
  Lambert]{billionnet2016exact}
A.~Billionnet, S.~Elloumi, and A.~Lambert.
\newblock Exact quadratic convex reformulations of mixed-integer quadratically
  constrained problems.
\newblock \emph{Mathematical Programming}, 158\penalty0 (1):\penalty0 235--266,
  2016.

\bibitem[Blekherman et~al.(2024)Blekherman, Dey, and
  Sun]{blekherman2024aggregations}
G.~Blekherman, S.~S. Dey, and S.~Sun.
\newblock Aggregations of quadratic inequalities and hidden hyperplane
  convexity.
\newblock \emph{{SIAM J.\ Optim.}}, 34\penalty0 (1):\penalty0 98--126, 2024.

\bibitem[Braun and Mitchell(2005)]{braun2005semidefinite}
S.~Braun and J.~E. Mitchell.
\newblock A semidefinite programming heuristic for quadratic programming
  problems with complementarity constraints.
\newblock \emph{Comp. Optim. and Appl.}, 31\penalty0 (1):\penalty0 5--29, 2005.

\bibitem[Burer and Ye(2019)]{burer2019exact}
S.~Burer and Y.~Ye.
\newblock Exact semidefinite formulations for a class of (random and
  non-random) nonconvex quadratic programs.
\newblock \emph{{Math.\ Program.}}, pages 1--17, 2019.

\bibitem[Burer(2015)]{burer2015gentle}
Samuel Burer.
\newblock A gentle, geometric introduction to copositive optimization.
\newblock \emph{{Math.\ Program.}}, 151:\penalty0 89--116, 2015.

\bibitem[Bustamante et~al.(2020)Bustamante, Mellon, and
  Velasco]{bustamante2020solving}
M.~D. Bustamante, P.~Mellon, and M.~V. Velasco.
\newblock Solving the problem of simultaneous diagonalization of complex
  symmetric matrices via congruence.
\newblock \emph{SIAM Journal on Matrix Analysis and Applications}, 41\penalty0
  (4):\penalty0 1616--1629, 2020.

\bibitem[Chen et~al.(2017)Chen, Atamt{\"u}rk, and Oren]{chen2017spatial}
C.~Chen, A.~Atamt{\"u}rk, and S.~S Oren.
\newblock A spatial branch-and-cut method for nonconvex qcqp with bounded
  complex variables.
\newblock \emph{Mathematical Programming}, 165\penalty0 (2):\penalty0 549--577,
  2017.

\bibitem[Chen and Burer(2012)]{chen2012globally}
J.~Chen and S.~Burer.
\newblock Globally solving nonconvex quadratic programming problems via
  completely positive programming.
\newblock \emph{Mathematical Programming Computation}, 4\penalty0 (1):\penalty0
  33--52, 2012.

\bibitem[Eltved and Burer(2022)]{eltved2022strengthened}
A.~Eltved and S.~Burer.
\newblock Strengthened sdp relaxation for an extended trust region subproblem
  with an application to optimal power flow.
\newblock \emph{Mathematical Programming}, pages 1--26, 2022.

\bibitem[Hiriart-Urruty(2007)]{hiriart2007potpourri}
J.~Hiriart-Urruty.
\newblock Potpourri of conjectures and open questions in nonlinear analysis and
  optimization.
\newblock \emph{SIAM review}, 49\penalty0 (2):\penalty0 255--273, 2007.

\bibitem[Ho-Nguyen and K{\i}l{\i}n\c{c}-Karzan(2017)]{ho2017second}
N.~Ho-Nguyen and F.~K{\i}l{\i}n\c{c}-Karzan.
\newblock A second-order cone based approach for solving the trust-region
  subproblem and its variants.
\newblock \emph{{SIAM J.\ Optim.}}, 27\penalty0 (3):\penalty0 1485--1512, 2017.

\bibitem[Horn and Johnson(2012)]{horn2012matrix}
R.~A. Horn and C.~R. Johnson.
\newblock \emph{Matrix analysis}.
\newblock Cambridge University Press, 2012.

\bibitem[Hsia and Sheu(2013)]{hsia2013trust}
Y.~Hsia and R.~Sheu.
\newblock Trust region subproblem with a fixed number of additional linear
  inequality constraints has polynomial complexity.
\newblock \emph{arXiv preprint}, \penalty0 (arXiv:1312.1398), 2013.

\bibitem[Huang and Sidiropoulos(2016)]{huang2016consensus}
K.~Huang and N.~D. Sidiropoulos.
\newblock Consensus-{ADMM} for general quadratically constrained quadratic
  programming.
\newblock \emph{IEEE Transactions on Signal Processing}, 64\penalty0
  (20):\penalty0 5297--5310, 2016.

\bibitem[Jeyakumar and Li(2014)]{jeyakumar2014trust}
J.~Jeyakumar and G.~Li.
\newblock Trust-region problems with linear inequality constraints: exact {SDP}
  relaxation, global optimality and robust optimization.
\newblock \emph{{Math.\ Program.}}, 147:\penalty0 171--206, 2014.

\bibitem[Jiang and Li(2016)]{jiang2016simultaneous}
R.~Jiang and D.~Li.
\newblock Simultaneous diagonalization of matrices and its applications in
  quadratically constrained quadratic programming.
\newblock \emph{{SIAM J.\ Optim.}}, 26\penalty0 (3):\penalty0 1649--1668, 2016.

\bibitem[Kronecker(1968)]{kronecker1968collected}
L.~Kronecker.
\newblock \emph{Collected works}.
\newblock {American Mathematical Society}, 1968.

\bibitem[Lancaster and Rodman(2005)]{lancaster2005canonical}
P.~Lancaster and L.~Rodman.
\newblock Canonical forms for {H}ermitian matrix pairs under strict equivalence
  and congruence.
\newblock \emph{SIAM Review}, 47\penalty0 (3):\penalty0 407--443, 2005.

\bibitem[Le and Nguyen(2022)]{le2020equivalent}
T.~H. Le and T.~N. Nguyen.
\newblock Simultaneous diagonalization via congruence of hermitian matrices:
  some equivalent conditions and a numerical solution.
\newblock \emph{SIAM Journal on Matrix Analysis and Applications}, 43\penalty0
  (2):\penalty0 882--911, 2022.

\bibitem[Linderoth(2005)]{linderoth2005simplicial}
J.~Linderoth.
\newblock A simplicial branch-and-bound algorithm for solving quadratically
  constrained quadratic programs.
\newblock \emph{Mathematical programming}, 103\penalty0 (2):\penalty0 251--282,
  2005.

\bibitem[Locatelli(2016)]{locatelli2016exactness}
M.~Locatelli.
\newblock Exactness conditions for an {SDP} relaxation of the extended trust
  region problem.
\newblock \emph{{Oper.\ Res.\ Lett.}}, 10\penalty0 (6):\penalty0 1141--1151,
  2016.

\bibitem[Lu et~al.(2017)Lu, Deng, and Jin]{lu2017eigenvalue}
C.~Lu, Z.~Deng, and Q.~Jin.
\newblock An eigenvalue decomposition based branch-and-bound algorithm for
  nonconvex quadratic programming problems with convex quadratic constraints.
\newblock \emph{Journal of Global Optimization}, 67\penalty0 (3):\penalty0
  475--493, 2017.

\bibitem[Luo et~al.(2024)Luo, Chen, Zhang, Li, and Wu]{luo2020effective}
H.~Luo, Y.~Chen, X.~Zhang, D.~Li, and H.~Wu.
\newblock Effective algorithms for optimal portfolio deleveraging problem with
  cross impact.
\newblock \emph{Mathematical Finance}, 34\penalty0 (1):\penalty0 36--89, 2024.

\bibitem[Marks and Wright(1978)]{marks1978general}
Barry~R Marks and Gordon~P Wright.
\newblock A general inner approximation algorithm for nonconvex mathematical
  programs.
\newblock \emph{Operations research}, 26\penalty0 (4):\penalty0 681--683, 1978.

\bibitem[M{OSEK ApS}(2021)]{mosek}
M{OSEK ApS}.
\newblock \emph{The MOSEK optimization toolbox for MATLAB manual. Version
  9.10.}, 2021.
\newblock URL \url{http://docs.mosek.com/9.0/toolbox/index.html}.

\bibitem[Motzkin and Taussky(1955)]{motzkin1955pairs}
T.~S. Motzkin and O.~Taussky.
\newblock Pairs of matrices with property {L}. {II}.
\newblock \emph{{Trans.\ Amer.\ Math.\ Soc.}}, 80\penalty0 (2):\penalty0
  387--401, 1955.

\bibitem[Nguyen et~al.(2020)Nguyen, Nguyen, Le, and
  Sheu]{nguyen2020simultaneous}
T.~Nguyen, V.~Nguyen, T.~Le, and R.~Sheu.
\newblock On simultaneous diagonalization via congruence of real symmetric
  matrices.
\newblock \emph{arXiv preprint}, \penalty0 (arXiv:2004.06360), 2020.

\bibitem[O'meara and Vinsonhaler(2006)]{omeara2006approximately}
K.~O'meara and C.~Vinsonhaler.
\newblock On approximately simultaneously diagonalizable matrices.
\newblock \emph{{Linear Algebra Appl.}}, 412:\penalty0 39--74, 2006.

\bibitem[Polyak(1998)]{polyak1998convexity}
B.~T. Polyak.
\newblock Convexity of quadratic transformations and its use in control and
  optimization.
\newblock \emph{{J. Optim. Theory Appl.}}, 99\penalty0 (3):\penalty0 553--583,
  1998.

\bibitem[Sherali and Adams(2013)]{sherali2013reformulation}
H.~D Sherali and W.~P Adams.
\newblock \emph{A reformulation-linearization technique for solving discrete
  and continuous nonconvex problems}, volume~31.
\newblock Springer Science \& Business Media, 2013.

\bibitem[Shor(1990)]{shor1990dual}
N.~Z. Shor.
\newblock Dual quadratic estimates in polynomial and boolean programming.
\newblock \emph{Ann.\ Oper.\ Res.}, 25:\penalty0 163--168, 1990.

\bibitem[Suprunenko and Tyshkevich(1968)]{suprunenko1968commutative}
D.~A. Suprunenko and R.~I. Tyshkevich.
\newblock \emph{Commutative matrices}.
\newblock Academic Press, 1968.

\bibitem[Uhlig(1976)]{uhlig1976canonical}
F.~Uhlig.
\newblock A canonical form for a pair of real symmetric matrices that generate
  a nonsingular pencil.
\newblock \emph{{Linear Algebra Appl.}}, 14\penalty0 (3):\penalty0 189--209,
  1976.

\bibitem[Vollgraf and K.~Obermayer(2006)]{vollgraf2006quadratic}
R.~Vollgraf and Klaus K.~Obermayer.
\newblock Quadratic optimization for simultaneous matrix diagonalization.
\newblock \emph{IEEE Trans. Signal Process.}, 54\penalty0 (9):\penalty0
  3270--3278, 2006.

\bibitem[Wang and K{\i}l{\i}n{\c{c}}-Karzan(2024)]{wang2020geometric}
A.~L. Wang and F.~K{\i}l{\i}n{\c{c}}-Karzan.
\newblock On semidefinite descriptions for convex hulls of quadratic programs.
\newblock \emph{{Oper.\ Res.\ Lett.}}, page 107108, 2024.

\bibitem[Wang and K{\i}l{\i}n\c{c}-Karzan(2020)]{wang2020generalized}
A.~L. Wang and F.~K{\i}l{\i}n\c{c}-Karzan.
\newblock The generalized trust region subproblem: solution complexity and
  convex hull results.
\newblock \emph{{Math.\ Program.}}, 2020.
\newblock \doi{10.1007/s10107-020-01560-8}.
\newblock Forthcoming.

\bibitem[Wang and K{\i}l{\i}n\c{c}-Karzan(2021)]{wang2021tightness}
A.~L. Wang and F.~K{\i}l{\i}n\c{c}-Karzan.
\newblock On the tightness of {SDP} relaxations of {QCQP}s.
\newblock \emph{{Math.\ Program.}}, 2021.
\newblock \doi{10.1007/s10107-020-01589-9}.
\newblock Forthcoming.

\bibitem[Weierstrass(1868)]{weierstrass1868zur}
K.~Weierstrass.
\newblock Zur {T}heorie der quadratischen und bilinearen {F}ormen.
\newblock \emph{Monatsber. Akad. Wiss., Berlin}, pages 310--338, 1868.

\bibitem[Xu and Zhou(2023)]{xu2023simultaneous}
Z.~Xu and J.~Zhou.
\newblock A simultaneous diagonalization based socp relaxation for portfolio
  optimization with an orthogonality constraint.
\newblock \emph{Computational Optimization and Applications}, 85\penalty0
  (1):\penalty0 247--261, 2023.

\bibitem[Zhou and Xu(2019)]{zhou2019simultaneous}
J.~Zhou and Z.~Xu.
\newblock A simultaneous diagonalization based {SOCP} relaxation for convex
  quadratic programs with linear complementarity constraints.
\newblock \emph{{Optim.\ Lett.}}, 13\penalty0 (7):\penalty0 1615--1630, 2019.

\bibitem[Zhou et~al.(2020)Zhou, Chen, Yu, and Tian]{zhou2020simultaneous}
J.~Zhou, S.~Chen, S.~Yu, and Y.~Tian.
\newblock A simultaneous diagonalization-based quadratic convex reformulation
  for nonconvex quadratically constrained quadratic program.
\newblock \emph{Optimization}, pages 1--17, 2020.

\end{thebibliography}

\appendix

\section{Proof of Propositions \ref{prop:sdc_characterization_invertible} and \ref{prop:sdc_characterization}}
\label{sec:proof_sdc_characterization}

\sdccharacterizationinvertible*
\begin{proof}
$(\Rightarrow)$ Let $P\in\R^{n\by n}$ furnished by SDC. For $A\in\cA$, note that
\begin{align*}
P^{-1}S^{-1}AP &= (P^\intercal SP)^{-1} (P^\intercal AP).
\end{align*}
Then, as $P^\intercal SP$ and $P^\intercal AP$ are both diagonal matrices with real entries, we deduce that $S^{-1}A$ is diagonalizable with real eigenvalues. The fact that $S^{-1}\cA$ is a set of commuting matrices follows similarly.

$(\Leftarrow)$ Recall that a commuting set of diagonalizable matrices can be simultaneously diagonalized via a similarity transformation, i.e., there exists an invertible $P\in\R^{n\by n}$ such that $P^{-1}S^{-1}AP$ is diagonal for each $A\in\cA$ \cite{horn2012matrix}. The diagonal entries of $P^{-1}S^{-1}AP$ are furthermore real by the assumption that $S^{-1}A$ has a real spectrum. For each $A\in\cA$, define
\begin{align*}
\bar A \coloneqq P^\intercal AP,\qquad
D_A \coloneqq P^{-1}S^{-1}AP.
\end{align*}
Next, note that the identity $P^{-1}S^{-1}AP = (P^\intercal SP)^{-1}(P^\intercal AP)$ can be expressed as $D_A = \bar S^{-1}\bar A$. Or, equivalently, $\bar S D_A = \bar A$ for all $A\in\cA$. For $i,j\in[n]$, we have the identity
\begin{align*}
\bar S_{i,j} (D_A)_{j,j} = \bar A_{i,j} = \bar A_{j,i} = \bar S_{j,i} (D_A)_{i,i} = \bar S_{i,j} (D_A)_{i,i}.
\end{align*}
Here, we have used that $\bar S$ and $\bar A$ are symmetric and $D_A$ is real diagonal.
In particular, if there exists some $A\in\cA$ such that $(D_A)_{i,i}\neq (D_A)_{j,j}$, then $\bar S_{i,j} = \bar A_{i,j} = 0$. Furthermore, by the relation $\bar SD_B = \bar B$, we also have that $\bar B_{i,j} = 0$ for all other $B\in\cA$.

We conclude that by permuting the columns of $P$ if necessary (so that $[n]$ is grouped according to the equivalence relation: $i\sim j$ if and only if $(D_A)_{i,i} = (D_A)_{j,j}$ for all $A\in\cA$), we can write $\bar S$ as a block diagonal matrix $\bar S=\Diag(S^{(1)},\dots,S^{(k)})$. Furthermore, for every $A\in\cA$, there exists $\lambda_1,\dots,\lambda_k\in\R$ such that $\bar A = \Diag(\lambda_1 S^{(1)},\dots, \lambda_k S^{(k)})$. It remains to note that each block $S^{(i)}$ can be diagonalized separately.
\end{proof}

\sdccharacterization*
\begin{proof}
It suffices to show that if $\cA$ is SDC then $\range(A)\subseteq\range(S)$ for every $A\in\cA$ as then applying \cref{lem:sdc_closed_under_padding} completes the proof.

Let $r =\rank(S)$.
Let $P\in\R^{n\by n}$ furnished by SDC. Note that by permuting the columns of $P$ if necessary, we may assume that $P^\intercal SP$ is a diagonal matrix with support contained in its first $r$-many diagonal entries. As $S$ is a max-rank element of $\spann(\cA)$, we similarly have that for every $A\in\cA$, the matrix $P^\intercal AP$ is a diagonal matrix with support contained in its first $r$-many diagonal entries. For $A\in\cA$, write $P^\intercal AP =\Diag(\bar A, 0_{(n-r)\times(n-r)})$ where $\bar A$ is a diagonal $r\by r$ matrix. Then,
\begin{align*}
\range(A) &= \range(P^{-\intercal}P^\intercal APP^{-1}) \subseteq \spann\set{q_1,\dots,q_r}.
\end{align*}
Here, $q_i\in\R^n$ is the $i$th column of $P^{-\intercal}$. On the other hand, as $\bar S$ has full rank, $\range(S) = \spann\set{q_1,\dots,q_r}$.
\end{proof}

\section{Facts about matrices with upper triangular Toeplitz blocks}
\label{sec:upper_tri_toeplitz}

\begin{lemma}
\label{lem:toeplitz_main_diagonals}
Let $(n_1,\dots, n_k)$ with $\sum_i n_i = n$.
Suppose $T\in\T$. Then, the characteristic polynomial of $T$ depends only on the entries $\set{t_{i,j}^{(1)}:\, n_i = n_j}$.
\end{lemma}

\begin{proof}
In this proof, we will use $a,b\in[n]$ to index entries in $T$ (specifically, $T_{a,b}\in\R$ is a scalar, not a matrix block).
For each $a\in[n]$, let $i_a\in[k]$ denote the block containing $a$, and let $\ell_a\in[n_k]$ denote the position of $a$ within block $i_a$. By the assumption that $T\in\T$, we have
\begin{align*}
T_{a,b}\neq 0 \implies \min\set{n_{i_a},\,n_{i_b}} - n_{i_a} + (\ell_a - \ell_b) \geq 0.
\end{align*}

Now, for each $a\in[n]$, assign the weight $w_a \coloneqq \ell_a - \frac{n_{i_a}}{2}$. Note that by construction, if $T_{a,b}\neq 0$, then
\begin{align*}
w_a - w_b &= \frac{n_{i_b}}{2} - \frac{n_{i_a}}{2}+ (\ell_a -\ell_b)\geq 0.
\end{align*}
Furthermore, note that if $T_{a,b}\neq 0$ and $w_a-w_b=0$, then $n_{i_a} = n_{i_b}$ and $\ell_a = \ell_b$.

Next, consider a permutation $\sigma\in S_n$ such that $\prod_{a=1}^n T_{a,\sigma(a)}\neq 0$. Note that
\begin{align*}
\sum_{a=1}^n w_a - w_{\sigma(a)}  &= \sum_{a=1}^n w_a - \sum_{a=1}^n w_{\sigma(a)} = 0.
\end{align*}
Then, by the above paragraph, we conclude that $\sigma$ satisfies $n_{i_a} = n_{i_{\sigma(a)}}$ and $\ell_a = \ell_\sigma(a)$ for all $a\in[n]$.

Returning to the previous notation, the characteristic polynomial of $T$ depends only on the entries $\set{t_{i,j}^{(1)}:\, n_i = n_j}$.
\end{proof}

\toeplitzmapping*
\begin{proof}
Without loss of generality, suppose $n_1\leq \dots\leq n_k$ and let $T\in\T$. By \cref{lem:toeplitz_main_diagonals}, $T$ has the same eigenvalues as the matrix $\hat T\in\T$ with entries
\begin{align*}
\hat T_{i,j}^{(\ell)} &= \begin{cases}
	T_{i,j}^{(\ell)} & \text{if } n_i = n_j,\, \ell = 1,\\
	0 & \text{else}.
\end{cases}
\end{align*}
Now, suppose that there are $m$ distinct block sizes $s_1,\dots, s_m$. Partitioning both $\Pi(T)$ and $\hat T$ according to $s_1,\dots, s_m$, we have that
\begin{align*}
\Pi(T) = \Diag(\tilde T_1, \dots, \tilde T_m)
\quad\text{and}\quad
\bar T = \Diag(\tilde T_1 \otimes I_{s_1},\dots, \tilde T_m \otimes I_{s_m}).
\end{align*}
We conclude that $\Pi(T)$ and $\bar T$ have the same eigenvalues.
\end{proof}
\section{Details for the Hermitian case}
\label{sec:hermitian_proofs}

Let $\H^n$ denote the real vector space of $n\by n$ Hermitian matrices. For $v\in\C^n$ and $A\in\C^{n\by n}$ let $v^*$ and $A^*$ denote the conjugate transpose of $v$ and $A$ respectively.

\subsection{Definitions and theorem statements}
Almost all of our results extend verbatim to the Hermitian setting. For brevity, we only state our more interesting definitions and results as adapted to this setting.

\begin{definition}
A set $\cA\subseteq\H^n$ is \emph{simultaneously diagonalizable via congruence} (SDC) if there exists an invertible $P\in\C^{n\by n}$ such that $P^*AP$ is diagonal for all $A\in\cA$.\mathprog{\qed}
\end{definition}

\begin{definition}
A set $\cA\subseteq\H^n$ is \emph{almost simultaneously diagonalizable via congruence} (ASDC) if there exist sequences $A_i\to A$ for every $A\in\cA$ such that for every $i\in\N$, the set $\set{A_i:\, A\in\cA}$ is SDC.\mathprog{\qed}
\end{definition}

\begin{definition}
A set $\cA\subseteq\H^n$ is \emph{nonsingular} if there exists a nonsingular $A\in\spann(\cA)$. Else, it is \emph{singular}.\mathprog{\qed}
\end{definition}

\begin{definition}
Given a set $\cA\subseteq\H^n$, we will say that $S\in\cA$ is a \textit{max-rank element} of $\spann(\cA)$ if $\rank(S) = \max_{A\in\cA}\rank(A)$.\mathprog{\qed}
\end{definition}

\begin{theorem}
\label{thm:hermitian_asdc_nonsingular_pair}
Let $A,B\in\H^n$ and suppose $A$ is invertible. Then, $\set{A,B}$ is ASDC if and only if $A^{-1}B$ has real eigenvalues.
\end{theorem}

\begin{theorem}
Let $\set{A,B}\subseteq\H^n$. If $\set{A,B}$ is singular, then it is ASDC.
\end{theorem}

\begin{theorem}
\label{thm:hermitian_asdc_nonsingular_triple}
Let $\set{A,B,C}\subseteq\H^n$ and suppose $A$ is invertible. Then, $\set{A,B,C}$ is ASDC if and only if $\set{A^{-1}B,A^{-1}C}$ are a pair of commuting matrices with real eigenvalues.
\end{theorem}

\begin{definition}
Let $\cA\subseteq\H^n$ and $d\in\N$. We will say that $\cA$ is \emph{$d$-restricted SDC} ($d$-RSDC) if there exist matrices $\bar A\in\H^{n+d}$ containing $A$ as its top-left $n\by n$ principal submatrix for every $A\in\cA$ such that $\set{\bar A:\, A\in\cA}$ is SDC.\mathprog{\qed}
\end{definition}

\begin{theorem}
Let $A,B\in\H^n$.
Then for every $\epsilon>0$, there exist $\tilde A,\tilde B\in\H^n$ such that $\norm{A - \tilde A},\norm{B - \tilde B}\leq \epsilon$ and $\set{\tilde A,\tilde B}$ is $1$-RSDC. Furthermore, if $A$ is invertible and $A^{-1}B$ has simple eigenvalues, then $\set{A,B}$ is itself $1$-RSDC.
\end{theorem}


\begin{theorem}
Let $\set{A=I_n,B,C}\subseteq\H^n$. Then, if $d<\rank([B,C])/2$, the set
\begin{align*}
\set{\begin{pmatrix}
        A&\\
        &0_{d}
\end{pmatrix}, \begin{pmatrix}
        B &\\
        & 0_d
\end{pmatrix}, \begin{pmatrix}
        C &\\
        & 0_d
\end{pmatrix}}
\end{align*}
is not ASDC.
\end{theorem}

\begin{theorem}
There exists a set $\cA = \set{A_1,\dots,A_5}\subseteq\H^4$ such that $A_1$ is invertible, $A_1^{-1}\cA$ is a set of commuting matrices with real eigenvalues, and $\cA$ is not ASDC.
\end{theorem}
In the Hermitian setting, the statement in \cref{thm:obstruction_m=7_invertible} should be changed to:
``There exists a set $\cA = \set{A_1,\dots,A_5}\subseteq\H^4$ such that $A_1$ is invertible, $A_1^{-1}\cA$ is a set of commuting matrices with real eigenvalues and $\cA$ is not ASDC.'' The proof is unchanged after setting
\begin{gather*}
A_1 = \begin{smallpmatrix}
        &&&1\\
        &&1&\\
        &1&&\\
        1&&&
\end{smallpmatrix},\quad
A_2 = \begin{smallpmatrix}
        0&&&\\
        &0&&\\
        &&1&\\
        &&&0
\end{smallpmatrix},\quad
A_3 = \begin{smallpmatrix}
        0&&&\\
        &0&&\\
        &&0&1\\
        &&1&0
\end{smallpmatrix},\\
A_4 = \begin{smallpmatrix}
        0&&&\\
        &0&&\\
        &&0&\imag\\
        &&-\imag&0
\end{smallpmatrix},\quad
A_5 = \begin{smallpmatrix}
        0&&&\\
        &0&&\\
        &&0&\\
        &&&1
\end{smallpmatrix}.
\end{gather*}
\subsection{Necessary modifications}
\label{sub:hermitian_proofs_modifications}
Next, we discuss technical changes that need to be made to adapt our proofs from the real symmetric setting to the Hermitian setting. For brevity, we only list changes beyond the trivial changes, e.g., replacing $\S^n$ by $\H^n$, $\R^{n\by n}$ by $\C^{n\by n}$, and ${}^\intercal$ by ${}^*$.

\begin{itemize}
\item In the Hermitian version of \cref{prop:canonical_form}, the $m_2$-many blocks corresponding to non-real eigenvalues (previously \eqref{eq:canonical_form_imaginary}) will have the form
\begin{align*}
S_i &= F_{2n_i},\qquad
T_i = F_{n_i} \otimes \begin{pmatrix}
         & \lambda_i^*\\
        \lambda_i &
\end{pmatrix} + G_{n_i} \otimes F_2
\end{align*}
where $n_i \in\N$ and $\lambda_i \in\C\setminus \R$.
See \cite[Theorem 9.2]{lancaster2005canonical} for further details.

\item In the proof of \cref{lem:perturbed_canonical_form}, note that for all $i\in[r+1,m]$, the block
\begin{align*}
S_i^{-1}\tilde T_i = I_{n_i}\otimes \begin{pmatrix}
        \lambda_i&\\&\lambda_i^*
\end{pmatrix} + (\eta_i I_{n_i} + F_{n_i}G_{n_i} + \delta F_{n_i}H_{n_i})\otimes I_2.
\end{align*}
The remainder of the proof is unchanged.

\item In the proof of \cref{thm:asdc_singular}, we will work in the basis furnished by the Hermitian version of \cref{prop:canonical_form} for $\C^{2k}$. That is, we may assume in the first two cases that $A$ and $B$ (previously \eqref{eq:asdc_singular_barA_barB}) have the form
\begin{align*}
 A &= \left(\begin{array}
        {@{}c|c|c|c@{}}
        \begin{smallmatrix}
                &1\\1&
        \end{smallmatrix} && \\\hline
        & \ddots & \\\hline
        && \begin{smallmatrix}
                &1\\1&
        \end{smallmatrix}\\\hline
        &&&S_m
\end{array}\right),\qquad
B = \left(\begin{array}
        {@{}c|c|c|c@{}}
        \begin{smallmatrix}
                &\lambda_1^*\\\lambda_1
        \end{smallmatrix} && \\\hline
        & \ddots & \\\hline
        && \begin{smallmatrix}
                &\lambda_k^*\\\lambda_k
        \end{smallmatrix}\\\hline
        &&&T_m
\end{array}\right).
\end{align*}

We will set $\tilde A_\delta$ as in the Hermitian case for both Cases 1 and 2. We will set $\tilde B_\delta$ to be
\begin{align*}
\tilde B_\delta = \left(\begin{array}
        {@{}c|c|c|c@{}}
        \begin{smallmatrix}
                &\lambda_1^*\\\lambda_1
        \end{smallmatrix} &&&\begin{smallmatrix}
                \alpha_1\sqrt{-\delta\imag/2}\\
                \left(\alpha_1\sqrt{-\delta\imag/2}\right)^*
        \end{smallmatrix}\\\hline
        & \ddots & &\vdots\\\hline
        && \begin{smallmatrix}
                &\lambda_k^*\\\lambda_k
        \end{smallmatrix} & \begin{smallmatrix}
                \alpha_k\sqrt{-\delta\imag/2}\\
                \left(\alpha_k\sqrt{-\delta\imag/2}\right)^*
        \end{smallmatrix}\\\hline
        \begin{smallmatrix}
                \left(\alpha_1\sqrt{\tfrac{-\delta\imag}{2}}\right)^*&
                \alpha_1\sqrt{\tfrac{-\delta\imag}{2}}
        \end{smallmatrix}&\dots&\begin{smallmatrix}
                \left(\alpha_k\sqrt{\tfrac{-\delta\imag}{2}}\right)^*&
                \alpha_k\sqrt{\tfrac{-\delta\imag}{2}}
        \end{smallmatrix}&\delta z
\end{array}\right)
\end{align*}
and
\begin{align*}
\tilde B_\delta = \left(\begin{array}
        {@{}c|c|c|c|c|c@{}}
        \begin{smallmatrix}
                &\lambda_1^*\\\lambda_1
        \end{smallmatrix} &&&& \begin{smallmatrix}
                \alpha_1\sqrt{-\delta\imag/2}\\
                \left(\alpha_1\sqrt{-\delta\imag/2}\right)^*
        \end{smallmatrix}&\\\hline
        & \ddots && &\vdots & \\\hline
        && \begin{smallmatrix}
                &\lambda_k^*\\\lambda_k
        \end{smallmatrix}&&\begin{smallmatrix}
                \alpha_k\sqrt{-\delta\imag/2}\\
                \left(\alpha_k\sqrt{-\delta\imag/2}\right)^*
        \end{smallmatrix}&\\\hline
        &&&&&G_{n_m}\\\hline
        \begin{smallmatrix}
                \left(\alpha_1\sqrt{\tfrac{-\delta\imag}{2}}\right)^*&
                \alpha_1\sqrt{\tfrac{-\delta\imag}{2}}
        \end{smallmatrix}&\cdots&\begin{smallmatrix}
                \left(\alpha_k\sqrt{\tfrac{-\delta\imag}{2}}\right)^*&
                \alpha_k\sqrt{\tfrac{-\delta\imag}{2}}
        \end{smallmatrix}&&\delta z& e_1^\intercal\\\hline
        &&&G_{n_m}&e_1
\end{array}\right)
\end{align*}
for Cases 1 and 2, respectively. Here, $\alpha\in\C^k$, $z\in\R$, and $\delta>0$.
The characteristic polynomials of $\tilde A_\delta^{-1}\tilde B_\delta$ are given by \eqref{eq:case_1_char_poly} and \eqref{eq:case_2_char_poly} in Cases 1 and 2 respectively. The remainder of the proof remains unchanged.
\end{itemize}


\section{An example where the SDC property is preserved under restriction}
\label{sec:hom_to_inhom_SDC}

In this section, we give an example of a setting in which the restriction of an SDC set to one of its principal submatrices results in another SDC set. This setting arises for example in QCQPs~\cite{jiang2016simultaneous}.

\begin{proposition}
Let $A_1,\dots,A_m\in\S^n$ such that $\spann(\set{A_1,\dots,A_m})$ contains a positive definite matrix. Let $b_1,\dots,b_m\in\R^n$ and $c_1,\dots,c_m\in\R$, and define
\begin{align*}
Q_i = \begin{pmatrix}
        A_i & b_i\\
        b_i^\intercal & c_i
\end{pmatrix}\in\S^{n+1}.
\end{align*}
If $\set{Q_1,\dots,Q_m,e_{n+1}e_{n+1}^\intercal}$ is SDC, then so is $\set{A_1,\dots,A_m}$.
\end{proposition}
\begin{proof}
Without loss of generality, let $A_1\succ 0$.
Note that for all $\lambda\in\R$ large enough, the matrix $S_\lambda \coloneqq Q_1 + \lambda e_{n+1}e_{n+1}^\intercal\succ 0$. By the inverse formula for a block matrix \cite{horn2012matrix}, we have that for all $\lambda$ large enough,
\begin{align*}
S_\lambda^{-1} &= \begin{pmatrix}
        A_1^{-1} + \frac{A_1^{-1}b_1b_1^\intercal A_1^{-1}}{\lambda + (c_1 - b_1^\intercal A_1b_1)} & \frac{-A_1^{-1}b_1}{\lambda + (c_1 - b_1A_1^{-1}b_1)}\\
        \frac{-b_1^\intercal A_1^{-1}}{\lambda + (c_1 - b_1A_1^{-1}b_1)} & \frac{1}{\lambda + (c_1 - b_1A_1^{-1}b_1)}
\end{pmatrix}.
\end{align*}
In particular,
\begin{align*}
\lim_{\lambda\to\infty}S_\lambda^{-1} = \begin{pmatrix}
        A_1^{-1} & \\
        &0
\end{pmatrix}.
\end{align*}
On the other hand, by \cref{lem:sdc_characterization_pd}, we have that for all $i,j\in[m]$,
\begin{align*}
0 &= \left[S_\lambda^{-1}Q_i,\, S_\lambda^{-1} Q_j\right].
\end{align*}
Finally, by continuity we have that
\begin{align*}
0 = \lim_{\lambda\to \infty} \left[S_\lambda^{-1}Q_i,\, S_\lambda^{-1} Q_j\right] = \begin{pmatrix}
        \left[A_1^{-1}A_i,\,A_1^{-1}A_j\right] &\\&0
\end{pmatrix}.
\end{align*}
We conclude that $A_1^{-1}\set{A_1,\dots,A_m}$ commute, whence by \cref{lem:sdc_characterization_pd} this set is SDC.
\end{proof}

\end{document}